\documentclass[11pt,reqno]{amsart}
\usepackage[utf8]{inputenc}
\usepackage{latexsym,amssymb,longtable}
\usepackage{hyperref,enumerate}
\usepackage{tikz}
\usetikzlibrary{decorations.pathmorphing,calc,arrows.meta}
\usepackage[eulergreek]{sansmath}

\usepackage{amsmath,amsthm,amsxtra}
\usepackage{amsfonts,eucal,mathrsfs,todonotes}
\usepackage{enumitem}
\usepackage{fullpage}
\usepackage{xcolor}
\usepackage{upgreek}
\usepackage[normalem]{ulem}

\numberwithin{equation}{section}

\definecolor{ddorange}{rgb}{1,0.5,0}
\definecolor{ddcyan}{rgb}{0,0.2,1.0}
\definecolor{dcyan}{rgb}{0,0.3,0.8}
\definecolor{ddmagenta}{rgb}{0.8,0,0.8}
\definecolor{turk}{rgb}{0,0.5,0.5}

\newcommand{\N}{\mathbb{N}}

\newcommand{\CE}[2]{\calC \calE([#1,#2])}

\newcommand{\ACE}[2]{\mathcal{A}\mathcal{C}\mathcal{E}([#1,#2])}
\newcommand{\ACEn}[2]{\mathcal{A}\mathcal{C}\mathcal{E}^n([#1,#2])}
\newcommand{\ECE}[2]{\mathcal{R}\mathcal{C}\mathcal{E}([#1,#2])}
\newcommand{\ECEn}[2]{\mathcal{R}\mathcal{C}\mathcal{E}^n([#1,#2])}

\newcommand{\R}{\mathbb{R}}

\newcommand{\BL}{\mathrm{BL}}

\newcommand{\calA}{\mathcal{A}}

\newcommand{\calC}{\mathcal{C}}

\newcommand{\calE}{\mathcal{E}}
\newcommand{\calF}{\mathcal{F}}

\newcommand{\calM}{\mathcal{M}}
\newcommand{\calP}{\mathcal{P}}
\newcommand{\calR}{\mathcal{R}}

\newcommand{\scrL}{\mathscr{L}}
\newcommand{\Lebone}{\lambda}
\newcommand{\Lip}{\mathrm{Lip}}
\newcommand{\Lipb}{\mathrm{Lip}_{\mathrm{b}}}

\newcommand{\Fmap}{\mathrm{F}}

\newcommand{\eps}{\varepsilon}

\newcommand{\dnabla}{\overline\nabla}
\newcommand{\supp}{\text{supp}\,}

\newcommand{\dd}{\mathrm{d}}

\newtheorem{theorem}{Theorem}[section]
\newtheorem{prop}[theorem]{Proposition}

\newtheorem{cor}[theorem]{Corollary}
\newtheorem{lemma}[theorem]{Lemma}
\newtheorem{remark}[theorem]{Remark}
\newtheorem{mainthm}{Main Theorem}

\theoremstyle{definition}
\newtheorem{definition}[theorem]{Definition}
\newtheorem{assumption}[theorem]{Assumption}
\newtheorem{example}[theorem]{Example}

\let\eps\ep

\newcommand{\foraa}{\text{for a.e.\ }}

\def\dd{\mathrm{d}}
\DeclareMathSymbol{\mtimes}{\mathord}{symbols}{"0A}

\newcommand{\dom}{\mathop{\rm dom}}
\newcommand{\edg}{E}

\newcommand{\ona}{\dnabla}

\newcommand{\odivn}{\overline{\mathrm{div}}}

\newcommand{\AC}{\mathrm{AC}}

\newcommand{\Cc}{\mathrm{C}_{\mathrm{c}}}
\newcommand{\Cb}{\mathrm{C}_{\mathrm{b}}}
\newcommand{\Bb}{\mathrm{B}_{\mathrm b}}

\newcommand{\Fish}{\mathscr{D}}

\newcommand{\weakto}{\rightharpoonup}

\newcommand{\piecewiseConstant}[2]{\overline{#1}_{\kern-1pt#2}}

\newcommand{\underpiecewiseConstant}[2]{\underline{#1}_{\kern-1pt#2}}

\newcommand{\pwM}[2]{\widetilde{#1}_{\kern-1pt#2}}
\def\gen(#1,#2){\calS_{#1}(#2)}

\newcommand{\teta}{\boldsymbol \vartheta}
\newcommand{\serifsigma}{{\sansmath \sigma}}

\newcommand{\tetapi}{\boldsymbol{\teta}_{\!\kappa}}
\newcommand{\tetapil}{\boldsymbol{\teta}_{\!\kappa}^d}

\newcommand{\tetapin}{\boldsymbol{\teta}_{\!\kappa_n}}

\newcommand{\pinfty}{{+\infty}}
\newcommand{\mres}{\kern1pt\mathbin{\vrule height 1.6ex depth 0pt width 0.13ex\vrule height 0.13ex depth 0pt width 1.3ex}}

\newcommand{\thalf}{\relax} 

\newcommand{\frB}{\mathfrak B}

\def\calS{\mathscr E}
\def\calF{\mathscr F}
\def\calR{\mathscr R}
\def\Aalpha{\upalpha}
\newcommand{\bnu}{\boldsymbol\upnu}
\newcommand{\bj}{{\boldsymbol j}}

\newcommand{\rmD}{\mathrm{D}}

\newcommand{\restr}[1]{\lower3pt\hbox{$|_{#1}$}}
\newcommand{\nchi}{{\raise.3ex\hbox{$\chi$}}}

\newcommand{\scrR}{\mathscr{R}}
\newcommand{\scrD}{\mathscr{D}}
\newcommand{\scrE}{\mathscr{E}}
\newcommand{\jj}{{\boldsymbol{j}}}
\newcommand{\rrho}{{\boldsymbol{\rho}}}

\newcommand{\Ed}{{E'}}
\newcommand{\dV}{\mathsf{d}}
\newcommand{\Mloc}{\mathcal{M}_{\mathrm{loc}}}
\newcommand{\pairing}[4]{ \sideset{_{#1 }}{_{ #2}}  {\mathop{\langle #3 , #4  \rangle}}}

\newcommand{\psih}{\mathfrak{f}}
\newcommand{\nbl}[1]{\|#1\|_{\mathrm{BL}} }

\newcommand{\bbeta}{\boldsymbol \beta}

\newcommand{\rmL}{\mathrm{L}}
\newcommand{\rmX}{\mathrm{X}}
\newcommand{\rmC}{\mathrm{C}}

\newcommand{\down}{\downarrow}

\newcommand{\RNEW}{\color{black}} 

 \newcommand{\EEE}{\color{black}}

 \newcommand{\lds}{\nu}

  \newcommand{\DRd}{\mathcal{E}}

\newcommand{\Young}{\mathcal{Y}}
\newcommand{\calX}{\mathcal{X}}
\newcommand{\VV}{\boldsymbol{V}}
\newcommand{\EE}{\boldsymbol{E}}

\newcommand{\TTheta}{\boldsymbol{\Theta}}
\newcommand{\ppi}{\boldsymbol{\pi}}
\newcommand{\kkappa}{\boldsymbol{\kappa}}
\newcommand{\mmu}{\boldsymbol{\mu}}
\newcommand{\xx}{\boldsymbol{x}}
\newcommand{\yy}{\boldsymbol{y}}
\newcommand{\zz}{\boldsymbol{z}}

\newcommand{\torus}{\mathbb{T}^d}
\newcommand{\trait}{\mathsf{S}}

\newcommand{\Dissipative}{\textsc{Dissipative}}
\newcommand{\Balanced}{\textsc{Balanced}}
\newcommand{\Reflecting}{\textsc{Reflecting}}

\newcommand{\B}{\mathrm{B}}
\newcommand{\Ms}{\mathcal{M}_{\sigma}}
\newcommand{\Rpe}{\overline{\R}_{\geq0}}

\numberwithin{equation}{section}

\title[Singular jump processes]{Singular jump processes as generalized gradient flows}

\begin{document}

\author{Jasper Hoeksema}
\address{J. Hoeksema, Department of Mathematics and Computer Science, Eindhoven University of Technology, 5600 MB Eindhoven, The Netherlands}
\email{j.hoeksema@tue.nl}

\author{Riccarda Rossi}
\address{R.\ Rossi, DIMI, Universit\`a degli Studi di Brescia. Via Branze 38, I--25133 Brescia -- Italy}
\email{riccarda.rossi@unibs.it}

\author{Oliver Tse}
\address{O. Tse, Department of Mathematics and Computer Science, Eindhoven University of Technology, 5600 MB  Eindhoven, The Netherlands}
\email{o.t.c.tse@tue.nl}

\begin{abstract}
 We extend the generalized gradient-flow framework of Peletier, Rossi, Savar\'e, and Tse to singular jump processes on abstract metric spaces, moving beyond the translation-invariant kernels considered in $\R^d$ and $\torus$ in previous contributions. To address the analytical challenges posed by singularities, we introduce \emph{reflecting} solutions, a new solution concept inspired by reflected Dirichlet forms, which ensures the validity of a chain rule and restores uniqueness. We establish existence, stability, and compactness results for these solutions by approximating singular kernels with regularized ones, and we show their robustness under such approximations. The framework encompasses \emph{dissipative} and \emph{balanced} solutions, clarifies their relations, and highlights the role of density properties of Lipschitz functions in upgrading weak formulations to \emph{reflecting} solutions. As an application, we demonstrate the versatility of our theory to nonlocal stochastic evolutions on configuration spaces.
\end{abstract}

\maketitle

\centerline{\today}

\tableofcontents

\section{Introduction}

In recent years, stochastic processes have increasingly been understood through the lens of generalized gradient flows in spaces of measures. Starting from the seminal work of Jordan, Kinderlehrer, and Otto \cite{JordanKinderlehrerOtto98}, this variational point of view has opened new perspectives for understanding the geometry underlying a variety of stochastic processes, ranging from diffusions and discrete Markov chains \cite{AGS08,Adams-Dirr-Peletier-Zimmer,Maas11,Mielke13CALCVAR,MielkePeletierRenger14,LieMiePeleRenger17} to more recent developments in jump processes \cite{Erbar14,PRST22,warren2025}, leading to new connections between probability, analysis, and thermodynamics.

The present work serves to extend the generalized gradient flow framework proposed in \cite{PRST22} to encompass \emph{singular} jump processes on abstract spaces beyond $\R^d$ or $\torus$. Such singular jump processes naturally arise in statistical mechanics, disordered systems, and certain scaling limits of interacting particle models \cite{bertini2015macroscopic,kipnis2013scaling,finkelshtein2010vlasov}. Yet, their variational interpretation remains largely unexplored (cf.\ Section~\ref{s:8} for an example in the realm of interacting particle systems). 

A fundamental obstacle to this generalization is that in $\R^d$ or $\torus$ the strongest variational notions of solutions, namely those that lead to both existence and uniqueness, seem to require (see \cite{Erbar14,warren2025}) that the jump kernel behaves suitably well under translation, to allow for the usage of convolution techniques. Such machinery indeed plays a pivotal role in guaranteeing the validity of a suitable chain rule (cf.\ \eqref{CR-intro} ahead), which is at the core of the variational approach. We will provide elementary examples that show that removing the restriction on the kernel related to translations can, in fact, result in non-uniqueness, but that this can be salvaged by making an explicit choice of the function spaces involved, which is, in turn, connected to choosing a set of boundary conditions. This will be clarified by Examples \ref{ex:Jasp-intro-1} and \ref{ex:Jasp-intro-2} ahead.

We will focus on a specific choice of function spaces and introduce $\Reflecting$ solutions for gradient-flow systems, involving jump kernels $\kappa$ over separable metric spaces $(V,\dV)$ for which the metric and kernel are compatible, in the sense that
\begin{equation}\label{eq:mainm}
\sup_{x\in V}\int_{V} (1{\wedge} \dV^2(x,y))\kappa(x,\dd y)<+\infty.
\end{equation}
Here, our naming convention for this notion of solution stems from the relationship between these solutions to \emph{reflected} Dirichlet forms \cite{Schilling2012,Schmidt2018}, and generalizations of Neumann boundary conditions for fractional Laplacians \cite{Guan2006}.

Together with the corresponding invariant measure $\pi$, tuples $(V,\dV,\kappa,\pi)$ satisfying \eqref{eq:mainm} are the straightforward extensions of L\'evy-type kernels \cite{Applebaum_2009}, for which $V=\R^d$ and $\dV(x,y)=|x-y|$, and arise both in the study of adapted or intrinsic metrics and stochastic completeness for infinite graphs \cite{Frank2014,Grigoryan2011}, or in the rich literature on nonlocal or fractional metric measure spaces \cite{ToniHeikkinen2013,DiMarino2019,Grka2022}. It should be noted that we neither require continuity of the kernel nor local compactness of the metric space $(V,\dV)$.

In spite of the analytical challenges posed by these singularities and the lack of local compactness of $V$, we show that the essential gradient-flow structure, which was unveiled in \cite{PRST22} in the case of bounded kernels, now survives in a suitably generalized sense. In particular, we identify the correct
functional-analytic setting in which energy-dissipation balances can be recovered, and we demonstrate how the variational formulation can be adapted to account for singular transition kernels. Moreover, we will establish that any sequence of \emph{cut-off} problems, namely those where the singular kernel is approximated from below by bounded kernels, converges to our framework of $\Reflecting$ solutions. 

\medskip

\paragraph{\bf Variational formulation for bounded kernels}
Let us briefly recall the variational approach set forth in \cite{PRST22}  for jump processes with bounded kernels, i.e., satisfying
\begin{equation}
\sup_{x\in V} \int_V \kappa(x,\dd y) < +\infty.
\end{equation}
The law of the corresponding jump process consists of time-dependent measures $t\mapsto \rho_t\in \calM^+(V)$ satisfying the following Forward Kolmogorov Equation
\begin{equation}\label{FKEm}\tag{\textsf{FKE}}
\partial_t \rho_t = Q^* \rho_t, \qquad (Q^* \rho)(\dd x)=\int_{y\in V}\kappa(y,\dd x)\rho(\dd y)-\rho(\dd x)\int_{y\in V} \kappa(x,\dd y).
\end{equation}
Here $Q^*$ is the dual of the infinitesimal generator $Q$, given by 
\begin{equation}
(Q \varphi)(x)=\int_V (\dnabla \varphi)(x,y) \kappa(x,\dd y), \qquad (\dnabla \varphi)(x,y):=\varphi(y)-\varphi(x).
\end{equation}
This system can be reformulated in terms of gradient flows, see \cite{Maas11,ChowHuangLiZhou12,Mielke13CALCVAR,MielkePeletierRenger14,LieroMielkePeletierRenger17} for finite graphs, and \cite{PRST22} for the general setting of bounded kernels over Polish spaces and for the extension to generalized gradient flows. The formal (generalized) gradient-flow structure for \eqref{FKEm} reads 
\begin{equation}\label{FGF-intro}
\begin{aligned}
\partial_t \rho + \odivn\,\jj &=0,
\\
\jj &= \partial_2 \calR^* (\rho,{-}\dnabla\mathrm{D} \scrE(\rho)).
\end{aligned}
\end{equation}

Here, the first equation, henceforth called the \emph{continuity equation}, can be interpreted in a suitably weak sense: For every \emph{bounded} function $\varphi\in \Bb(V)$, the \emph{density-flux} pair $(\rho,\jj)$ satisfies
\begin{equation}\label{eq:contm}
\int_V \varphi(x)\,\rho_t(\dd x) - \int_V \varphi(x)\,\rho_s(\dd x) = \int_s^t \!\!\iint_{V \times V}\dnabla \varphi(x,y)\,\bj_r(\dd x \dd y)\,\dd r\qquad\text{for all $[s,t]\subset[0,T]$}.
\end{equation}
The building blocks of this structure are the driving  entropy functional 
\[
   \calM^+(V)\ni \rho\mapsto \scrE(\rho):=\int_V \upphi\left(\frac{\dd \rho}{\dd \pi}\right)\dd \pi,
\]
with entropy density $\upphi$, and the dissipation potential $\calR$, with dual $\calR^*$, given by
\begin{equation}
\label{formal_definition}
\begin{aligned}
&
\scrR(\rho,\jj): = \frac12 \iint_{V\times V} \uppsi \left( 2 \frac{\dd \jj}{\dd \bnu_\rho} \right) \bnu_\rho(\dd x \dd y), \qquad \scrR^*(\rho,\xi): = \frac12 \iint_{V\times V} \uppsi^*(\xi)  \,\bnu_\rho(\dd x \dd y),
\end{aligned}
\end{equation}
where $\upalpha$ is a concave mean-function, $(\uppsi,\uppsi^*)$ a convex conjugate pair of dissipation densities, and
\[ \bnu_\rho(\dd x \dd y) :=\upalpha(u(x),u(y)) \, \tetapi (\dd x, \dd y), \quad \tetapi (\dd x, \dd y):=\pi(\dd x)\kappa(x,\dd y), \quad u = \frac{\dd \rho}{\dd \pi}. \]

\begin{example}
\label{admissible-triples-intro}
It has been shown in \cite{PRST22} that, as soon as the triple  $(\upphi, \uppsi,\upalpha)$ satisfies the compatibility property
\begin{equation}
\label{admiss-compat}
(\Psi^*)'\big[\upphi'(v){-}\upphi'(u)\big]\upalpha(u,v) = v-u
   \quad
   \text{for every }u,v>0,
\end{equation}
solutions to \eqref{FGF-intro} do solve the Forward Kolmogorov Equation
 \eqref{FKEm}. 
 Thus, various choices of the above elements are possible. For instance, if we make the \emph{canonical} choice 
 \[
 \upphi(s)=s \log s-s+1,
 \]
  turning $\scrE$ into the relative Boltzmann  entropy of $\rho$ with respect to $\pi$, pairs 
  $(\uppsi^*,\upalpha)$ complying with 
  \eqref{admiss-compat} are given by 
\begin{align}
\uppsi^*(\xi)&:=\frac{1}{2}\xi^2, &\quad \alpha(u,v) &:=\frac{u-v}{\log{u}-\log{v}}, & \text{(\emph{linear quadratic})}
 \\[0.4em]
\uppsi^*(\xi) &:= 4 \left(\cosh(\xi/2)-1\right), &\quad \alpha(u,v) &:=\sqrt{uv}. & \text{(\emph{linear cosh})}
\end{align}
Here, the word `linear' hints at the resulting linear equation \eqref{FKEm}. The linear-quadratic structure was discovered in \cite{Maas11,ChowHuangLiZhou12,Mielke13CALCVAR}, while the linear-$\cosh$ structure stems from large deviation functionals of empirical measures of jump processes \cite{MielkePeletierRenger14}. 

 \par
 Nonetheless, we emphasize that different choices for $(\upphi, \uppsi,\upalpha)$ can lead to gradient-flow representations of various numerical schemes \cite{hraivoronska2023diffusive,Hraivoronska2024,Esposito2025}. The setting can be further generalized to one-way fluxes, even in the case of unbounded but locally bounded jump kernels \cite{HoeksemaTh}.
\end{example}
\medskip

As shown in \cite{PRST22}, equivalent formulations of the corresponding gradient-flow solution exist, either stating the evolution as the (possibly nonlinear) integro-differential equation
\begin{equation}\tag{$\mathsf{IDE}$}
\label{integro-diff-Fmapm}
\partial_t u_t (x) = \int_V \Fmap (u_t(x), u_t(y))\, \kappa (x, \dd y), \qquad \Fmap (u,v):=(\uppsi^*)'\left( \upphi'(v){-}\upphi'(u)\right) \upalpha (u,v),
\end{equation}
(in fact, the compatibility property \eqref{admiss-compat} reformulates as $ \Fmap (u,v) = v-u$), or characterizing the solution curve $(\rho,\jj)$ as the null-minimizer of the energy-dissipation functional 
 \begin{equation}
\label{trajectoryi}
\scrL_T(\rho,\jj) := \calS(\rho_T)-\calS(\rho_0)+ \int_0^T \bigl(\scrR(\rho_t, \bj_t) {+} \Fish(\rho_t)\bigr)\, \dd t,
\end{equation}
with the Fisher information functional $\Fish$ (see Section \ref{ss:dissd}). In fact, the validity of the aforementioned chain rule precisely amounts to having 
\begin{equation}
\label{CR-intro}
\scrL_t(\tilde\rho,\tilde\jj) \geq 0 \qquad \text{for all } t\in [0,T],
\end{equation}
for arbitrary curves $(\tilde\rho, \tilde\jj)$, 
whereas along a solution curve $(\rho,\jj) $  it holds 
\begin{equation}
\label{null-minimizer}
\scrL_T(\rho,\jj) =0 \,.
\end{equation}
Again, we emphasize that the validity of the chain rule, and hence the characterization of solutions in terms of \eqref{null-minimizer}, is made possible, in the setup of  \cite{PRST22}, by the regularity of the kernel and the corresponding continuity equation. 
\medskip

It is clear that a suitable generalization to singular kernels involves the following challenges:
\begin{itemize}
\item finding a notion of unbounded signed measures for arbitrary separable metric spaces, to represent the fluxes $\jj$,
\item replacing the class of bounded functions in the continuity equation \eqref{eq:contm} with an appropriate function space for which the equation is well defined,
\item introducing classes of solutions for which a suitable chain rule holds. 
\end{itemize}

The first challenge arises due to the fact that, in the present singular setting, the fluxes $\jj$ are not extended signed measures, which are defined over every set, and that real-valued Radon measures are not available because of the lack of (local) compactness. We handle this by introducing a concept of $\sigma$-finite signed measures on a given separable metric space $Y$
(the corresponding spaces will be denoted as $\Ms(Y)$;   $\Ms(Y;\R^d)$ for vector-valued measures), see Section \ref{ss:smmeas}. In essence,  a $\sigma$-finite signed measure  consists of tuples of mutually singular \emph{nonnegative} $\sigma$-finite measures. 

The last two challenges are more fundamental and intertwined: in particular, the choice of test functions for the continuity equation affects the validity of a chain rule. Therefore, a straightforward generalization of the works of \cite{Erbar14,warren2025} using Lipschitz functions as test functions is not viable, due to the absence of a general chain rule when Lipschitz functions are not dense in the relevant nonlocal Sobolev spaces.

\medskip

\paragraph{\bf Setup and dissipative solutions}
We are now going to illustrate the solution concepts introduced in this paper. Overlooking some technicalities (see Section \ref{s:3} for the precise definitions of all the objects involved), let $(\upphi, \uppsi^*,\alpha)$ be as above. Additionally, hereafter in this Introduction, we will assume that the density of the Fisher information functional $\mathscr{D}$ is convex, that the entropy density $\upphi$ is strictly convex, and that the dual dissipation density $\uppsi^*$ is asymptotically quadratic near zero. While all these conditions hold true for the canonical settings (cf.\ Example~\ref{admissible-triples-intro}), we invoke all of them here solely for expository purposes; in the upcoming Sections \ref{s:3-NEW} and \ref{s:4}, their specific role in the various statements will become clearer.
  
The weakest formulation of the continuity equation we shall address involves the class $\Lip_{\mathrm{b}}(V)$ of bounded $\dV$-Lipschitz functions. We consider density-flux pairs $(\rho,\jj)$, with $\rho_t \in \calM^+(V)$ for all $t\in [0,T]$ and a measurable family $(\jj_t)_{t\in [0,T]} \subset \Ms(V{\times}V)$, satisfying the \emph{continuity equation} in the following sense: for any interval $[s,t]\subset [0,T]$,
\begin{equation}\label{CE-INTRO}\tag{\textsf{CE}}
   \begin{aligned}
		\int_V \varphi(x)\, \rho_t(\dd x)  - \int_V \varphi(x)\, \rho_s(\dd x)   = \int_s^t\!\! \iint_{V\times V} \dnabla\varphi\,(x,y)\,\jj_r(\dd x\dd y)\,\dd r\qquad \text{for all } \varphi \in \Lip_{\mathrm{b}}(V)\,,
  \end{aligned}
\end{equation}
and write, in this case, $(\rho, \jj) \in \CE0T$.

We now define $\Dissipative$ solutions (see Section \ref{ss:3.notions-sols}) of the $(\scrE,\scrR,\scrR^*)$ evolution system as those pairs $(\rho, \jj) \in \CE0T$ 
for which $\calS(\rho_0)<\pinfty$ and $\scrL_t(\rho,\jj)\leq 0$ for every $t\in [0,T]$, i.e, those that satisfy the \emph{energy-dissipation inequality}
\begin{equation}
\label{UEDEi}\tag{\textsf{EDI}}
 \int_0^t \left( \scrR(\rho_r, \bj_r) + \Fish(\rho_r) \right) \dd r+ \calS(\rho_t)   \leq \calS(\rho_0)   \qquad \text{on any interval $[0,t]\subset [0,T]$}. 
\end{equation}

The relevance and usefulness of $\Dissipative$ solutions stems from the fact that the trajectory functional $\scrL_t$ is lower semicontinuous w.r.t.\ $(\rho,\jj)$, and also with respect to a varying kernel $\kappa$, under appropriate assumptions and strong convergence of the initial data. Because of this, $\Dissipative$ solutions frequently arise as limits in various types of variational convergence. For example, see \cite{Carrillo2022} for the transition of a sequence of grazing Boltzmann equations to the Landau equation, or \cite{Erbar16TR,Fathi2016,ErbarFathiLaschosSchlichting16TR,HT2023,HLS2025} for an exposition on its use in limits of interacting particle systems. 

In particular, in our current setting, we obtain the following existence result after approximating the singular kernel $\kappa$ by a sequence of bounded kernels 
\begin{equation}\label{eq:ikr}
\kappa_n(x,\dd y):=a_n(x,y) \kappa(x,\dd y),
\end{equation}
with $a_n(x,y)\in [0,1]$ an appropriately chosen sequence of regularizers or cut-off functions. 

\begin{mainthm}[Existence of $\Dissipative$ solutions]
\label{thm:mdissex}
There exists a $\Dissipative$ solution for any initial data $\rho_0 \in \dom(\scrE)$.
\end{mainthm}

Unfortunately, $\Dissipative$ solutions are not necessarily unique, as will be discussed below, unless the initial datum is bounded, and bounded Lipschitz functions are dense in the nonlocal Sobolev space 
\begin{equation}
\label{X2intro}
 \mathcal{X}_2:= \left\{ \varphi \in \rmL^\infty (V;\pi) : \ \iint_{V{\times}V} |\dnabla \varphi(x,y)|^2\ \tetapi(\dd x \dd y)<\infty \right\}.
\end{equation}
\par
Therefore, for our next main result, we need to consider curves that satisfy a stronger solution concept, which will heavily depend on a suitable solvability notion for the continuity equation. 

\medskip

\paragraph{\bf $\Reflecting$ solutions} 
  We start by specifying an \emph{enhanced} continuity equation underlying our stronger solutions. We say that a curve $(\rho,j)$ satisfies the 
\emph{reflecting continuity equation} on $(0,T)$, and write   $(\rho, \jj) \in \ECE 0T$, if 
\begin{enumerate}
\item for $u_t = \frac{\dd \rho_t}{\dd \pi}$ there holds $\sup\nolimits_{t\in [0,T]} \|u_t\|_{\rmL^\infty(\pi)} <\infty$;
\item the pair $(\rho,\jj)$ has finite action and finite Fisher information 
\begin{equation}
\label{finite-act+Fish-intro}
 \int_0^T \left(\scrR(\rho_t,\bj_t){+} \scrD(\rho_t) \right) \dd t <\infty, 
 \end{equation}
\item for any interval $[s,t]\subset [0,T]$, 
\begin{equation}
\label{RCE-INTRO}\tag{\textsf{RCE}}
		\int_V \varphi(x)\, \rho_t(\dd x)  - \int_V \varphi(x)\, \rho_s(\dd x)   = \int_s^t \!\!\iint_{V{\times}V} \dnabla\varphi\,(x,y)\,\jj_r(\dd x\dd y)\,\dd r \qquad \text{for all } \varphi \in \mathcal{X}_2\,.
	\end{equation}
\end{enumerate}
\par
Our stronger solution concept also involves 
the energy-dissipation \emph{balance} on \emph{every} sub-interval $[s,t]\subset [0,T]$. 
We say that a curve $ (\rho, \jj) \in \ECE 0T$ is a $\Reflecting$ solution
of the $(\scrE,\scrR,\scrR^*)$ evolution system
if $\scrL_t(\rho,\jj)=0$ for all $t\in [0,T]$, i.e., if the pair $(\rho,\bj)$ complies with the  {\em $(\scrE,\scrR,\scrR^*)$
  energy-dissipation balance}:  \EEE
\begin{equation}
\label{R-Rstar-balance-intro}\tag{\textsf{EDB}}
\int_s^t \left( \scrR(\rho_r, \bj_r) + \Fish(\rho_r) \right) \dd r+ \calS(\rho_t)   = \calS(\rho_s) \qquad \text{on any interval $[s,t]\subset [0,T]$} .
\end{equation}

\par 
We emphasize that, because of the validity of \eqref{R-Rstar-balance-intro}, $\Reflecting$ solutions are in particular $\Balanced$, in the sense of the solution concept introduced 
in \cite{PRST22} (cf.\ also Definition \ref{def:weak-solution} ahead).
Nonetheless, we have chosen the terminology `reflecting' to emphasize the role of the  \emph{reflecting continuity equation} for the validity of the chain rule, which is apparent in the following result. 
\begin{mainthm}[Characterization]
\label{thm:mchar}
For any curve $(\rho,\jj)\in \ECE 0T$ the following properties are equivalent:
\begin{itemize}
\item The energy-dissipation inequality holds,  i.e. $\scrL_t(\rho,\jj)\leq 0$ for all $t\in [0,T]$;
\item The energy-dissipation balance holds, i.e. $\scrL_t(\rho,\jj)= 0$ for all $t\in [0,T]$;
\item The equation \eqref{integro-diff-Fmapm} is satisfied in the sense that 
\[2\jj_t (\dd x, \dd y) : = -\Fmap(u_t(x),u_t(y))\tetapi (\dd x \dd y ) \qquad \foraa\, t \in (0,T).\]
\end{itemize}
\end{mainthm}
\par 
The reflecting continuity equation is at the core of the proof of the above characterization, which may not hold if the continuity equation is satisfied in a weaker sense, as illustrated by Figure \ref{fig1}. An immediate consequence of this is that any $\Dissipative$ solution $(\rho, \jj)$, originating from a bounded density, that also satisfies the reflecting continuity equation, is also a $\Reflecting$ solution (cf.\ Theorem \ref{thm:ECE}) 
\par
In Section \ref{ss:ASIDE} ahead, we will further explore the relationship between \eqref{CE-INTRO} and \eqref{RCE-INTRO}. In particular, we will show that curves $(\rho,\jj)\in \CE 0T$  with density $u = \frac{\dd\rho}{\dd \pi} \in \rmL^\infty(0,T;\rmL^\infty(V;\pi))$, do upgrade to $(\rho,\jj)\in \ECE 0T$, if the functional-analytic setup is such that bounded Lipschitz functions are \emph{dense} in $\mathcal{X}_2$. Accordingly, $\Dissipative$ solutions with bounded density upgrade to $\Reflecting$ solutions (cf.\ Corollary \ref{cor:density-vindicated}). 

\begin{figure}
\label{fig1}
\centering
\begin{tikzpicture}

  \draw[thick, rounded corners=6pt] (-5,-1.8) rectangle (5,1.8);

  \draw[thick, rotate=10] (-1.5,0.2) ellipse (3cm and 1.4cm);
  \node[anchor=south] at (-.5,.3) {\Large $\mathcal{R}\mathcal{C}\mathcal{E}$};
  \node[align=center, font=\itshape] at (3.8,1.2) {\Large$\mathcal{C}\mathcal{E}$};

  \fill ( -3.8, -0.6) node[above right, xshift=3pt] {\text{Dissipative} $=$ \text{Reflective}};
  \fill (  1.8,-0.9) node[right, xshift=-3pt] {$\begin{matrix}\exists \,\text{Dissipative} \\ \qquad\ne \text{Reflective}\end{matrix}$};
\end{tikzpicture}
\caption{An illustration of solution sets for bounded initial data. }
\label{fig:dissipative-balance}
\end{figure}
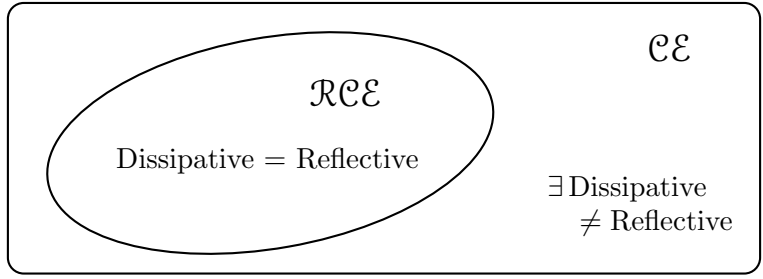

\medskip
A key feature of $\Reflecting$ solutions is that they are unique for strictly convex entropy densities $\upphi$; $\Reflecting$ solutions exist when the initial densities are bounded.
\begin{mainthm}[Existence and uniqueness $\Reflecting$ solutions]
\label{thm:mrexist}
For any fixed $\rho_0 \in \dom(\scrE)$, $\rho_0 = u_0 \pi$ with $\|u_0\|_{\rmL^\infty(\pi)}<\infty$, $\Reflecting$ solutions exist.
\par
Moreover, they 
are \emph{unique} within the class 
 $\mathcal{R}\mathcal{C}\mathcal{E}([0,T])$ with initial datum $\rho_0$.  
\end{mainthm}

In our existence result, we again use a sequence of regularized kernels as in \eqref{eq:ikr}, with appropriate regularizations $(a_n)_n$. In fact, due to the uniqueness mentioned above, this limit is independent of the choice of the sequence $(a_n)_n$. 

\begin{mainthm}[Robustness]
\label{thm:mrobust}
Let $a_n: V\times V \to  [0,1]$, $n \in \N$, form a sequence of pointwise converging functions, defining approximate kernels $(\kappa_n)_n$ via \eqref{eq:ikr}. Let $(\rho^n_0)_n\subset\dom(\scrE)$ be a sequence of initial data converging in entropy to $\rho_0$, i.e., $\scrE(\rho_0^n)\to \scrE(\rho_0)$.
\par
Then, the sequence of $\Reflecting$ solutions of the evolution systems $(\scrE, \scrR^n,(\scrR^n)^*)_n$ associated with the kernels $(\kappa_n)_n$ converge to the unique $\Reflecting$ solution of the evolution system $(\scrE, \scrR,\scrR^*)$ corresponding to $\kappa$. 
\end{mainthm}

\paragraph{\bf Reflecting boundary conditions and non-uniqueness}
At this point, it might be helpful to clarify our naming convention for $\Reflecting$ solutions, since, in the case of bounded kernels, the (unique) gradient-flow solutions are also $\Reflecting$ solutions. The following examples put the issue into clearer focus. 

\begin{example}
\label{ex:Jasp-intro-1}
Let us first consider the local setting, namely the canonical entropic Wasserstein gradient flow over some smooth domain $\Omega \subset \R^d$, corresponding to linear diffusion, \cite{AGS08}. Suppose that $\rho_t=u_t \Lebone^d$, $\jj_t=w_t \Lebone^d$ for all $t\in [0,T]$, with both $u_t$ and $w_t$ being smooth in time and space. Then the continuity equation usually takes the form
\begin{equation}
\int_{\Omega} \varphi(x)\, \rho_t(\dd x)  - \int_{\Omega} \varphi(x)\, \rho_s(\dd x) = \int_s^t \!\!\int_{\Omega} \nabla\varphi\,(x)\cdot \jj_r(\dd x)\,\dd r \qquad \text{for all } \varphi \in \mathrm{C}^{\infty}(\Omega),
\end{equation}
(or equivalently, under suitable regularity conditions on the fluxes, for all  $\varphi \in \Lipb(\Omega)$). Note that $\nabla \varphi$ is the usual gradient, contrasted with the discrete gradient $\dnabla \varphi$ used throughout this paper. 

On the other hand, the analogue of the reflecting continuity equation would involve curves that have finite action and Fisher information,  which, in the setting of linear diffusion, reads as
\[
    \int_0^T \left(\int_{\Omega} |w_t|^2 \rho_t(\dd x) +\int_\Omega |\nabla \sqrt{u_t}|^2 \dd x\right) \, \dd t < \infty\,,
\]
and the `enhanced' continuity equation (cf.\ \eqref{RCE-INTRO}) 
\begin{equation}
\int_{\Omega} \varphi(x)\, \rho_t(\dd x)  - \int_{\Omega} \varphi(x)\, \rho_s(\dd x)   = \int_s^t \!\!\int_{\Omega} \nabla\varphi\,(x)\cdot \jj_r(\dd x)\,\dd r \qquad  \text{for all }\varphi\in W^{1,2}(\Omega) \cap L^\infty(\Omega). 
\end{equation}
It is straightforward to check that in the latter case, the no-flux boundary condition 
\[\jj_t(x) \cdot \mathbf{n}_{\Omega}(x)=0\qquad  \mbox{ for all } x \in \partial \Omega, \, t\in [0,T], \]
is enforced, signifying reflection. In particular, if $\Omega=\Omega_1\cup \Omega_2$ with $\Omega_1$ and $\Omega_2$ both smooth and strongly separated, the masses $[0,T]\ni  t \mapsto \rho_t(\Omega_i)$ are constant in time. 

However, problems arise when $\Omega_1$ and $\Omega_2$ are not strongly separated, for example, when $\Omega$ is a punctured domain. Consider $\Omega=(-1,0) \cup (0,1)$. Then, Lipschitz functions are clearly not dense in $W^{1,2}(\Omega)$. Note that the classical Meyers-Serrin density result implies the density of smooth functions over $\Omega$, which need not be Lipschitz over $\overline\Omega$. Clearly, in this setting, the Euclidean metric is not the `right' metric for this problem, unless solutions are considered for \emph{permeable} membranes, i.e., corresponding to processes of particles that pass or transmit through the origin (see \cite{Chiarini2018} for a discussion on a similar phenomenon for diffusion processes with singular drifts). 
\end{example}

\begin{example}
\label{ex:Jasp-intro-2}
Similar issues might occur with singular jump kernels, but they could be less obvious due to potential degeneracies in the kernel. Take for example $\Omega=[-1,1]$, $\Omega_1=[-1,0)$, $\Omega_2=(0,1]$, and the singular kernel
\[\kappa(x,\dd y) = a(x,y) |x{-}y|^{-(1{+}2s)}, \qquad a(x,y):=1_{\Omega_1\times \Omega_1 \cup \Omega_2 \times \Omega_2}(x,y),\]
with a sufficiently strong singularity, namely $s>\tfrac{1}{2}$. It can be shown (see Section \ref{app:examples-density}) that bounded Lipschitz functions are not dense in $\mathcal{X}_2$ (cf.\ \eqref{X2intro}), since their closure results in functions $\varphi$ that are continuous in $x=0$, while $\mathcal{X}_2$ allows for discontinuities, representing the difference between reflection and transmission. This leads to  nonuniqueness of $\Dissipative$ solutions, potentially resulting in two distinct solutions for different admissible metrics $d_1$ and $d_2$, as shown in Figure \ref{figece}.

Both for this example and more general domains in $\R^d$, this can also be expressed using nonlocal analogs of partial integration (but still \emph{local} boundary terms) and traces for fractional kernels (\cite{Guan2006}). Moreover, $\mathcal{X}_2$ can be viewed as the largest space for which the expression is well-defined and can be identified as the maximal Silverstein extension of subservient Dirichlet forms, and is also known as \emph{active reflected} Dirichlet form \cite{Schmidt2018}. 
\par
In Section \ref{app:examples-density}, we will give various examples where the density of Lipschitz functions in $\mathcal{X}_2$ can still be proven, typically relying on convolution techniques.
\end{example}

\begin{figure}[h!]\centering
\begin{tikzpicture}
  \tikzset{ set/.style={thick}, arrow/.style={->,thick} }

  \draw[set] (0,0) circle (1.3cm);
  \node[font=\Large] at (0,0) {$\mathcal{R}\mathcal{C}\mathcal{E}$};

  \draw[set, rotate=12]  (0,0) ellipse (4cm and 1.8cm);
  \draw[set, rotate=-12] (0,0) ellipse (4cm and 1.8cm);

  \node[font=\Large\itshape] at (4.5,1.5) {$\mathcal{C}\mathcal{E}_{d_1}$};
  \node[font=\Large\itshape] at (-4.5,1.5) {$\mathcal{C}\mathcal{E}_{d_2}$};

  \node at (0,-3) (balance) {$d_2$-Dissipative $=$  Reflecting  $=$ $d_1$-Dissipative};
  \draw[arrow] (balance.north) -- (0,-1);

  \node at (7.2,0.2) (notbalance) {$\exists$ $d_1$-Dissipative $\neq$  Reflecting};
  \draw[arrow] (notbalance.west) -- (2.9,1.0);

\end{tikzpicture}
\caption{\label{figece}An illustration of solution sets for bounded initial data showing that
\emph{(i)} for two different metrics $d_1$, $d_2$ on the \emph{same} space $V$, the class of $\Dissipative$ solutions may or may not coincide with that of $\Reflecting$ solutions: For instance, if $(V,d_2)$ enjoys the additional property that functions in $\Lipb(V,d_2)$ are dense in $\mathcal{X}_2$, then $d_2$-$\Dissipative$ solutions are also $\Reflecting$, cf.\ Corollary \ref{cor:density-vindicated}; \emph{(ii)} for both metrics $d_1$ and $d_2$, $\Dissipative$ solutions that satisfy the reflecting continuity equation \eqref{RCE-INTRO} are also $\Reflecting$, see Theorem \ref{thm:ECE}.}
\label{fig:dissipative-balance2}
\end{figure}
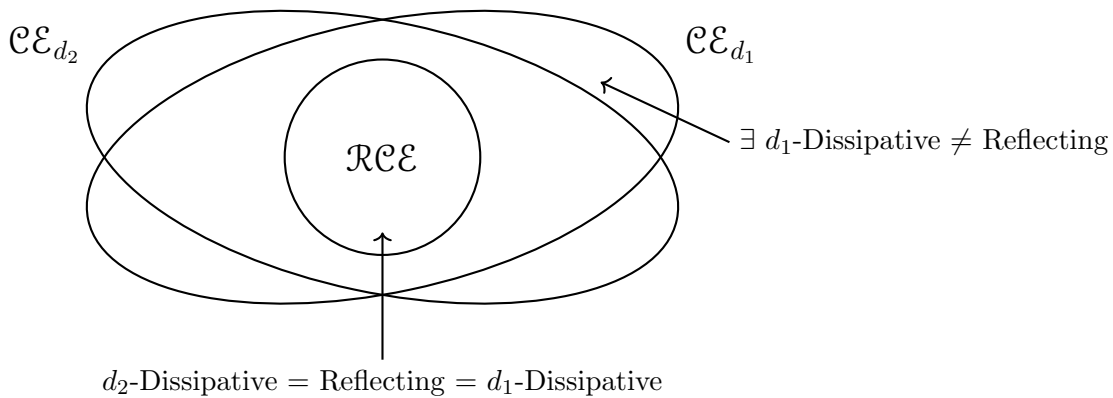

\medskip

\paragraph{\bf Outlook to generalizations}

Our proposed reflecting continuity equation incorporates a bound on the densities $u_t = \frac{\dd \rho_t}{\dd \pi}$, namely the constraint $\sup_{t\in [0,T]} \|u_t\|_{\rmL^\infty(\pi)} <\infty$, in order to establish a chain rule. In fact, the proof of the chain rule depends on bounds on $u$ from above and below; see Section \ref{thm:ECE}, where the lower bound can be removed via a convexity argument. Adjustments can be made to $\mathcal{R}\mathcal{C}\mathcal{E}$ to remove the upper bound, while still avoiding issues like those in Example \ref{ex:Jasp-intro-2}, by choosing a strictly smaller space than $\mathcal{X}_2$, but possibly larger than the space of bounded Lipschitz functions. However, due to the absence of a chain rule, these solutions may face the same problems as our current $\Dissipative$ solutions, particularly a lack of uniqueness. It should be noted that when the kernel is translation invariant, convolution techniques can be employed to remove the upper bound, similar to the approach taken in \cite{Erbar14}. 
%

A related and more fundamental issue is in our choice of function space $\mathcal{X}_2$, which leads to our class of $\Reflecting$ solutions. A natural question is whether our framework can be generalized to any other class of functions $\mathscr{F}\subset \mathcal{X}_2$, whether bounded Lipschitz functions or otherwise, that can relate to various other notions of boundary conditions. 

In the Hilbertian setting, with $\phi(s)=s^2$, $\alpha(u,v)=1$, $\uppsi^*(\xi)=\frac{1}{2}\xi^2$, a partial positive answer exists via the use of Dirichlet forms in \cite{Fukushima1994}, where any $\mathscr{F}$ that is closed in our nonlocal Sobolev space and satisfies suitable assumptions (including the class of bounded Lipschitz functions) generates a unique strongly continuous semigroup. This can be adapted to our setting but results in a chain rule and uniqueness for a corresponding Fisher information $\Fish_{\mathscr{F}}$ (possibly) strictly smaller than our current $\Fish$, with $\Fish_{\mathscr{F}}(\rho)=+\infty$ if $u=\frac{\dd \rho}{\dd \pi} \notin \mathscr{F}$. 

The problem lies in the fact that our current \emph{existence} result uses a specific choice of regularizations, and naturally leads to $\Reflecting$ solutions with the Fisher information $\Fish$. Other techniques that lead to existence, such as minimizing movement schemes with (possibly generalized) local slopes as found in \cite{PRST22}, either need precise information on geodesics for the solutions as in \cite{Erbar14}, or control on approximate solutions \cite{PRST22}. It remains to be seen if a combination of such techniques can be generalized to our setting with arbitrary $\mathscr{F}$. 

\medskip


Finally, our choice of objects involved leads to a very general but specific class of evolution equations, which in particular satisfies a maximum principle and incorporates a one-particle jump kernel $\kappa$. At the same time, there is a now a rich literature on gradient flows for the spatially homogeneous Boltzmann and Landau equation (see e.g.\ \cite{Erbar16TR,Carrillo2022,Carrillo2024}) or exchange-type equations \cite{Schlichting2019,HLS2025}, incorporating collision or two-particle kernels of the form 
\[\kappa(x,y;\dd x,\dd y).\] 
An important question, arising from all of the above remarks, is whether our current results can be generalized to Boltzmann-type equations over general metric spaces, picking out a specific class of solutions, \emph{reflecting} or otherwise. 

\medskip

\paragraph{\bf Plan of the paper} \underline{Section \ref{s:prelims}} settles the preliminary measure-theoretic definition and notions upon which we will build our analysis. 
\par
In \underline{Section \ref{s:3}}, we introduce the setup for our variational structure, by detailing our conditions on the ambient space $V$, with reference measure $\pi$, on the kernels $(\kappa(x,\cdot))_{x\in V}$, on the entropy, flux, and dissipation densities $\upphi$, $\upalpha$, and $\uppsi$, respectively.  We then proceed to specify our first solvability notion for the continuity equation and rigorously define the primal and dual potentials.
\par
\underline{Section \ref{s:3-NEW}} revolves around the concepts of solutions for the evolution system studied in this paper: $\Dissipative$, $\Balanced$, and $\Reflecting$ solutions (the latter properly defined after the introduction of the reflecting continuity equation). We explore their relationship, particularly proving \textsc{Main Theorem} $2$. Moreover, we show that if bounded Lipschitz functions are dense in $\mathcal{X}_2$, then curves fulfilling the continuity equation in the sense of Section \ref{s:3-NEW}, with additional bounds, can be upgraded to curves satisfying the reflecting continuity equation. 
 \par
 In \underline{Section \ref{s:4}}, we present the approximation of our evolution system $(\mathscr{E}, \mathscr{R},  \mathscr{R}^*)$ and state our main existence results, in particular yielding  \textsc{Main Theorems} $1$ \& $3$. We prove them in \underline{Section \ref{s:5}} via a robustness result that incorporates 
 \textsc{Main Theorem} $4$.
 \par
In \underline{Section \ref{s:8}}, we develop an application of our results for singular nonlocal stochastic evolutions in configuration spaces, demonstrating the flexibility of our assumptions.
\par
The \underline{Appendix} collects miscellaneous results and proofs, as well as examples of functional analytic setups in which the density property holds.

\section{Preliminaries of measure theory}
\label{s:prelims}

\paragraph{\bf List of symbols.}
Throughout the paper we will use the following notation.
\begin{center}
\newcommand{\specialcell}[2][c]{%
  \begin{tabular}[#1]{@{}l@{}}#2\end{tabular}}
\begin{small}
\begin{longtable}{lll}
$\R_+:=[0,\pinfty[$, $\R^m_+:=(\R_+)^m$ 
\\
$\overline{\R}^{m} = (\overline{\R}{\times}\ldots {\times}\overline{\R})$  & extended positive half-line, extended real line & 
\\
$\Rpe=[0,+\infty], \, \overline\R$  &  & 
\\
$ \B(Y), \, \B^+(Y; \overline{\R}) , \, \B(Y;\overline{\R})  $ &  real-valued, non-negative extended, or extended (Borel) 
\\
 & measurable functions on a separable metric space $Y$
\\
$\Bb(Y), \,  \Bb(Y;\R^m)$ & bounded Borel real-valued,  $\R^m$-valued functions 
\\
$\| f\|_\infty$ & $\sup$-norm of a function $f\in \Bb(Y)$ or $\Bb(Y;\R^m)$ 
\\
$\Cb(Y;\R^m), \Lip(Y;\R^m)$  & bounded continuous (Lipschitz, resp.)
$\R^m$-valued functions 
\\
$\Lipb(Y;\R^m) = \Lip(Y;\R^m) \cap \Cb(Y;\R^m)$ & bounded Lipschitz 
$\R^m$-valued functions & 
\\
$\mathcal{M}(Y;\R^m)$, $\mathcal{M}^+(Y)$ & Borel vector-valued (positive Borel)  measures on $Y$, & 
\\
 & with finite total variation & 
\\
$\Ms^+(Y)$ & positive $\sigma$-finite measures
\\
$\Ms(Y)$,\, $\Ms(Y;\overline{\R}^m)$ &  $\sigma$-finite signed measures with values in $\overline\R$,  $\overline{\R}^m$ resp.\
\\
$\boldsymbol{1}_{A}$ & indicator function of a set $A$
\\
$\Lebone$ ($\Lebone^d$) & Lebesgue measure on $\R$ ($\R^d$, resp.) & 
\\
$V$, $E=V{\times}V $ & space of states \& space of edges &
\\
$\Ed = E{\setminus} \{ (x,x)\, : x \in V\} $
\\
$\ona$ & graph gradient &
\\
$\kappa$, $\tetapi$ & jump kernel \& coupling associated with $\kappa$ \& $\pi$ &
\\
$\upphi$ and $\calS$ &  entropy density and entropy functional &
\\
$\uppsi$ and $\uppsi^*$ & dual pair of dissipation functions &
\\
$\calR$, $\calR^*$ & dual pair of dissipation potentials & 
\\
$\bnu_\rho$ & edge measure in definition of $\calR^*$, $\calR$
\\
$\upalpha(\cdot,\cdot)$ & multiplier in flux rate $\bnu_\rho$ & 
\\
$\Fish$ & Fisher-information functional &
\\
 $\rmD_\upphi$, $\rmD^+_\upphi$ & integrands
             defining the Fisher information $\Fish$ &
             \\
$\CE ab$ & set of pairs $(\rho,\bj )$ satisfying the \eqref{CE-INTRO} & 
             \\
$\ACE ab$ & set of pairs $(\rho,\bj ) \in \CE ab$ 
with finite entropy, & 
\\ & action and Fisher information &
\\
$\ECE ab$ & set of pairs $(\rho,\bj ) \in \ACE ab$ satisfying \eqref{RCE-INTRO} &
\end{longtable}
\end{small}
\end{center}

\paragraph{\bf Finite measures.}
Let $(Y,d)$ be separable metric space, and let 
$\mathfrak{B}(Y)$ be its associated Borel $\sigma$-algebra. 
We denote by 
 $\calM(Y;\R^m)$  the space of  $\sigma$-additive measures  on
 $\mu: \frB(Y) \to \R^m$ 
 of \emph{finite} total variation, 
 $\|\mu\|_{\mathrm{TV}}: =|\mu|(Y)<\pinfty$, where for every $B\in\frB(Y)$
 \[
   |\mu|(B): = \sup \left\{ \sum_{i=0}^\pinfty |\mu(B_i)|\, : \ B_i \in \frB_Y,\, \ B_i \text{ pairwise disjoint}, \ B = \bigcup_{i=0}^\pinfty B_i \right\}.
 \]
 We shall also refer to the elements in $\calM(Y;\R^m)$ as \emph{finite Borel measures}.  
 The set function  $|\mu|: \frB(Y) \to [0,\pinfty)$  is a positive
 finite
 measure on $\frB(Y)$ \cite[Thm.\ 1.6]{AmFuPa05FBVF}
 and $(\calM(Y;\R^m),\|\cdot\|_{\mathrm{TV}})$ is a Banach space.
 In the case $m=1$, we will simply write $\calM(Y)$,
 and we shall denote the space of \emph{positive} finite
measures on $\frB(Y)$ by $\calM^+(Y)$.
 By the Hahn decomposition, any $\mu \in \calM(Y)$ can be written as the difference of two positive finite measures
$\mu^\pm \in \calM^+(Y)$, i.e.,
$\mu = \mu^+ - \mu^-$.

   For $m>1$,
 we will identify any element $\mu \in \calM(Y;\R^m)$ with  a vector
 $(\mu_1,\ldots,\mu_m)$, with $\mu_i \in \calM(Y)$ for all
 $i=1,\ldots, m$,  with respective positive and negative parts $\mu_i^{\pm}$.  
 We will denote by $\Bb(Y;\R^m)$  the set of bounded $\R^m$-valued
 $\frB$-measurable
 maps;
 if $\varphi
 =(\varphi_1,\ldots,\varphi_m)\in \Bb(Y;\R^m)$, the duality between $\mu \in \calM(Y;\R^m)$ and $\varphi$
 can be expressed by 
\begin{equation}
\label{vector-duality-measures}
 \langle\mu,\varphi\rangle : = \int_{Y} \varphi \cdot \mu (\dd x) =
 \sum_{i=1}^m \int_Y  \varphi_i(x) \mu_i(\dd x).
 \end{equation}

\paragraph{\bf Extended and signed measures}
We call \emph{Borel measure} any measure $\mu: \mathfrak{B}(Y) \to \Rpe$, with the extended Borel $\sigma$-algebra on $\Rpe$.
\par
Furthermore,  we 
denote by $\calM^+_{\sigma}(Y)$ the space of $\sigma$-finite Borel measures. 
Throughout, we will repeatedly use the fact that for any $\mu,\nu\in \calM_{\sigma}^+(Y)$ with corresponding monotone sequences $(A_n)_{n\in \N}$ and $(B_n)_{n\in \N}$ that exhaust $Y$, we can  always find 
\begin{equation}
\label{common-exhaustion}
\begin{gathered}
\text{a new monotone sequence that exhausts $Y$ with both $\mu,\nu$ finite on these sets,}
\\
\text{by renumbering $C_{n,m}:=A_n \cap B_m$.}
\end{gathered}
\end{equation}
\par
Finally, recall from \cite[Definition 10.1]{Yeh} that any \emph{extended signed} Borel measure $\mu:\mathfrak{B}(Y) \to \overline{\R}$ (with the extended Borel $\sigma$-algebra on $\R$), 
is a countably additive  set function, with  $\mu(\emptyset)=\emptyset$, such that  
 either $\mu(B)\in (-\infty,+\infty]$ for every $B\in  \mathfrak{B}(Y)$, or
 $\mu(B)\in [-\infty, +\infty)$ 
 for every $B\in  \mathfrak{B}(Y)$. 
It should be noted that an extended signed measure is defined over the \emph{whole} $\sigma$-algebra, and is forbidden to attain $+\infty$ and $-\infty$ on two different sets. 

\par Nonetheless, in this paper we will need to handle measures like
\begin{subequations}
\label{clarifying-example}
\begin{equation}
\mathfrak{m}(\dd x):=x^{-1} \Lebone(\dd x)
\quad \text{ over $\R\backslash \{0\}$,}
\end{equation}
which is not an extended signed measure, since $\mathfrak{m}(B)$ is not defined for e.g.\ $B=(-1,0)\cup (0,1)$. 

However, $\mathfrak{m}$ is still a \emph{Radon} measure over $\R\backslash \{0\}$, where (see \cite{Bourbaki2004IntI,folland1999real}) a Radon measure $\nu$ on a locally compact separable $Y$ can equivalently be defined as either a set function defined over $\cup_{n \in \N} \mathfrak{B}(K_n)$ for some compact exhaustion $(K_n)_{n\in \N}$ of $Y$, or as a continuous functional over $\Cc(Y)$ with the inductive limit topology induced by $\mathrm{C}(K_n)$, or as the difference between two mutually singular positive Radon measures $\nu^{\pm}$. Note that, in the latter case, the decomposition $\nu=\nu^+-\nu^-$ is unique. For $\mathfrak{m}$, 
such decomposition is provided by 
\begin{equation}
\label{clar-ex-pm}
\mathfrak{m}=\mathfrak{m}^+-\mathfrak{m}^- \qquad  \text{with $\mathfrak{m}^{\pm}(\dd x ):=\boldsymbol{1}_{\{\pm r>0\}}(x)  |x|^{-1} \Lebone(\dd x)$.}
\end{equation}
\end{subequations}

Here, $Y$ can be a general separable metric space and is not required to be locally compact. Consequently, real-valued Radon measures are not applicable to our flux measures. Therefore, in the forthcoming section, we shall present a new construction in which $\sigma$-finiteness will assume a pivotal role, compensating for the lack of local compactness.

\subsection{Sigma-finite signed measures}\label{ss:smmeas}
The above preliminaries have set the stage for our own version of signed `measures', which need not be extended signed measures or real-valued Radon measures. 
In our setting, the $\sigma$-finiteness of all the measures involved
in Definition \ref{def:msigma}
will play a similar role as the compact exhaustion for Radon measures, ensuring that our definition is both well-posed and natural for the spaces we are interested in. 



\begin{definition}\label{def:msigma}
The space $\calM_{\sigma}(Y)$ consists of pairs $(\mu^+,\mu^-)$ such that $\mu^+\!\perp \mu^-$ and $\mu^{\pm}\in \calM_{\sigma}^+(Y)$. 
We call $(\mu^+,\mu^-)$ a \emph{$\sigma$-finite signed measure}. 
\par
For any such pair, we set $|(\mu^+,\mu^-)|:=\mu^+ + \mu^-$, and consider the collection 
\[
    \mathfrak{A}_{(\mu^+,\mu^-)} := \bigl\{ B\in \mathfrak{B}(Y) : \mbox{either $\mu^+(B)<\infty$ or $\mu^{-}(B)<\infty$} \bigr\}.
\]
We define the signed set function $\mu:\mathfrak{A}_{(\mu^+,\mu^-)}\to \overline\R$ associated to $(\mu^+,\mu^-)$ by
\begin{equation}
\label{induced-mu}
\mu(B):=\mu^+(B)-\mu^-(B), \qquad B\in \mathfrak{A}_{(\mu^+,\mu^-)}.  
\end{equation}
\par
In addition, for any $f\in \B(Y;\overline \R)$ and $(\mu^+,\mu^-) \in \calM_{\sigma}(Y)$ such that $|f|\in L^1(|(\mu^+,\mu^-)|)$ we define the finite integral
\begin{equation}\label{eq:fsint}
\int_{Y} f \dd \mu:= \int_Y f^+ \dd \mu^+ - \int_Y f^+ \dd \mu^- - \int_Y f^- \dd \mu^+ + \int_Y f^- \dd\mu^-.
\end{equation}
Similarly, $\calM_{\sigma}(Y;\overline{\R}^m)$ consists of tuples $(\mmu^+,\mmu^-) = ((\mu_1^{+},\mu_1^{-}), ,\dots, (\mu_m^{+},\mu_m^{-})) \in \calM_{\sigma}^+(Y)$, with $|(\mmu^+,\mmu^-) |:=\sum_{i=1}^m (\mu^+_i+\mu^-_i)$, and we define a vector-valued set-valued function $\mmu$
via \eqref{induced-mu}, componentwise. 
\end{definition}
We have the following result, whose proof is postponed to Appendix \ref{app:preliminary}. 
\begin{lemma}\label{lm:sigmap}
The following holds:
\begin{enumerate}
\item For any two pairs 
$(\mu^+,\mu^-),\,(\nu^+,\nu^-)\in \calM_{\sigma}(Y)$ such that 
\[
    \mu^+(B) - \mu^-(B) =\nu^+(B) - \nu^-(B) \qquad \text{for all } B \in \mathfrak{A}_{(\mu^+,\mu^-)}\cap \mathfrak{A}_{(\nu^+,\nu^-)},
\]
we have $\mu^\pm=\nu^\pm$.
\item Let $\{A_n\}_{n\in \N}$ be an increasing exhaustion
of $Y$ and $\nu:\cup_{n=1}^{\infty} \mathfrak{B}(A_n) \to \R$ be any set function such that for every $n\in \N$  the restriction $\nu\mres A_n$ is a finite signed Borel measure. Then there exists a unique pair $(\mu^+,\mu^-)\in \calM^+_{\sigma}(Y)$ such that $\nu=\mu^+-\mu^-$ on $\cup_{n=1}^{\infty} \mathfrak{B}(A_n)$. 
\end{enumerate}
\end{lemma}
\par
Relying on statement (1) of Lemma~\ref{lm:sigmap}, we may identify the pair $(\mu^+,\mu^-)$ with the set-function $\mu $ from \eqref{induced-mu}.
Therefore, hereafter we will write $\mu \in \Ms(Y)$ 
($\mmu \in \Ms(Y;\overline{\R}^m)$, respectively). 
With $\mu \in \Ms(Y)$ we can associate an exhaustion $(C_n)_n$ defined from the exhaustions for $\mu^+$ and $\mu^-$ via the procedure described in \eqref{common-exhaustion}.

\begin{remark}
\label{rmk:warranted-disc}
As far as the authors are aware, the above approach is not standard and warrants some discussion.
\par
Firstly, observe that, unlike extended signed measures, the  set-functions $\mu$ from \eqref{induced-mu} are not defined on the whole of $\mathfrak{B}(Y)$. In turn, any real-valued Radon measure on a locally compact separable space is $\sigma$-finite signed measure, since it can be represented as a difference of two mutually singular $\sigma$-finite measures. But, $\sigma$-finite signed measures are more general than real-valued Radon measures. 

Indeed, let us return to $\mathfrak{m}^+, \,\mathfrak{m}^-$ defined on $\R\backslash \{0\}$ from our example \eqref{clarifying-example}. When we extend $\mathfrak{m}^{\pm}$ to measures over $\R$ instead of $\R\backslash \{0\}$, setting $\mathfrak{m}_{\mathrm{ext}}^{\pm}(\{0\})=0$, it is clear that $\mathfrak{m}_{\mathrm{ext}}^{\pm}$ and $\mathfrak{m}_{\mathrm{ext}} = \mathfrak{m}_{\mathrm{ext}}^+ - \mathfrak{m}_{\mathrm{ext}}^-$ are no longer Radon measures on $\R$. However, since $\mathfrak{m}_{\mathrm{ext}}^{\pm}$ are still $\sigma$-finite, we have that $\mathfrak{m}_{\mathrm{ext}}^{\pm}$ is a $\sigma$-finite signed measure. 
\end{remark}
 

\paragraph{\bf Properties of $\sigma$-finite signed measures.}
The following result collects some basic facts that will be used afterwards:
 $\calM_{\sigma}$ is closed under suitable operations and a version of dominated convergence theorem holds. 
\begin{lemma}
\label{l:basic-props}
\begin{enumerate}
\item For any $\mu,\nu \in \calM_{\sigma}(Y)$ there exists a unique element $\eta\in \calM_{\sigma}(Y)$ such that $\eta(B)=\mu(B)+\nu(B)$ for all $B$ for which the terms involved are finite. In particular, when denoting this element as $\mu+\nu$, $\calM_{\sigma}(Y)$ becomes a vector space.
\item For any $f\in \B(Y;\overline{\R})$ and $\nu\in \calM_{\sigma}(Y)$ such that $|f|$ is $|\nu|$-a.e.\ finite, we have
\[
    f \nu:=(f^+ \nu^+{+}f^- \nu^-,f^- \nu^+{+}f^+\nu^-) \in \calM_{\sigma}(Y).
\]
\item The modified dominated convergence theorem holds:

\noindent For any $(f_n)_{n\in \N} \subset \B(Y;\overline{\R})$, $f\in \B(Y;\overline{\R})$, $(g_n)_{n\in N} \subset \B^+(Y;\overline{\R}), g\in \B^+(Y;\overline{\R})$, we have  
\[
    \left.\begin{gathered}
     \text{$f_n\to f$ $|\nu|$-a.e.},\;\;\text{$|f_n|\leq g_n$ $|\nu|$-a.e.}, \\
     \text{$g_n\to g$ $|\nu|$-a.e.\ with 
$\textstyle\int_{Y} g_n \, \dd |\nu|  \to \int_Y g\, \dd |\nu|$}
    \end{gathered}\;\right\}\quad\Longrightarrow\quad \text{$\int_Y f_n \,\dd \nu \to \int_Y  f \,\dd \nu$}.
\]
\end{enumerate}

\end{lemma}
\noindent
As for point \emph{(2)}, note that for arbitrary extended measurable functions we might no longer have the desired identification, since then the resulting measures might no longer be $\sigma$-finite. 
We postpone the proof of Lemma \ref{l:basic-props}, as well, to Appendix \ref{app:preliminary}.

\medskip
\paragraph{\bf Restriction.}
For every $\mu\in \Ms(Y;\overline{\R}^m)$ and $B\in \frB(Y)$
 we denote by $\mu{\mres} B$ the
 pair $(\mu^+{\mres}B, \mu^-{\mres}B)$. Hence, we have  
 \[
 (\mu{\mres} B)(A):=\mu(A{\cap}B)
 \]
for every $A\in \frB(Y)$ such that 
$\mu(A{\cap}B)$ is well-defined. 

\medskip
\paragraph{\bf Lebesgue decomposition.} 
 
 For every $\mu\in \Ms(Y;\overline{\R}^m)$ and $\gamma\in \Ms^+(Y)$ there exists
  a unique (up to the modification
 in a $\gamma$-negligible set) measurable map
 $\frac{\dd\mu}{\dd\gamma}:
 Y\to\R^m$, a $\gamma$-negligible set $N\in \frB$
 and a unique measure $\mu^\perp\in \Ms(Y;\overline{\R}^m)$
 yielding 
 \begin{equation}
   \label{eq:Leb}
   \begin{gathered}
     \mu=\mu^{\mathrm{a}}+\mu^\perp,\quad \mu^{\mathrm{a}}=\frac{\dd\mu}{\dd\gamma}\,\gamma=
     \mu{\mres}(Y\backslash N),\quad \mu^\perp=\mu{\mres} N,\quad
     \gamma(N)=0\\
     |\mu^\perp|\perp \gamma,\quad |\mu|(Y)=\int_Y
     \left|\frac{\dd \mu}{\dd\gamma}\right|\,\dd\gamma+|\mu^\perp|(Y).
   \end{gathered}
 \end{equation}
 This decomposition is obtained by combining  
 the Lebesgue decompositions of $\mu^{\pm}$ w.r.t.\ $\gamma$.
 \EEE
\subsection{Convergences  of measures}
\label{ss:2.2}
The space 
$\calM(Y;\R^m)$ is clearly endowed with the notion of convergence induced by the total variation norm  $\|\cdot\|_{\mathrm{TV}}$; we shall refer to it as \emph{strong convergence}. On
$\calM(Y;\R^m)$, we also consider \emph{setwise convergence}, defined by
  \begin{equation}
  \label{setwise-convergence}
  \mu_n \to \mu \text{ \emph{setwise} }  \quad \text{if}  \quad \lim_{n\to\infty}  \mu_n(B) = \mu(B) \quad \text{for all } B \in \frB(Y).
  \end{equation}
  Observe that,
  as a consequence of \cite[Theorem 4.6.3]{Bogachev07},
  along a setwise converging sequence $(\mu_n)_n$ we have
  $ \sup_{n} \|\mu_n\|_{\mathrm{TV}} < \infty$.  
   In fact, the  topology of setwise convergence is the coarsest one on $\calM(Y;\R^m)$ making all the functions $ \frB(Y)\ni B
  \mapsto  \mu(B) $ continuous. 
 \EEE
  Among the various characterizations of setwise convergence, we recall  \cite[\S 4.7(v)]{Bogachev07}
 \begin{enumerate}
 \item \emph{Convergence in duality with $\Bb(Y;\R^m)$}: $ \mu_n \to \mu$ setwise if and only if 
   \begin{equation}
     \label{eq:70}
     \lim_{n\to\pinfty}\langle \mu_n,\varphi\rangle=
     \langle \mu,\varphi\rangle
     \qquad
     \text{for every $\varphi\in \Bb(Y;\R^m)$}.
   \end{equation}
   ($\Bb(Y;\R^m)$ denoting the space of  bounded Borel $\R^m$-valued functions);
 \item \emph{Weak topology convergence in $ \calM(Y;\R^m)$}: $ \mu_n \to \mu$ setwise if and only if 
 $\mu_n$ converges to $\mu$ w.r.t.\ the weak topology
   of the Banach space $(\calM(Y;\R^m);\|{\cdot}\|_{\mathrm{TV}})$.
   \end{enumerate}
   Furthermore,  \RNEW we mention that, 
   by  \cite[Thm.\ 4.7.25)]{Bogachev07}, given a sequence $(\mu_n)_n \subset \calM(Y;\R^m) $, 
   the following properties are equivalent: \EEE
   \begin{itemize}
   \item[(i)]    $(\mu_n)_n$ is relatively compact w.r.t.\ setwise convergence;
   \item[(ii)]    
   \RNEW   there exists $\gamma \in \calM^+(Y)$ such that 
      \begin{equation}
         \label{eq:73}
         \forall\,\eps>0\ \exists\,\delta>0:
         \quad
         B\in \frB(Y),\ \gamma(B)\le \delta\quad
         \Rightarrow\quad
         \sup_{n}|\mu_n|(B)\le \eps;
       \end{equation}
          \item[(iii)] there exists $\gamma \in \calM^+(Y)$ such that $\mu_n \ll \gamma$ for all $n\in \N$ and the sequence $\bigl(\frac{\dd \mu_n}{\dd \gamma}\bigr)_n$
          admits a subsequence weakly converging in the topology of $\rmL^1(Y,\gamma;\R^m)$.
       \end{itemize}
       \EEE
       \par

\begin{remark}
It should be noted that \eqref{eq:70} can be slightly improved. Namely, suppose $\mu_n\to \mu$ converges setwise, and consider any uniformly bounded sequence $(\varphi_n)_n \subset \Bb(Y;\R^m)$ converging pointwise (or almost everywhere up to a common dominating measure for  $(\mu_n)_{n}$) to some $\varphi\in \Bb(Y;\R^m)$. Then, by a similar argument as in \cite[4.7.130]{Bogachev07}, we have 
\begin{equation}
\label{eq:pointset}
    \lim_{n\to\pinfty}\langle \mu_n,\varphi_n\rangle = \langle \mu,\varphi\rangle.
\end{equation}
\end{remark}


We recall that a sequence $(\mu_n)_n \subset \calM(Y;\R^m)$ converges \emph{narrowly} to some $\mu \in  \calM(Y;\R^m)$
if \eqref{eq:70} holds for every continuous \emph{bounded} function $\varphi \in \Cb(Y)$. 
\par
Finally, let us endow the space $\Lip_{\mathrm{b}}(Y)$
with the norm
\begin{equation}
\label{Lip-bounded-b}
\RNEW \| \varphi \|_{\Lip_{\mathrm{b}}(V)}: = \sup_{x\in Y} |\varphi(x)| +  \sup_{x,y \in Y, \, x \neq y} \frac{|\dnabla\varphi\,(x,y)|}{1{\wedge}d(x,y)}\,.
\end{equation}
 It is immediate to check  that the above norm in indeed equivalent to the usual norm for $\Lip_{\mathrm{b}}(Y)$, defined as 
 $ \sup_{x\in Y} |\varphi(x)| +  \sup_{x \neq y \in Y} \frac{|\dnabla\varphi\,(x,y)|}{d(x,y)}$.
 Then, we can define  on  $\calM(Y)$ the norm obtained by taking the duality against functions in $\Lip_{\mathrm{b}}(Y)$, i.e.\
\begin{equation}
\label{nbl-norm}
\nbl{\rho}: = \sup\left \{ \left| \int_V \varphi(x) \rho(\dd x ) \right| \, : \varphi \in \Lip_{\mathrm{b}}(Y),  \ \|\varphi\|_{\Lip_{\mathrm{b}}(Y)}\leq 1  \EEE  \right\}\,.
\end{equation}
Since $(Y,d)$ is separable, the narrow topology on $\calM^+(Y)$ is generated by the bounded Lipschitz metric, 
 as seen in  \cite[Theorem 8.3.2]{Bogachev07}. 
  \par
\medskip

\paragraph{\bf Convergence in $\Ms(Y)$.} 

As the setwise analog to vague convergence, i.e., convergence against  continuous functions with compact support, we give the following
\begin{definition}
We say that a sequence $(\mu_n)_{n\in \N}\subset \Ms(Y)$ converges \emph{$\sigma$-setwise} to some $\mu\in \Ms(Y)$ if, for some exhaustion $(A_k)_{k\in \N}$ corresponding to $\mu$, all restrictions $\mu_n {\mres} A_k$ are finite signed measures that converge setwise to $\mu {\mres} A_k$. 
\end{definition}
Note that by the properties of setwise convergence mentioned previously, this implies that $\sup_{n} |\mu_n|(A_k)< \infty$ for every $k\in \N$. 
\par
Since this notion of convergence is less standard, let us collect one additional result, also proved in Appendix \ref{app:preliminary}. 

\begin{lemma}\label{lm:sigmasetwise}
A sequence $(\mu_n)_{n\in \N}\subset \Ms(Y)$ that converges $\sigma$-setwise has a unique limit.

Moreover, consider any sequence $(\mu_n)_{n}\subset \Ms(Y)$ and a  function $f\in \Bb^+(Y)$ that is strictly positive $\mu_n$-a.e., and suppose that for some $\nu 
\in \calM(Y)$, we have
\[f \mu_n \to \nu \qquad \mbox{setwise in $\calM(Y)$}.\]
Then, $f$ is strictly positive $\nu$-a.e.\ in $Y$, and the sequence  $(\mu_n)_n$ converges $\sigma$-setwise to the measure  \[\mu(\dd x) :=\boldsymbol{1}_{f>0}(x) \frac1{f(x)} \nu(\dd x)\in \Ms(Y).\] 

In particular, if 
\[f \mu_n \to f \mu^* \qquad \mbox{setwise in $\calM(Y)$}.\] 
with $f\in \Bb^+(Y)$ strictly positive $\mu_n$, $\mu^*$-a.e.\ in $Y$ for some $\mu^*\in \Ms(Y)$, then $(\mu_n)_n$ converges $\sigma$-setwise to $\mu^*$.  

Viceversa, for any $\sigma$-setwise converging sequence $(\mu_n)_{n}\subset \Ms(Y)$ there \emph{exists} some $f\in \Bb^+(Y)$ that is strictly positive $\mu_n,\mu$-a.e.\ such that $(f \mu_n)_{n}\subset \calM(Y)$ converges setwise to $f \mu \in \calM(Y)$. 
\end{lemma}



\subsection{Convex functionals of measures}
\label{ss:prelims-cvx-funct}
We will work with functionals depending on pairs of $\sigma$-finite measures, defined in this way:
let 
$\Upsilon:\R^m\to [0,\pinfty]$ be proper,  convex and lower semicontinuous
and let us denote
by $\Upsilon^\infty:\R^m\to [0,\pinfty]$ its recession function
\begin{equation}
  \label{recession-upsi} \Upsilon^\infty(z):=\lim_{t\to\pinfty}\frac{\Upsilon(tz)}t=\sup_{t>0}\frac{\Upsilon(tz)-\Upsilon(0)}t\,.
\end{equation}
and with $\hat\Upsilon:\R^{m+1}\to[0,\infty]$ the perspective function 
\begin{equation}
\hat \Upsilon(z,t):=
      \begin{cases}
        \Upsilon(z/t)t&\text{if }t>0,\\
        \Upsilon^\infty(z)&\text{if }t=0,\\
        \pinfty&\text{if }t<0.
      \end{cases}
\end{equation}
We note that $\Upsilon(z,t)$ is 1-homogeneous and convex, with $\Upsilon^\infty$ being convex, lower semicontinuous, and positively $1$-homogeneous, with $\Upsilon^\infty(0)=0$. Hence, we define 
\begin{equation}
\label{def:F-F}
\begin{aligned}
&
\calF_\Upsilon:\Ms(Y;\overline{\R}^m) \times \Ms^+(Y)\mapsto [0,\pinfty], \qquad
\\
&
\calF_\Upsilon(\mu|\nu) :=
\int_Y \Upsilon \Bigl(\frac{\dd \mu}{\dd \nu}\Bigr)\,\dd\nu+
\int_Y \Upsilon^\infty\Bigl(\frac{\dd \mu^\perp}{\dd |\mu^\perp|}\Bigr) \,
\dd |\mu^\perp|,\qquad \text{for }\mu=\frac{\dd \mu}{\dd \nu}\nu+\mu^\perp.
\end{aligned}
\end{equation}
\par
In \cite[Lemma 2.3]{PRST22}, we proved some properties of functionals pertaining to this class, albeit defined on 
$\calM(Y;\R^m) \times \calM^+(Y)$; most of them extend to the present case, in which we consider functionals on $\Ms(Y;\overline{\R}^m) \times \Ms^+(Y)$. 
In Lemma \ref{l:crucial-F} below, we collect the properties 
  that will be used in what follows: in particular, let us highlight the lower semicontinuity of $\calF_\Upsilon$ with respect to \emph{$\sigma$-setwise} convergence in $\Ms$,  which is in the spirit of the lower semicontinuity results \`a la Reshetnyak from 
  \cite{AmFuPa05FBVF, ButtazzoBOOK89, Reshetnyak68, SlepcevWarren23}. 
\begin{lemma}
\label{l:crucial-F}
The functional $\calF_\Upsilon:\Ms(Y;\overline{\R}^m) \times \Ms^+(Y)\to [0,\pinfty]$ enjoys the following properties:
\begin{enumerate}
\item \RNEW if $\Upsilon(0)=0$, then for every $\mu \in \Ms(Y;\overline{\R}^m)$ and $\nu,\, \nu' \in \Ms^+(Y)$ there holds
\begin{equation}
\label{AC-monotonicity}
\nu \leq \nu' \ \Longrightarrow  \ \calF_\Upsilon(\mu|\nu') \leq  \calF_\Upsilon(\mu|\nu)\,; \EEE 
\end{equation} \EEE
 \item if $\Upsilon$ is superlinear, then
\begin{equation}
\label{superlinearity}
  \calF_\Upsilon(\mu|\nu)<\infty\quad\Longrightarrow\quad
  \mu\ll\nu,\quad
  \calF_\Upsilon(\mu|\nu)=
\int_Y \Upsilon \Bigl(\frac{\dd \mu}{\dd \nu}\Bigr)\,\dd\nu\,;
\end{equation}
\item if $ \Upsilon$ is positively $1$-homogeneous, then
    $\Upsilon\equiv \Upsilon^\infty$, $\calF_\Upsilon(\cdot|\nu)  =: \calF_\Upsilon$ is
    independent of $\nu$ and 
    \begin{equation}
      \label{eq:78}
       \calF_\Upsilon(\mu) \EEE =\int_Y
      \Upsilon\left(\frac{\dd\mu}{\dd\gamma}\right)
      \dd\gamma\quad
      \text{for every }\gamma\in \Ms^+(Y)\text{ such that } \mu\ll\gamma\,;
    \end{equation}
    \item Jensen's inequality: for every $B\in \frB(Y)$
    \begin{equation}\label{eq:jens}
     \hat \Upsilon(\mu^a(B)|\nu^a(B)) +\Upsilon^{\infty}(\mu^{\perp}(B))\leq  \calF_\Upsilon(\mu{\mres}B|\nu{\mres} B) \leq \calF_\Upsilon(\mu|\nu).
    \end{equation}
    \item  $\calF_\Upsilon$ is sequentially  lower
    semicontinuous in $\Ms(Y;\overline{\R}^m) \times \Ms^+(Y)$ with respect
    to $\sigma$-setwise convergence.
    \item Suppose that $\Upsilon$ is superlinear, that  the sequence $(\nu_n)_{n}\subset \calM^+(Y)$ converges setwise
    to some $\nu \in \calM^+(Y)$, and the sequence $(\mu_n)_{n}\subset \calM(Y)$ fulfills
    \begin{equation}
    \label{eq:crucsb}
    \lim_{n\to \infty} \calF_\Upsilon(\mu_n|\nu_n) < \infty. 
    \end{equation}
    Then $(\mu_n)_{n}$ is relatively compact w.r.t.\ setwise convergence. 
    
    In addition, for any sequence $(\mu_n)_{n}\subset \calM^+(Y)$ satisfying \eqref{eq:crucsb} and converging setwise to $\mu$, we have $\mu\ll \nu$.     
\end{enumerate}

\end{lemma}

\begin{proof}
For statements \textbf{(1)} up to \textbf{(4)} we refer to the proof of the analogous items in \cite[Lemma 2.3]{PRST22}. We first address here the proof of \textbf{(5)}. Thus, consider a sequence $(\mu_n)_{n\in \N}$ converging $\sigma$-setwise to some $\mu$ in $\Ms(Y;\overline{\R}^m)$ and a sequence $(\nu_n)_{n\in \N}$ converging $\sigma$-setwise to some $\nu$ in $\Ms^+(Y)$. Then we can find a common exhaustion $(A_k)_{k\in \N}$ such that both $\mu_n{\mres} A_k$ and $\nu_n {\mres} A_k$ are finite and converge setwise for every $k\in \N$. Therefore, due to nonnegativity and the results of \cite[Lemma 2.3]{PRST22}, we obtain
\begin{align*}
\liminf_{n\to \infty} \calF_\Upsilon(\mu^n|\nu^n)&\geq \liminf_{n\to \infty} \int_{A_k} \Upsilon \Bigl(\frac{\dd \mu^n}{\dd \nu^n}\Bigr)\,\dd\nu^n+\liminf_{n\to \infty} 
\int_{A_k} \Upsilon^\infty\Bigl(\frac{\dd (\mu^n)^\perp}{\dd |(\mu^n)^\perp|}\Bigr) \,
\dd |(\mu^n)^\perp| \\
&\geq \int_{A_k} \Upsilon \Bigl(\frac{\dd \mu}{\dd \nu}\Bigr)\,\dd\nu +
\int_{A_k} \Upsilon^\infty\Bigl(\frac{\dd \mu^\perp}{\dd |\mu^\perp|}\Bigr) \,
\dd |\mu^\perp|. 
\end{align*}
Taking the limit $k\to \infty$ via monotone convergence, we obtain the desired inequality. 

Now, for the case \textbf{(6)}, we will provide a modification of the proof in \cite{HoeksemaTh}. By assumption we have $\mu_n\ll \nu_n$, at least for  sufficiently big $n$.
Then,  after choosing an appropriate (not relabeled) subsequence, by Jensen's inequality \eqref{eq:jens}, we have
\begin{equation*}
\sup_{n\in \N} \hat \Upsilon(\mu_n(B)|\nu_n(B)) \leq \sup_{n \in \N}\calF_\Upsilon(\mu_n|\nu_n) =:M, \quad \mbox{for every $B\in \frB(Y)$}.
\end{equation*}
Since the sequence $(\nu_n)_n$ converges setwise to some $\nu$, there exists a  measure $\gamma_{\nu} \in \mathcal{M}^+(Y)$ such that \eqref{eq:73} holds. Now,
in order to show that  $(\mu_n)_n$ is itself relatively compact w.r.t.\ setwise convergence, we need to verify condition \eqref{eq:73}. For this, take an arbitrary $\eps>0$
and set 
\begin{align*}
    \eps_{\nu} &:= \frac{1}{2}\inf\{ b>0 \, : \ \exists\, a \in \R \text{ such that }  \hat \Upsilon(a|b)\leq M \text{ and }  |a|>  \eps\,\} \\
    &< \frac{1}{2}\inf\{ b>0 \, : \ \exists\, a \in \R \text{ such that }  \hat \Upsilon(a|b)\leq M \text{ and }  |a|>  \eps\,\}\,.
\end{align*}
Observe that, by the superlinearity of $\Upsilon$, $\eps_{\nu}$ is strictly positive. Therefore, by  \eqref{eq:73},  there exists some $\delta_{\nu}>0$ such that 
\[
 B\in \frB(Y),\ \gamma_{\nu}(B)\le \delta_\nu\ 
         \Longrightarrow\ 
         \sup_{n}|\nu_n|(B)\le \eps_\nu \text{ and, a fortiori, }  \sup_{n}|\mu_n|(B)\le \eps\,.
\]
Thus, \eqref{eq:73} is indeed valid for  $(\mu_n)_n$.

The final statement follows from the lower semicontinuity of $\calF_\Upsilon$.
\end{proof}

\paragraph{\bf Concave transformations of vector measures}
We will also need to extend  a construction, set forth in \cite[Sec.\ 2.3]{PRST22}
for vector and positive finite measures, to  the case of vector $\sigma$-finite signed measures, and $\sigma$-finite  positive measures. Namely,
 recall that   $\R_+:=[0,\pinfty[$, $\R^m_+:=(\R_+)^m$, and let
$\upalpha:\R^m_+\to\R_+$ be a continuous and concave function. It is
obvious that $\upalpha$ is non-decreasing with respect to each variable.
As in \eqref{recession-upsi}, the recession function $\upalpha^\infty$ is defined
by
\begin{equation}
  \label{eq:1}
  \upalpha^\infty(z):=\lim_{t\to\pinfty}\frac{\upalpha(tz)}t=\inf_{t>0}\frac{\upalpha(tz)-\upalpha(0)}t,\quad
  z\in \R^m_+.
\end{equation}
With slight abuse of notation, we will denote by $\Aalpha$ also the mapping   
\begin{equation}
  \label{alpha-mu-gamma}
  \begin{aligned}
  &
  \Aalpha:\Ms(Y;\overline{\R}^m_+)\times\Ms^+(Y)\to\Ms^+(Y), 
  \\
  &
  \Aalpha[\mu|\gamma]:=
  \upalpha\Bigl(\frac{\dd\mu}{\dd\gamma}\Bigr)\gamma+
  \upalpha^\infty\Bigl(\frac{\dd\mu}{\dd
    |\mu^\perp|}\Bigr)|\mu^\perp|\quad
  \mu\in \Ms(Y;\overline{\R}^m_+),\ \gamma\in \Ms^+(Y),
  \end{aligned}
\end{equation}
where  $\mu=\frac{\dd\mu}{\dd\gamma}\gamma+\mu^\perp$ is the
Lebesgue decomposition of $\mu$ with respect to ~$\gamma$.

\section{Setup for variational structures}
\label{s:3}
\noindent
We start by collecting our conditions 
\begin{itemize}
\item[-]
on the  vertex space $V$;
 \item[-] on the reference measure $\pi$, which will be invariant under the evolution
generated by our generalized gradient system;
\item[-]
 on the family of singular kernels $(\kappa(x,\cdot))_{x\in V} $, whose `singularity' is  encoded in the fact that 
for each $x\in V$ $\kappa(x,\cdot)$ is a $\sigma$-finite measure 
on $V\backslash\{x\}$. 
\end{itemize}

\EEE
 \paragraph{\bf Vertex space and kernels}
 Throughout the paper, 
\begin{enumerate}
\item the vertex space
\begin{equation*}
(V,\dV) \text{ is a separable metric space, with Borel $\sigma$-algebra $\mathfrak{B}(V)$;}
\end{equation*}
\item the reference measure
$\pi \in \calM^+(V)$ is a finite positive measure;
\item
for any $x\in V$,  $\kappa(x,\cdot)\in \Ms^+(V\backslash\{x\})$ 
and the kernels 
form a measurable family of measures, in the sense that, for all $f \in \Bb(V)$,
\begin{subequations}\label{ass:kappa}
\begin{equation}\label{true-measurability}
\mbox{the map $\displaystyle V\ni x \mapsto \int_{V\backslash\{x\}} f(y) (1{\wedge}\dV^2(x,y)) \, \kappa (x,\dd y)$ is measurable.}\tag{\textsf{A}$_\kappa$a}
 \end{equation}
 Furthermore, there holds
\begin{equation}
\label{mitigation of singularity}
\sup_{x\in V}\int_{V} (1{\wedge} \dV^2(x,y)) \,\kappa(x,\dd y) =: c_\kappa <+\infty.\tag{\textsf{A}$_\kappa$b}
\end{equation}
\end{subequations}
\end{enumerate}
It will be convenient, though, to have each $\kappa(x,\cdot)$ defined on the whole of $V$, and we will do so by setting
\[
 \kappa(x,\{x\}):= \int_{V} \boldsymbol{1}_{\{x\}}(y) \, \kappa(x, \dd y ) \doteq 0\,. 
\]
(the extended measure will be denoted by the same symbol). In this way, we  obtain a family $(\kappa(x,\cdot))_{x\in V} \subset \Ms^+(V)$ that, is indeed, thanks to 
\eqref{true-measurability}, a \emph{Borel} family of measures.
 Then, the kernels induce a
 measure 
 on the edge space $E=V{\times}V$ via 
 \begin{equation}
\label{teta-coupling}
    \tetapi(A{\times}B) = \int_{A} \kappa (x,B)\, \pi(\dd x)  \qquad \text{for all } A,\, B \in \mathfrak{B}(V)\,.
\end{equation}
Note that, due to \eqref{mitigation of singularity}, 
 $\tetapi$ is a  \emph{$\sigma$-finite} positive  measure on $E$. 
  We postpone to Appendix \ref{ss:appB} some further comments on  $\tetapi$
 and on equivalent formulations for the measurability condition \eqref{true-measurability}. 
 \par
 In addition to the previous assumptions,  \EEE
\begin{enumerate}
\setcounter{enumi}{3}
\item we require that 
the \emph{detailed balance condition} holds, i.e.,
\begin{equation}
\label{DBC}
s_\# \tetapi = \tetapi, \tag{\textsf{A}$_\text{DB}$}
\end{equation}
where 
$s: E \to E$, 
 $s(x,y):=(y,x)$ is the symmetry map and $s_\# \tetapi$ is the push-forward of $\tetapi$ under $s$. 
\end{enumerate}
\medskip

\par
Let $\Ed$ be the subset obtained removing from $E$ its diagonal, i.e.,
\[
    \Ed: = E \backslash \{ (x,x) : x \in V\}.
\]
In what follows, we will crucially use the fact that 
\begin{equation}
\label{crucial-facts-E'}
\tetapi(E\backslash \Ed)=0, \qquad \dV(x,y)>0 \quad \text{for all } (x,y) \in \Ed
\end{equation}
Let us also point out that 
\eqref{mitigation of singularity} implies that 
 \begin{equation}
 \label{tetapi-lambda}
 \tetapil (A) :=  \iint_A  (1{\wedge} \dV^2(x,y)) \EEE \, \tetapi (\dd x \dd y) \text{ for all } A \in \mathfrak{B}(E)
 \end{equation}
defines a (\emph{finite}) measure in $\calM^+(E)$.




 \begin{remark}
 \label{rmk:lambda}
 \sl
Let us dig deeper  on our conditions on the kernels  $(\kappa(x,\cdot))_{x\in V} $:
 \begin{enumerate}

\item
\EEE In the classical setting, where $V=\R^d$ with $\dV$ the Euclidean metric and $\pi=\Lebone^d$ (for the moment disregarding the finiteness of $\pi$), the general measurability and boundedness assumptions on the kernel ensure that $\kappa$ is a L\'evy-type kernel, with the associated conditions, semigroups, and processes treated in \cite[\S 3.5.1]{Applebaum_2009}. However, to frame the equation in terms of a gradient-flow formulation, sometimes the stronger requirement of narrow continuity and/or uniform integrability of $x\mapsto \dV^2(x,\cdot) \kappa(x,\cdot)$ is assumed, as in \cite{Erbar14, warren2025}, which is related to stability of the continuity equation and related terms with respect to narrow convergence. We would like to stress that, instead, we  do not assume any continuity of the kernel, due to the fact that we routinely work with setwise convergence instead of narrow convergence, with compactness provided by bounds on the entropy. 

Moreover, in both \cite{Erbar14, warren2025} the kernels are assumed to be (approximately) translation invariant, which allows for the use of convolution techniques to establish density in the sense of Assumption \eqref{Ass:F-bis}, and hence the chain rule. \EEE

\item
 While the analysis in \cite{PRST22} was carried out in the context of a \emph{standard Borel space} $V$, i.e., requiring $V$ to be a measurable space homeomorphic to a Borel subset of a complete and separable metric space, 
in the present setup we  rely on the fact that $V$ is endowed with a metric $\dV$ 
satisfying \eqref{mitigation of singularity}.
In turn, the existence of such a metric does reflect a restriction on the singularity of $\kappa$ and  ensures that  
nontrivial functions exist in the nonlocal Sobolev space, formed by bounded $\dV$-Lipschitz functions. 

This is  well illustrated by the following example: take $V=[0,1]$, $\kappa(x,\dd y):=|x-y|^{-3}$
(in $d$-dimensions, one could take $V$ as bounded domain in $\R^d$ and $\kappa(x,\dd y):=|x-y|^{-d-2}$), and assume that there exists a bounded, jointly measurable, and separable metric $\dV:E\to [0,+\infty)$ satisfying our assumptions. Fix any point $x^*$  and set $f(x):=1{\wedge} \dV(x^*,x)$. Then, 
\[ \int_E |\dnabla f|^2 \kappa(x,\dd y)\Lebone(\dd x) \leq\int_E \dV^2(x,y) \kappa(x,\dd y)\Lebone(\dd x) < \infty, \]
but, by \cite[Proposition 2]{Brezis2007}, this implies that $f=C$ and hence $\dV(x^*,\cdot)=C$ a.e.\ for some constant $C$, which contradicts  the validity of \eqref{mitigation of singularity} at the point $x^*$. 
\end{enumerate}
 \end{remark}


\medskip

 \paragraph{\bf Energy and dissipation.}
Our variational structure  will also be  based on a driving entropy functional, on a dissipation potential, and on a flux density map whose properties are collected below.
\begin{assumption}[Entropy]
\label{Ass:E} 
\sl
The entropy functional
$\scrE: \mathcal{M}^+(V) \to [0,+\infty]$  is of the form
\[
\scrE(\rho) = \begin{cases}
\displaystyle \int_V \upphi\left(\frac{\dd\rho}{\dd \pi} \right)\, \dd \pi & \text{if } \rho \ll \pi,
\\
+\infty  & \text{otherwise.}
\end{cases}
\]
We require the entropy density to fulfill
\begin{equation}
\label{ass-phi}
\begin{gathered}
\upphi \in \mathrm{C}([0,+\infty] ) \cap \mathrm{C}^1(]0,+\infty] ), \  \text{$\upphi$  convex with  $\min \upphi =0$,}
\\
\lim_{r\to+\infty}  \frac{\upphi(r)}{r} =+\infty.
\end{gathered}
\end{equation}
\end{assumption}
\begin{assumption}[Dissipation]
\label{Ass:D}
\sl
The dual dissipation density $\uppsi^*:\R \to [0,+\infty)$ is convex, differentiable, even, with $\uppsi^*(0)=0$ and 
\begin{align}
\label{growth-psistar}
&
 \lim_{|\xi|\to+\infty} \frac{\uppsi^*(\xi)}{|\xi|} =+\infty,
\\
&
 \label{quadratic-at-0}
 \lim_{\xi \to 0} \frac{\uppsi^*(\xi)}{|\xi|^2} =c_0 \in (0,+\infty)\,. \EEE
 \end{align}
 \end{assumption}
\begin{assumption}[Flux density]
\label{Ass:flux-density}
The mapping $\upalpha: [0,+\infty)\times [0,+\infty) \to [0,+\infty)$ is non-degenerate (not identically null), continuous, concave, symmetric, with 
\begin{equation}
\label{props-alpha}
\upalpha(u_1,u_2) = \upalpha(u_2,u_1) \qquad\forall\, (u_1,u_2) \in  [0,+\infty)\times [0,+\infty)\,.
\end{equation}
\end{assumption}
The conditions on $\upalpha$ ensure that there exists a constant $c_{\upalpha}>0$ providing the bound
\begin{equation}
     \label{eq:alphab}
 \upalpha(u_1,u_2)\leq c_{\upalpha}(1{+}u_2{+}u_2).
\end{equation}

Before moving on, let us pin down two key growth properties that will be extensively used ahead. 

\begin{lemma}
\label{l:props-uppsistar}
The function $\uppsi^*$ is non-decreasing on $[0,\infty)$.
Moreover,
\begin{enumerate}
\item
 there exists a 
conjugate pair $(\psih,\psih^*)$, with $\psih, \, \psih^* : \R \to [0,+\infty)$
convex, even, and superlinear, 
such that  \EEE
\begin{equation}
\label{needed-control-bound}
\forall\, \zeta \in \R \ \  \forall\, B>0 \ \ \forall\, \beta \in [{-}B, B] \,: \qquad \uppsi^*(\beta \zeta) \leq \frac1{B^2} \beta^2 \EEE \psih^*(B \zeta )\,;
\end{equation}
 as a consequence, 
\begin{equation}
\label{corollary-needed-control-bound}
\forall\, w \in \R \ \  \forall\, \delta \in (0,1]  \,: \qquad 
\psih(w) \leq \frac1{\delta^2} \uppsi(\delta w)\,;
\end{equation}
\item \EEE  $\uppsi^*$ enjoys a quadratic bound from below:
\begin{equation}
 \forall\, M>0 \ \ \exists\, K_M>0   \ \  \forall\, \xi \in [{-}M,M]\,:  \quad  |\xi|^2 \leq K_M \uppsi^*(\xi)\,. 
\label{Olli-est-2}
\end{equation}
\end{enumerate}
\end{lemma}
\begin{proof}
Since $0=\uppsi^*(0) = \min_{\R} \uppsi^*$, we have $(\uppsi^*)'(0)=0$ and thus $(\uppsi^*)'(\xi)\geq 0$ for all $\xi \geq 0$. Hence, the monotonicity  statement for 
$\uppsi^*$. In what follows, we will crucially use that, by
\eqref{quadratic-at-0}, 
\begin{equation}
\label{easy-conseq}
\exists\, r>0 \ \ \forall\, x \in  [{-}r,r]\,:  \   \frac12 c_0 |x|^2 \leq \uppsi^*(x)\leq \frac32 c_0 |x|^2\,.
\end{equation}
\smallskip

\noindent
$\vartriangleright \eqref{needed-control-bound}$
We will in fact show
\begin{equation}
\label{needed-control-normal}
\uppsi^*(\eta\zeta) \leq \eta^2 \psih^*(\zeta) \qquad \text{for all } \eta \in [-1,1], \ \zeta \in \R\,,
\end{equation}
which, of course, yields \eqref{needed-control-bound}. 
We distinguish two cases: 
\begin{enumerate}
\item $|\eta \zeta| \leq r$: then 
\begin{equation}
\label{est4needed-1}
\uppsi^*(\eta \zeta) \leq  \frac32 c_0 \eta^2 \zeta^2\,. 
\end{equation}
\item $|\eta \zeta| > r$: then, 
\begin{equation}
\label{est4needed-2}
\uppsi^*(\eta \zeta) \overset{(1)}{=}\uppsi^*(|\eta \zeta|)\overset{(2)}{\leq}\uppsi^*( \zeta)  \overset{(3)}{\leq}\frac{|\eta\zeta|}{r}\uppsi^*( \zeta)    \overset{(4)}{\leq}  \frac{|\zeta|^2}{r^2} |\eta|^2 \uppsi^*( \zeta)  
\,,
\end{equation}
where {\footnotesize (1)} follows from the evenness of $\uppsi^*$, {\footnotesize (2)} from its monotonicity on $[0,+\infty)$
 and the fact that $|\eta|\leq 1$,
  and  {\footnotesize (3), (4)}  from the fact that $\tfrac{|\eta \zeta|} r>1$. 
\end{enumerate}
All in all, we conclude that 
\[
\uppsi^*(\eta \zeta) \leq  \eta^2 \left( \frac32 c_0  \zeta^2 {+}  \frac{|\zeta|^2}{r^2}  \uppsi^*( \zeta)  \right) \doteq  \eta^2{\mathfrak{g}(\zeta)} \,.
\]
 Now, the function $\mathfrak{g}  : \R \to [0,+\infty)$ is convex, even, and with superlinear growth at infinity.  It is immediate to check that $\psih: = \mathfrak{g}^* $ enjoys the same properties,  and the above estimate holds with $\psih^*=  \mathfrak{g}$, yielding  \eqref{needed-control-normal}.
\par
Then, \eqref{corollary-needed-control-bound} follows from observing that, for all
$ B>0$ and for all $ \beta \in [{-}B, B] $
\[
\begin{aligned}
\psih(w) = \sup_{\zeta \in \R} (w B\zeta - \psih^*(B\zeta))  & 
=
\left( \frac B\beta \right)^2 \sup_{\zeta \in \R} \left(\frac{\beta^2}B w \zeta - \left( \frac\beta B\right)^2\psih^*(B\zeta)\right)
\\
& \stackrel{\eqref{needed-control-bound}}\leq 
\left( \frac B\beta \right)^2 \sup_{\zeta \in \R} \left(\frac{\beta^2}B w \zeta - \uppsi^*(\beta\zeta)\right)
= \left( \frac B\beta \right)^2 \uppsi \left( w \frac\beta B\right) \,.
\end{aligned}
\]
\smallskip

\noindent
\EEE  $\vartriangleright  \eqref{Olli-est-2}:$ let $r>0$
be as in  \eqref{easy-conseq} and $\xi \in[{-}M,M] $. If $|\xi| \leq r$, then $
|\xi|^2 \leq \tfrac2{c_0}  \uppsi^*(\xi)\,.
$
Otherwise, if $|\xi|> r$, then $  \uppsi^*(\xi)\geq  \uppsi^*(r)$ by monotonicity and evenness. All in all, we conclude that 
\[
\forall\, \xi \in [{-}M,M]\,: \qquad |\xi|^2 \leq  \frac2{c_0}  \uppsi^*(\xi) + M^2 \frac{\uppsi^*(\xi)}{\uppsi^*(r)}  \doteq  K_M \uppsi^*(\xi)\,. \qedhere
\] 
\end{proof}

\subsection{The continuity equation}
Hereafter, for a given function $\mu :I \to \calM(V)$, or $\mu : I  \to \calM(\edg)$, with $I=[a,b]\subset\R$, we shall often write $\mu_t$ in place of $\mu(t)$ for a given $t\in I$ and denote the time-dependent function $\mu $ by  $(\mu_t)_{t\in I}$. For the weakest formulation of continuity equation considered in this paper, we will choose test functions in the space
 \begin{subequations}
 \label{Lip-bounded}
\begin{align}
&
\label{Lip-bounded-a}
\Lip_{\mathrm{b}}(V) := \Lip(V) \cap \mathrm{B}_{\mathrm{b}}(V),
\intertext{which we consider with the norm}
&
\label{Lip-bounded-b2}
\RNEW \| \varphi \|_{\Lip_{\mathrm{b}}(V)}: = \sup_{x\in V} |\varphi(x)| +  \sup_{x,y \in V, \, x \neq y} \frac{|\dnabla\varphi\,(x,y)|}{1{\wedge}\dV(x,y)}\,,
\end{align}
\end{subequations}
where we recall the notation 
\begin{equation}
\label{eq:def:ona-grad}
(\ona \varphi)(x,y) := \varphi(y)-\varphi(x) \qquad \text{for any Borel function }\varphi \in \B(V)\,.
\end{equation}
 As previously mentioned, the above norm in indeed equivalent to the usual norm for $\Lip_{\mathrm{b}}(V)$, defined as 
 $ \sup_{x\in V} |\varphi(x)| +  \sup_{x \neq y \in V} \frac{|\dnabla\varphi\,(x,y)|}{\dV(x,y)}$.
 In fact, we have opted for the definition in 
 \eqref{Lip-bounded-b2} only to have more transparent estimates, cf.\  Remark \ref{rmk:lip-test-functions} ahead.
 
\medskip
We are now in a position to introduce our first formulation for  the continuity equation.

\begin{definition}[Solutions to the continuity equation]
\label{def:CE}
Let $[0,T] \subset \R$.
We denote by $\CE 0T$  the set of pairs $(\rho,\jj)$ such that
\begin{itemize}
    \item[-] $\rho= (\rho_t)_{t\in [0,T]} $ is a family of time-dependent measures in $\mathcal{M}^+(V)$;
    \item[-] $\jj = (\jj_t)_{t\in [0,T]}$ is a measurable family of measures in  $\Ms(E)$ (i.e.\ with $(\jj_t^{\pm})_{t\in [0,T]}$ both Borel families) such that $\jj_t(E\setminus \Ed)=0$ for all $t\in [0,T]$, and 
    \begin{equation}
    \label{crucial-bound-j}
\int_0^T \iint_E  (1{\wedge}\dV(x,y)) \, |\jj_t|(\dd x \dd y)\,\dd t<+\infty.
    \end{equation}
    \item[-] the continuity equation holds in this sense: For all $\varphi \in \Lip_{\mathrm{b}}(V) \text{ and all } [s,t]\subset [0,T]$,
    \begin{equation}
    \label{CE}
   \begin{aligned}
		\int_V \varphi(x)\, \rho_t(\dd x)  - \int_V \varphi(x)\, \rho_s(\dd x)   = \int_s^t \iint_\Ed \dnabla\varphi\,(x,y)\,\jj_r(\dd x\dd y)\,\dd r.
  \end{aligned}
	\end{equation}
\end{itemize}	
In what follows, we will use the  short-hand notation $\bj_{{\!\Lebone}}$ for the measure on $[0,T]\times E$ given by 
$\bj_{{\!\Lebone}}(A): = \iiint_{A} \bj_t(\dd x \dd y) \, \dd t$ for all $A\in \frB([0,T]{\times} E)$. 
\end{definition}
\begin{remark}
\label{rmk:lip-test-functions}
\sl
By our choice of test functions, the integral on the right-hand side of \eqref{CE} is well defined. Indeed, for every $[s,t]\subset [0,T]$ we have 
\begin{equation}
\label{obvious-estimate}
\left| \int_s^t \iint \dnabla\varphi\,(x,y)\,\jj_r(\dd x\dd y)\,\dd r \right| \leq \| \varphi \|_{\Lip_{\mathrm{b}}(V)} \int_s^t 
 \iint_\Ed  (1{\wedge}\dV(x,y))   \, |\jj_t|(\dd x \dd y)\,\dd t<+\infty
\end{equation}
thanks to \eqref{crucial-bound-j}. Combining this with \eqref{CE} we obtain that
(recall the definition of $\|{\cdot}\|_{\BL}$ from 
\eqref{nbl-norm})
\begin{equation}
\label{j-controls-rho}
\| \rho(t){-}\rho(s)\|_{\BL} \leq \int_s^t 
 \iint_\Ed  (1{\wedge}\dV(x,y))   \, |\jj_r|(\dd x \dd y)\,\dd r\,,
\end{equation}
whence we deduce that the curve $[0,T]\ni t \mapsto \rho_t$ is absolutely continuous with values in $(\mathcal{M}^+(V), \| {\cdot}\|_{\BL})$. 
 In turn, since  the $\|{\cdot}\|_{\BL}$-norm topologizes narrow convergence, a fortiori 
we can conclude that the curve $[0,T]\ni t \mapsto \rho_t$ is narrowly continuous.
\end{remark} 

\par
It can easily checked that the concatenation of two solutions to the continuity equation, in the above sense, is again a solution to the continuity equation, and that the family of solutions is closed under time rescaling. The proof of the following results trivially adapts the argument for \cite[Lemma 8.1.3]{AGS08} and is thus omitted.

\begin{lemma}[Concatenation and time rescaling]
\label{l:concatenation&rescaling}
\begin{enumerate}
\item Let $(\rho^i,\bj^i) \in \CE 0{T_i}$, $i=1,2$, with $\rho_{T_1}^1 = \rho_0^2$. Define $(\rho_t,\bj_t)_{t\in [0,T_{1}+T_2]}$ by 
\[
\rho_t: = \begin{cases}
\rho_t^1 & \text{ if } t \in [0,T_1],
\\
\rho_{t-T_1}^2 & \text{ if } t \in [T_1,T_1+T_2],
\end{cases}
\qquad \qquad 
\bj_t: = \begin{cases}
\bj_t^1 & \text{ if } t \in [0,T_1],
\\
\bj_{t-T_1}^2 & \text{ if } t \in [T_1,T_1+T_2]\,.
\end{cases}
\]
Then, $(\rho,\bj ) \in \CE 0{T_1+T_2}$.
\item
Let $\mathsf{t} : [0,\hat{T}] \to [0,T]$ be strictly increasing and absolutely continuous, with inverse $\mathsf{s}: [0,T]\to [0,\hat{T}]$. Then, $(\rho, \bj ) \in \CE 0T$ if and only if 
$\hat \rho: = \rho \circ \mathsf{t}$ and $\hat \bj : = \mathsf{t}' (\bj  {\circ} \mathsf{t})$  fulfill $(\hat \rho, \hat \bj ) \in \CE 0{\hat T}$. 
\end{enumerate}
\end{lemma}

\subsection{The primal and dual dissipation potentials}\label{ss:dissd}
Based on Assumptions   \ref{Ass:D},  and 
\ref{Ass:flux-density},  we rigorously define the primal 
dissipation
 and dual dissipation
 potentials
$
\scrR$ and $\scrR^*$ formally introduced in \eqref{formal_definition}. 
 As in \cite{PRST22}, first of all with any $\rho \in \calM^+(V)$ we associate the 
 measures
 $ \teta_{\!\rho}^\pm$ defined 
  on $E$ by 
\begin{equation}
\label{rig-def:tetarho}
\begin{aligned}
  \teta_{\!\rho}^-(\dd x\,\dd y) :=
  \rho(\dd x)\kappa(x,\dd y),\qquad
  \teta_{\!\rho}^+(\dd x\,\dd y) :=
  \rho(\dd y)\kappa(y,\dd x)=
  s_{\#}\teta_{\!\rho}^-(\dd x\,\dd y),
\end{aligned}
\end{equation}
Observe that $\teta_{\!\rho}^\pm \in  \Ms(E)$; 
we may think of them as `$\rho$-adjusted jump rates'. 
 Let us also observe that, since we have set $\kappa(x,\{x\})=0$ for every $x\in V$,
\begin{equation}\label{trivial-but-needed}
\teta_{\!\rho}^-(E{\setminus}\Ed) = \iint_{\{(x,x):\,  x\in V\}} \kappa(x,\dd y)\rho(\dd x) =0 
\end{equation}
and we clearly have $\teta_{\!\rho}^+(E{\setminus}\Ed)=0$, too.  
The following result extends 
\cite[Lemma 2.4]{PRST22}.
\begin{lemma}
\label{l:3.4}
Let $(\rho_n)_n,\rho \in \calM^+(V)$ satisfy $\rho_n\to \rho$ setwise as $n\to\infty$. Then, 
\begin{equation}
\label{1st-assertion}
(1{\wedge} \dV^2) \teta_{\!\rho_n}^\pm \to (1{\wedge} \dV^2) \teta_{\!\rho}^\pm
\qquad \text{setwise in $\calM^+(E)$.}
\end{equation}
In particular, 
\begin{equation}
\label{crucial-conseq-for-tetapm}
\teta_{\!\rho_n}^\pm \to \teta_{\!\rho}^{\pm} \qquad  \text{$\sigma$-setwise in $\Ms^+(E)$.}
\end{equation}
\end{lemma}

\begin{proof}
$\vartriangleright \eqref{1st-assertion}:$
Clearly, it is sufficient to prove \eqref{1st-assertion} for, e.g.,  $(\teta_{\!\rho_n}^-)_n$. For any $\phi\in \Bb(E)$, then the map $x\mapsto \int_V \phi(x,y)(1{\wedge} \dV^2(x,y))\kappa(x,\dd y)$ is
a Bounded borel mapping
\EEE by \eqref{true-measurability} and \eqref{mitigation of singularity}. \EEE
 Therefore, since $\rho_n\to \rho$ setwise, we have 
\[
\begin{aligned}
&
\iint_{E} \phi(x,y) \, (1{\wedge} \dV^2(x,y)) \teta_{\!\rho_n}^- (\dd x \dd y) = \int_{V} \left(  \int_{V} \phi(x,y) \,(1{\wedge} \dV^2(x,y))\kappa(x,\dd y) \right) \, \rho_n(\dd x) 
\\
&
\longrightarrow  \int_{V} \left(  \int_{V} \phi(x,y) \,(1{\wedge} \dV^2(x,y))\kappa(x,\dd y) \right) \, \rho(\dd x)  = \iint_{E} \phi(x,y) \, (1{\wedge} \dV^2(x,y))\teta_{\!\rho}^- (\dd x \dd y) \,.
\end{aligned}
\]
$\vartriangleright \eqref{crucial-conseq-for-tetapm}:$
\EEE We now apply Lemma \ref{lm:sigmasetwise}
 with the choice 
$f(x,y) = (1{\wedge} \dV^2(x,y))$, noting that $f\in \Bb^+(E)$ and that, indeed, 
$f>0$
$\teta_{\!\rho_n}^\pm$, $\teta_{\!\rho}^\pm$-a.e.-in $E$, cf.\ \eqref{trivial-but-needed}.  \EEE
\end{proof}
\par
Based on the construction of $\teta_{\!\rho}^\pm$, 
 we may rigorously define the measure $\bnu_{\!\rho} \in  \Ms(E)$ from \eqref{formal_definition}
\RNEW  as concave transformation of the couplings $\teta_{\!\rho}^\pm$
  (cf.\ construction set forth in \eqref{alpha-mu-gamma}), \EEE namely
\begin{equation}
\label{rig-nu-rho}
\bnu_\rho: =  \Aalpha[\teta_{\!\rho}^-,\teta_{\!\rho}^+|\tetapi],
\end{equation}
where we have used the simplified notation $\Aalpha[\teta_{\!\rho}^-,\teta_{\!\rho}^+|\tetapi]$ in place of 
$\Aalpha[(\teta_{\!\rho}^-,\teta_{\!\rho}^+)|\tetapi]$. 
In the definition of $\scrR$ we will also resort to the construction from \eqref{def:F-F}.
\begin{definition}
\label{def:primal-and-dual}
We define $\scrR: \calM^+(V)\times \Ms(E) \to [0,+\infty]$ and $\scrR^*: \calM^+(V)\times \B(E) \to [0,+\infty]$ via  
\begin{align}
&
\label{rigR}
\scrR(\rho,\bj): =  \frac12 \calF_\uppsi(2\bj |\bnu_\rho)\,,
\\
&
\label{Rig-Rstar}
\scrR^*(\rho,\xi): = \frac12 \iint_{\Ed} \uppsi^*(\xi)  \,\bnu_\rho(\dd x \dd y)\,,
\end{align}
\end{definition}
The following result collects some key facts about $\scrR$ and $\scrR^*$. Firstly, we  provide an equivalent representation, cf.\  
\eqref{equivalent-upsilon} below,
for 
$\scrR$ involving the convex and lower semicontinuous function
$\Upsilon: [0,+\infty) \times [0,+\infty) \times \R \to [0,+\infty]$ given by 
\[\Upsilon(u,v,w): = \hat\uppsi (w,\upalpha(u,v)),\]
where $\hat\uppsi:\R^{2}\to[0,\infty]$ is the 1-homogeneous, convex, perspective function associated  with $\uppsi$ defined in \eqref{def:F-F}.

Based on \eqref{equivalent-upsilon} we may prove \eqref{R-lsc}, which extends the sequential lower semicontinuity of $\scrR(\rho, \cdot)$, 
proved in \cite[Lemma 4.10]{PRST22}
for setwise  convergence, to the case of 
 $\sigma$-setwise   convergence. 
\begin{lemma}
\label{lm:proprr}
The following properties hold:
\begin{enumerate}
\item
For $\scrR$ we have the equivalent representation
\begin{equation}
\label{equivalent-upsilon}
\scrR(\rho,\bj) = \frac12 \calF_\Upsilon ((\teta_{\!\rho}^-,\teta_{\!\rho}^+,2 \bj)|\tetapi);
\end{equation}
\item if $(\rho,\bj)$ fulfill $\scrR(\rho,\bj)<+\infty$ and $\rho \ll \pi$, 
 then
 \begin{itemize}
 \item[-]
 $\bj \ll \tetapi$ with $w=\frac{\dd( 2\bj)}{\dd \tetapi}$, 
\item[-] setting 
\[
 E_\upalpha: = \{ (x,y)\in \Ed\, : \ \upalpha(u(x),u(y))>0 \} \qquad \text{we have } |\bj |(\Ed {\backslash} E_\upalpha)=0\,;
\]
\item[-] $\scrR$ rewrites as 
\begin{equation}
\label{nice-representation}
\scrR(\rho,\bj)  = 
    \frac12\iint_{E_\upalpha}
    \Psi\Bigl(\frac{w(x,y)}{\upalpha(u(x),u(y))}\Bigr)\upalpha(u(x),u(y))\,\tetapi(\dd
    x,\dd y)\,.
\end{equation}
\end{itemize}

 \item For all  $(\rho_n)\, \rho \in \calM^+(V)$ and $(\bj_n)_n,\, \bj \in \Ms(E) $ there holds
 \begin{equation}
 \label{R-lsc}
\left[ \rho_n \to \rho \text{ setwise  and } \bj_n \to \bj \text{ $\sigma$-setwise} \right] \Longrightarrow \liminf_{n\to\infty} \scrR(\rho_n,\bj_n) \geq \scrR(\rho,\bj)\,.
 \end{equation}
 \item For all $\rho\in \calM^+(V)$, 
$\beta\neq 0$, and $\xi\in \Bb(E)$ with $\|\xi\|_{\infty}\leq M$,
\begin{equation}
\label{eq:disspb}
\begin{aligned}
\scrR^*(\rho,\beta \xi) &\leq \frac{c_{\upalpha}}{2M^2} \psih^*(\beta M) \iint_{\Ed} \xi^2(x,y)  (\teta_{\rho}^+{+}\teta_{\rho}^-{+}\tetapi)(\dd x 
\dd y )\\
&\leq \frac{c_{\kappa}\psih^*(\beta M)c_{\upalpha}(2\rho(V)+\pi(V))}{2M^2}\left\|\frac{\xi}{1{\wedge}\dV}\right\|_{\infty}^2
\end{aligned}
\end{equation}
with $\psih$ from Lemma \ref{l:props-uppsistar}.
\item Finally, for any $\rho\in \calM^+(V)$, $\bj \in \Ms(E)$,
\begin{equation}
\label{eq:disspbd}
\begin{aligned}
\scrR(\rho,\bj)&\geq  \frac12 \calF_\psih\left(2(1{\wedge}\dV)\bj\,\Big|\,c_{\upalpha}(1{\wedge}\dV)^2 (\teta_{\rho}^+{+}\teta_{\rho}^-{+}\tetapi\right) \\
&\geq \frac12  \hat \psih\bigl(2\|(1{\wedge}\dV)\bj\|_{\mathrm{TV}}\,,\, c_{\upalpha}c_{\kappa}(2\rho(V){+}\pi(V))\bigr).
\end{aligned}
\end{equation}
with $\hat\psih$ the perspective function of $\psih$
 \end{enumerate}
\end{lemma}
\begin{proof}
Formula \eqref{equivalent-upsilon} follows by the same arguments as in \cite[Sec.\ 4.2]{PRST22}; likewise, the proof of the first part of Claim \textbf{(2)}
follows by   \cite[Lemma 4.10]{PRST22}. 
\par
Now, 
when $\rho \ll \pi$ then we also have $\teta_{\!\rho}^\pm \ll \tetapi$ and (recalling that $\teta_{\!\rho}^+= s_{\#}\teta_{\!\rho}^-$)
\[
\frac{\dd \teta_{\!\rho}^-}{\dd \tetapi}(x,y) = u(x), \qquad \frac{\dd \teta_{\!\rho}^+}{\dd \tetapi}(x,y) = u(y),
\]
so that 
\[
\bnu_\rho(\dd x \, \dd y) = \upalpha \left(\frac{\dd \teta_{\!\rho}^-}{\dd \tetapi},\frac{\dd \teta_{\!\rho}^+}{\dd \tetapi} \right) \,\tetapi(\dd x \dd y) = \upalpha(u(x),u(y)) \,  \tetapi(\dd x \dd y) \,.
\]
Since
 $\bj\ll \tetapi$, we can write $2\jj = w \tetapi$, so that 
 \[
\Upsilon \left(\frac{\dd \teta_{\!\rho}^-}{\dd \tetapi} ,\frac{\dd \teta_{\!\rho}^+}{\dd \tetapi},\frac{\dd 2\jj}{\dd \tetapi} \right)(x,y) =\hat{\uppsi}(w(x,y),\upalpha( u(x),u(y)))
\] 
  and \eqref{nice-representation} follows by the definition of the perspective function $\hat\uppsi$, 
 cf.\ also \cite[Sec.\ 4.2]{PRST22}.
    \par
  Next, let $(\rho_n)\, \rho \in \calM^+(V)$ and $(\bj_n)_n,\, \bj \in \Ms(E) $ be as in \eqref{R-lsc}. By Lemma \ref{l:3.4} we infer that
  $(\teta_{\!\rho_n}^-,\teta_{\!\rho_n}^+,2\bj_n) \to (\teta_{\!\rho}^-,\teta_{\!\rho}^+,2\bj)$ 
 $\sigma$-setwise 
  in $\Ms(E;\R^3)$, and the assertion follows 
  from Lemma \ref{l:crucial-F}.

  We turn to the bounds 
  \eqref{eq:disspb}
  on $\scrR^*(\rho,\beta \xi)$. Let $\xi\in \Bb(E)$ with $\|\xi\|_{\infty}\leq M$, then due to the property \eqref{needed-control-bound} we obtain the estimate 
  \[
   \scrR^*(\rho,\beta \xi) \leq \frac12 \iint_{\Ed} \uppsi^*(\beta \xi(x,y)  \,\bnu_\rho(\dd x \dd y)
  \leq \frac12 \iint_{\Ed} 
 \frac{\psih^*(\beta M) \xi^2(x,y)}{M^2}\,  \bnu_\rho (\dd x \dd y).
 \]
  The first bound in \eqref{eq:disspbd}
   now follows from \eqref{eq:alphab}, which implies 
  \[
   \bnu_\rho \leq c_{\upalpha}(\teta_{\rho}^+{+}\teta_{\rho}^-{+}\tetapi)\,.
   \]
 The second inequality follows from estimates  $\|(1{\wedge} \dV^2) \teta_{\!\rho}^\pm\|_{\mathrm{TV}}\leq c_{\kappa} \rho(V)$, $\|(1{\wedge} \dV^2) \tetapi\|_{\mathrm{TV}}\leq c_{\kappa} \pi(V)$. 
Finally, for the dual estimate \eqref{eq:disspbd}, we use the  property \eqref{corollary-needed-control-bound} and monotonicity of $\psih$, ensuring that for any $z\in (0,1]$
\begin{align*}
\hat{\Psi}\left(w,\upalpha(u,v)\right)   = \frac{1}{z^2}\hat{\Psi}\left(z^2 w,z^2 \upalpha(u,v)\right)
    &\geq z^2 \upalpha(u,v) \psih \left(\frac{zw}{z^2 \upalpha(u,v)} \right) \\
    &= \hat{\psih}(z w,z^2 \upalpha(u,v)) \geq \hat{\psih}(z w,c_{\upalpha}z^2(1{+}u{+}v)),
\end{align*}
and therefore we can derive the estimate
\begin{equation*}
\begin{aligned}
\calF_\uppsi(2\bj |\bnu_\rho)&=\iint_{E_\upalpha} \frac{1}{(1{\wedge} \dV)^2}\Psi\Bigl((1{\wedge} \dV)\frac{(1{\wedge} \dV)w(x,y)}{(1{\wedge} \dV^2)\upalpha(u(x),u(y))}\Bigr)(1{\wedge} \dV)^2\upalpha(u(x),u(y))\,\tetapi(\dd
    x\dd y) \\
&\geq \iint_{E_\upalpha} \psih\Bigl(\frac{(1{\wedge} \dV)w(x,y)}{c_{\upalpha}(1{\wedge} \dV^2)(1{+}u(x){+}u(y))}\Bigr)c_{\upalpha}(1{\wedge} \dV^2)(1{+}u(x){+}u(y))\,\tetapi(\dd  x\dd y),
\end{aligned}
\end{equation*}
from which the desired first bound follows, since $w \equiv 0$ on  the complementary $E_\upalpha$. The final bound in \eqref{eq:disspbd} is obtained from the fact that $\psih$ is monotone and $\psih(w)=\psih(|w|)$ for all $w\in \R$ by evenness of $\psih$, and an application of Jensen's inequality \eqref{eq:jens}.
\end{proof}
\paragraph{\bf The Fisher information}
Ultimately, before specifying our notions of solution we need to properly introduce the Fisher information functional. Formally, it is given by 
\[
\Fish(\rho) = \scrR^*\bigl(\rho,-\thalf \ona \upphi(u)\bigr)
=
 \frac12
\iint_\edg \Psi^*\bigl( -\thalf(\upphi'(u(y)){-}\upphi'(u(x))\bigr) \bnu_\rho(\dd x \, \dd y),\qquad \rho = u\pi\, .
\]
However, note that $\upphi$ need not be differentiable at $0$. Hence,  in order to
 rigorously define $\Fish$ in the present context, we mimic the ideas of \cite{PRST22}. Firstly, we introduce the function
$\Lambda_\upphi: \R_+\times \R_+ \to [-\infty,+\infty]$
\begin{equation}
\label{Lambda_upphi}
\Lambda_\upphi (u,v): =\begin{cases}
\upphi'(v){-}\upphi'(u) & \text{if } u,v \in \R_+\times \R_+ \backslash \{ (0,0)\},
\\
0 & \text{if } u=v=0\,,
\end{cases}
\end{equation}
where we have set $\upphi'(0) := \lim_{r\downarrow 0} \upphi'(r) \in [-\infty,+\infty) $. 
Hence, we define the function $\mathrm{D}_\upphi^+: \R_+\times \R_+ \to [0,+\infty]$ by
\[
\mathrm{D}_\upphi^+(u,v): = \begin{cases}
\uppsi^*\left( \Lambda_\upphi (u,v) \right) \upalpha (u,v) & \text{if } \upalpha (u,v)>0,
\\
0 
& \text{if } u=v=0,
\\
+\infty & \text{if } \upalpha(u,v)=0 \text{ with } u \neq v\,.
\end{cases}
\]
Finally, we consider
\begin{equation}
\label{dupphi}
\text{the \emph{lower semicontinous envelope} } \mathrm{D}_\upphi: \R_+\times \R_+ \to [0,+\infty]  \text{ of }   \mathrm{D}_\upphi^+\,.
\end{equation}
\begin{definition}[Fisher information]
\label{Def:Fisher}
The Fisher information $\Fish: \mathrm{dom}(\scrE)\to [0,+\infty]$ is defined as
\[
\Fish(\rho): = 
\begin{cases}\displaystyle
\frac12 \iint_{E}  \mathrm{D}_\upphi(u(x),u(y)) \, \tetapi(\dd x \dd y)  &  \text{if  } \rho = u \pi, 
\\
+\infty &\text{ otherwise.}
\end{cases} 
\]
\end{definition}

 One of the key ingredients of for the theory in \cite{PRST22} is the property that  \EEE
\begin{equation}
\label{FishLSC}
\tag{$\mathrm{LSC}\scrD$}
\text{for all  $(\rho_n)_n,\, \rho \in \mathrm{dom}(\scrE)$:} \qquad  \rho_n\to \rho \text{ setwise in } \calM^+(V) \ \Longrightarrow \ \liminf_{n\to\infty} \scrD(\rho_n)\geq \scrD(\rho)\,.
\end{equation}
\noindent The following result, 
whose  \emph{proof} follows the same lines as the argument for \cite[Prop.\ 5.3]{PRST22}, provides sufficient conditions for \eqref{FishLSC}.

\begin{prop}
\label{prop:Fisher-lsc}
    Suppose that either $\pi$ is purely atomic, or that the function  $\mathrm{D}_\upphi: \R_+\times \R_+ \to [0,+\infty]$ is convex.  Then, \eqref{FishLSC} holds. 
\end{prop} 

It should be noted that it is sufficient to verify whether the function $\mathrm{D}_\upphi^+$ restricted on $(0,+\infty)\times (0,+\infty)$ is convex, since any lower semicontinuous relaxation of a convex function defined over an open convex domain $\Omega$ to the closure $\overline{\Omega}$ is itself lower semicontinuous and convex.

\section{Gradient system structures}\label{s:3-NEW}
First of all, in Section \ref{ss:3.notions-sols},  we formalize our first two concepts of solutions for the evolution system associated with 
$(V,\kappa,\upphi,\uppsi,\upalpha)$  by introducing $\Dissipative$ and $\Balanced$ solutions. We discuss their relationship in Section \ref{ss:4.2-true}
and therein provide a `differential' characterization of $\Balanced$ solutions. In Section \ref{ss:ECE} we bring into play the $\Reflecting$ continuity equation and introduce $\Reflecting$ solutions. 
We prove that, under suitable conditions, $\Dissipative$ solutions that also solve the reflecting continuity equation turn out to be $\Reflecting$ solutions, ultimately; hence, \textsc{Main Theorem} \ref{thm:mchar} follows from Proposition \ref{prop:characteriz} and Theorem 
\ref{thm:ECE}. Finally, in Section \ref{ss:ASIDE} we  explore the role of a density property in upgrading solutions of the continuity equation  with upper and lower bounds, to $\Reflecting$ solutions.


\medskip

\par
Let us highlight that, throughout this section we will consider 
Assumptions \ref{Ass:E}, \ref{Ass:D}, \ref{Ass:flux-density}, and the lower semicontinuity property \eqref{FishLSC} in standing. We shall thus omit to explicitly invoke them in the various results of this section.

\subsection{Dissipative and Balanced solutions}
\label{ss:3.notions-sols}
 Firstly, as mentioned in the Introduction, it is worthwhile to pin down the weakest solution concept, based on the energy-dissipation \emph{inequality}, and its upgrade, featuring the energy-dissipation \emph{balance} on \emph{every} sub-interval $[s,t]\subset [0,T]$.
 
\begin{definition}[$\Dissipative$ and $\Balanced$ solutions] 
\label{def:weak-solution}
We say that a curve $\rho: [0,T] \to \calM^+(V)$ is a $\Dissipative$ solution \EEE of the $(\scrE,\scrR,\scrR^*)$ evolution system it satisfies the following conditions:
\begin{itemize}
\item[(1)]  $\calS(\rho_0)<\pinfty$;
\item[(2)]  There exists a measurable family $(\bj_t)_{t\in [0,T]} \subset \Ms(E)$ such that $(\rho,\bj)\in \CE0T$;
\item[(3)]  
The pair $(\rho,\bj)$ complies with the
the {\em $(\scrE,\scrR,\scrR^*)$
  energy-dissipation inequality}
\begin{equation}
\label{UEDE}
 \int_0^t \left( \scrR(\rho_r, \bj_r) + \Fish(\rho_r) \right) \dd r+ \calS(\rho_t)   \leq \calS(\rho_0)   \qquad \text{on any interval $[0,t]\subset [0,T]$}. 
\end{equation}
\end{itemize}
If  in addition  $(\rho,\bj)$  satisfies
\begin{itemize}
\item[(3')] 
the {\em $(\scrE,\scrR,\scrR^*)$ energy-dissipation balance}
\begin{equation}
\label{R-Rstar-balance}
\int_s^t \left( \scrR(\rho_r, \bj_r) + \Fish(\rho_r) \right) \dd r+ \calS(\rho_t) = \calS(\rho_s)   \qquad \text{on any interval $[s,t]\subset [0,T]$}, 
\end{equation}
\end{itemize}
we say that $\rho: [0,T] \to \calM^+(V)$ is a   $\Balanced$ solution of the $(\scrE,\scrR,\scrR^*)$ evolution system.
\end{definition}
\noindent
In what follows, with slight abuse, we will often refer to the \emph{pair} $(\rho,\bj)$ as a $\Dissipative/\Balanced$
solution to the system. 

\medskip
\par
The key link between these two concepts will be  provided by the following property.
\begin{definition}
\label{def:CR}
We say that the chain rule holds along a curve $(\rho,\jj)\in\CE 0T$, fulfilling
 \begin{equation}
\label{finite-entropy+action-gen}
    \sup_{t\in [0,T]} \scrE(\rho_t) < \infty,\qquad \int_0^T \left(\scrR(\rho_t,\bj_t){+} \scrD(\rho_t) \right) \dd t < \infty,
\end{equation} 
if we have that 
\begin{enumerate}
\item the mapping
\[
  \text{$[0,T]\ni t\mapsto \calS(\rho_t)$ is absolutely continuous,}
\]
  \item
  there holds, with $\rho_t = u_t \pi$, 
 \begin{equation}
 \label{chain-rule-gen}
 \tag{$\mathsf{CR}$}
  -  \frac{\dd}{\dd t} \calS(\rho_t)  = 
     \iint_{\Ed} ({-}\dnabla \upphi'{\circ} u_t)(x,y)\, \jj_t(\dd x \dd y) \qquad \text{for a.e.\ $t\in (0,T)$. }
 \end{equation} 
 \end{enumerate}
 \end{definition}
 \noindent In what follows, whenever citing \eqref{chain-rule-gen} we will also implicitly refer to the property that $t\mapsto \calS(\rho_t)$  be absolutely continuous. \EEE
 \begin{remark}
 Note that we do not claim the chain rule as a \emph{general property} of our evolution system $(\scrE,\scrR,\scrR^*)$. Namely, we do not require that the chain rule \eqref{chain-rule-gen} holds along \emph{every} curve $(\rho,\jj)\in\CE 0T$ for which the bounds \eqref{finite-entropy+action-gen} hold. Instead, we define it as a property that holds along an \emph{individual} curve $(\rho,\jj)\in\CE 0T$.
\par
This reflects the \emph{singular} character of the system $(\scrE,\scrR,\scrR^*)$, which contrasts the case of bounded kernels (cf.\ \cite[Thm. 4.16, Cor. 4.20]{PRST22}), where the chain rule was shown to be a property of the evolution system.
\par
Likewise, in this singular context, we no longer claim that the energy-dissipation functional 
\begin{equation*}
\label{trajectory}
\scrL_T(\rho,\jj) = \calS(\rho_T)-\calS(\rho_0)+ \int_0^T \bigl(\scrR(\rho_t, \bj_t) {+} \Fish(\rho_t)\bigr)\, \dd t 
\end{equation*}
is positive along \emph{all} solutions of the continuity equation for which it is finite.
\par
Instead, in the present setup we will prove the chain rule \eqref{chain-rule-gen} along curves satisfying the reflecting continuity equation.
\end{remark}
 \par
Due to the crucial role played by the bound \eqref{finite-entropy+action-gen}, it is important to specifically qualify solutions to the continuity equation with bounded entropy, finite action, and finite Fisher information. Thus, we introduce the class  
 \begin{equation*}\label{ACE}
 \ACE 0T  = \left\{ (\rho,\jj) \in \CE 0T\, : \  \sup_{t\in [0,T]} \scrE(\rho_t) < \infty,\ \int_0^T \left(\scrR(\rho_t,\bj_t){+} \scrD(\rho_t) \right) \dd t < \infty\right\}.
 \end{equation*}

\subsection{Dissipative versus Balanced solutions}
\label{ss:4.2-true}
 The following result shows that, in the context of the  evolution system associated with  $(V,\kappa,\upphi,\uppsi,\upalpha)$ as well 
the following property, typical of gradient systems, holds: $\Dissipative$ solutions upgrade to $\Balanced$ as soon as the chain rule holds along them. \EEE
\begin{theorem}
\label{th:diss.vs.bal}
Let $(\rho,\jj)\in\ACE 0T$ fulfill   the chain rule \eqref{chain-rule-gen}.
 \par
 Then, the following properties are equivalent:
 \begin{enumerate}
\item $(\rho, \jj) $ is a $\Dissipative$ solution of the $(\scrE,\scrR,\scrR^*)$ evolution system;
\item  $(\rho, \jj) $ is a $\Balanced$ solution of the $(\scrE,\scrR,\scrR^*)$ system, additionally fulfilling
the pointwise energy-dissipation balance
\begin{equation}
\label{enh-EDB}
\scrR(\rho_t, \bj_t) + \Fish(\rho_t) = -\frac{\dd }{\dd t}  \calS(\rho_t)    = \iint_{E}({-}\overline\nabla (\upphi'(u_t) ))\,  \bj_t (\dd x \dd y)  \qquad \foraa\, t \in (0,T).
\end{equation}
\end{enumerate}
\end{theorem}
\begin{proof}
Suppose that \eqref{chain-rule-gen} holds along a curve $(\rho,\jj)\in\ACE 0T$. 
Let  us pick  $t \in (0,T)$, out of a $\Lebone$-negligible set, such that $\frac{\dd }{\dd t}  \calS(\rho_t) $ exists; we can assume without loss of generality that 
$\Fish(\rho_t)<\infty$
and $\scrR(\rho_t,\bj_t) $, too, so that 
\[
|\bj_t |(\Ed {\backslash} E_{\upalpha_t}) =0 \quad \text{with } E_{\upalpha_t}: = \{ (x,y)\in \Ed\, : \ \upalpha(u_t(x),u_t(y))>0 \},
\]
and then we have (cf.\ \eqref{nice-representation})
\[
\scrR(\rho_t,\bj_t)  =
    \frac12\iint_{E_{\upalpha_t}}
    \Psi\Bigl(\frac{w_t(x,y)}{\upalpha(u_t(x),u_t(y))}\Bigr)\upalpha(u_t(x),u_t(y))\,\tetapi(\dd
    x,\dd y) 
    \]   
with  $w_t=\frac{\dd( 2\bj_t)}{\dd \tetapi}$.
Therefore, with Young's inequality we have 
\begin{equation}
\label{conseq-CR}
\begin{aligned}
 -\frac{\dd }{\dd t}  \calS(\rho_t)    &  \stackrel{\eqref{chain-rule-gen}}{=} 
 \frac12 \iint_{E_{\upalpha_t}} 
  \frac1{\upalpha(u_t(x),u_t(y))}  ({-}\overline\nabla \upphi'(u_t) ) \upalpha(u_t(x),u_t(y))  w_t(x,y) \, \tetapi (\dd x \dd y) 
  \\
  & 
 \leq  
   \frac12\iint_{E_{\upalpha_t}}
    \Psi\Bigl(\frac{w_t(x,y)}{\upalpha(u_t(x),u_t(y))}\Bigr)\upalpha(u_t(x),u_t(y))\,\tetapi(\dd
    x,\dd y)  
    \\
    &
    \quad + \frac12 
    \iint_{E_{\upalpha_t}}
    \Psi^*( {-}\overline\nabla \upphi'(u_t))\upalpha(u_t(x),u_t(y))\,\tetapi(\dd
    x,\dd y) 
    \\
    &
    = \scrR(\rho_t,\bj_t)+ \Fish(\rho_t) \,.
 \end{aligned}
\end{equation}
\par
Hence, let $(\rho,\jj)$ be a $\Dissipative$ solution. Combining 
\eqref{conseq-CR}
with the energy-dissipation inequality \eqref{R-Rstar-balance} we obtain that the latter turns into an energy balance on any sub-interval $[0,t]\subset [0,T]$. Then, 
$(\rho, \jj) $ is a \emph{Balanced solution}. Differentiating 
\eqref{R-Rstar-balance} we ultimately find \eqref{enh-EDB}.
\end{proof}
\par
In the case of \emph{bounded} kernels, it was shown in 
\cite[Thm.\ 5.7]{PRST22}, 
that, 
if $\rho : [0,T] \to \calM^+(V)$ is a $\Balanced$ solution,
then its density  $u:  [0,T] 
\to V $, $u_t = 
\frac{\dd \rho_t}{\dd \pi}$ 
satisfies an evolutionary equation, suitably involving the entropy and dissipation densities $\upphi$ and $\uppsi$, as well as $\upalpha$. In fact, let us introduce the mapping
\begin{equation}
\label{Fmap}
\Fmap : \R_+{\times} \R_+ 
\to [{-}\infty,{+}\infty], \qquad 
\Fmap (u,v) = \begin{cases}
(\uppsi^*)'\left( \Lambda_\upphi (u,v) \right) \upalpha (u,v) & \text{if } \upalpha (u,v)>0,
\\
0 
& \text{if } u=v=0,
\end{cases}
\end{equation}
with the convention $(\uppsi^*)'({\pm}\infty): = \pm 
\infty$, and let us consider the integro-differential equation
\begin{equation}
\label{integro-diff-Fmap}
\partial_t u_t (x) = \int_V \Fmap (u_t(x), u_t(y))\, \kappa (x, \dd y) \qquad \text{for $\pi$-a.e.\ $x\in V$, for almost every $t\in (0,T)$.}
\end{equation}
The proof 
that any $\Balanced$ solution complies with \eqref{integro-diff-Fmap}
 \cite[Thm.\ 5.7]{PRST22} heavily relied on the validity of the chain rule for the  $(\scrE,\scrR,\scrR^*)$  evolution system valid for bounded kernels,
 \emph{and} on the stronger character of the continuity equation considered in that regular setup. 
 \par
In the present singular setting, we  have a weaker analogue  of  \cite[Thm.\ 5.7]{PRST22}, in that we show that $\Balanced$ solutions \emph{along which the chain rule holds}, 
do satisfy a version of \eqref{integro-diff-Fmap} involving the flux measure $\jj$. Namely, we the following holds.
%
%
\begin{prop}[Differential characterization of $\Balanced$ solutions]
\label{prop:characteriz}
  Let $(\rho,\jj)\in\ACE 0T$ and suppose that the chain rule \eqref{chain-rule-gen} holds along $(\rho,\jj)$.
 \par
 Then, the following properties are equivalent:
  \begin{enumerate}
\item $\rho$ is a $\Balanced$ solution of the $(\scrE,\scrR,\scrR^*)$   system 
\item 

for the curve $[0,T] \ni t
\mapsto u_t = \frac{\dd \rho_t}{\dd t} \in \rmL^1(V;\pi)$ 
\begin{enumerate}
\item  
we have the bound
\begin{equation}
\label{cond58}
\int_0^T \iint_{\Ed} |\Fmap(u_t(x),u_t(y))| \tetapi (\dd x \dd y ) \dd t <+\infty;
\end{equation}
\item setting
\begin{equation}
\label{integro-diff-jj}
2\jj_t (\dd x, \dd y) : = -\Fmap(u_t(x),u_t(y))\tetapi (\dd x \dd y ) \qquad \foraa\, t \in (0,T),
\end{equation}
we have that $ (\rho,\bj)\in \CE0T$. 
  \end{enumerate}
  \end{enumerate}
\end{prop}
 The proof of Proposition~\ref{prop:characteriz} follows the very same lines as the argument for \cite[Thm.\ 5.7]{PRST22}, to which we refer the reader. \EEE
\subsection{The reflecting continuity equation and reflecting solutions}
\label{ss:ECE}
 In this section we bring under the spotlight an \emph{enhanced} version of the continuity equation, which draws on a richer set of test functions.
\par
\begin{definition}[Reflecting continuity equation]
\label{def:ECE}
We denote by $\ECE 0T$  the set of pairs $(\rho,\jj) \in \ACE 0T$ with $u = \frac{\dd \rho}{\dd \pi} \in \rmL^\infty (0,T;\rmL^\infty(V;\pi))$ fulfilling the continuity equation
 \begin{equation}
 \label{eq:RCE}
    \int_V \varphi(x)\,\rho_t(\dd x) - \int_V \varphi(x)\,\rho_s(\dd x) = \int_s^t \iint_\Ed \dnabla \varphi(x,y)\,\bj_r(\dd x \dd y)\,\dd r\qquad\text{for any $[s,t]\subset[0,T]$}
 \end{equation}
 for all test functions
  \begin{equation}
  \label{eq:test_extend-quadr}
    \varphi \in \mathcal{X}_2= \left\{ \varphi \in \rmL^\infty (V;\pi) : \   \iint_\Ed |\dnabla \varphi(x,y)|^2 \,\tetapi(\dd x \dd y) <\infty \right\}. 
\end{equation}
In this case, we say that the pair $(\rho,\jj)$ satisfies the reflecting continuity equation, or simply write $(\rho, \jj) \in \ECE 0T$.
\end{definition}

Moreover, bounded Lipschitz functions are in $\mathcal{X}_2$: indeed, for any $\varphi \in \Lipb(V)$, we have
 \[
\begin{aligned}
    \iint_{E} |\overline \nabla \varphi(x,y)|^2 \tetapi (\dd x \dd y ) & \leq   \| \varphi\|_{\Lipb(V)}^2 \iint_{E} (1{\wedge}\dV^2(x,y)) \tetapi (\dd x \dd y )<  +\infty\,.
\end{aligned}
\]
In fact, in  \eqref{eq:test_extend-quadr} $ |\dnabla \varphi|^2$ can be replaced by $\uppsi^*(\dnabla \varphi)$, as shown in the following result. This fact will play a crucial role for in the proof of Theorem \ref{thm:ECE} below.

\begin{lemma}
\label{l:worthwhile-highlight}
Let $\varphi \in \Bb(V)$. Then, $\varphi \in \mathcal{X}_2$ if and only if
  \begin{equation}
  \label{eq:test_extend}
\varphi \in \mathcal{X}^{\uppsi^*} = \left\{ \varphi \in \rmL^\infty (V;\pi) : \   \iint_\Ed \uppsi^*(\dnabla \varphi(x,y))\,\tetapi(\dd x \dd y)  <\infty \right\}.
\end{equation}
\end{lemma}
\begin{proof}
$\vartriangleright [\eqref{eq:test_extend} \, \Rightarrow \, \eqref{eq:test_extend-quadr}]$: We apply estimate \eqref{Olli-est-2} from Lemma \ref{l:props-uppsistar} with $M = \|\ona\varphi\|_{\infty}$, yielding 
\[
    |\dnabla \varphi|^2 \leq K_M \uppsi^*(\dnabla \varphi)  \qquad \tetapi\text{-a.e.\ in $\Ed$.}
\]
$\vartriangleright [\eqref{eq:test_extend-quadr} \, \Rightarrow \, \eqref{eq:test_extend}]$: From \eqref{needed-control-normal} we obtain that 
\[
    \uppsi^*(\dnabla \varphi) \leq C_{\varphi} |\dnabla \varphi|^2 \qquad \tetapi\text{-a.e.\ in $\Ed$,} \quad  \text{with } C_{\varphi}: = \frac{\psih^*(\|\ona\varphi\|_\infty)}{\|\ona\varphi\|_\infty^2}\,. \qedhere
\]
\end{proof}

In view of Lemma~\ref{l:worthwhile-highlight}, one easily verifies that the $\rmL^{\infty}$-bounds of $u$ and the finite action of $(\rho,\jj)$, imply that the integrals of \eqref{eq:RCE} are well-defined and finite, since it results in the existence of a constant $C=C(\|u\|_{\rmL^\infty(\lambda\otimes\pi)})>0$, such that
\begin{equation*}
\begin{aligned}
\iiint_{[0,T]\times E} |\dnabla \varphi|(x,y)  \, |\bj_{\!\Lebone}|(\dd t \dd x \dd y) 
\leq \int_0^T \scrR(\rho_t,\bj_t)\, \dd t +C T\iint_{E} |\dnabla \varphi|^2(x,y) \, \tetapi(\dd x \dd y).
\end{aligned}
\end{equation*}

\medskip
We are now in a position to introduce our strongest solution concept.
\begin{definition}[$\Reflecting$ solutions]
We say that a curve $\rho: [0,T] \to \calM^+(V)$ is a  $\Reflecting$ solution \EEE of the $(\scrE,\scrR,\scrR^*)$ evolution system it satisfies the following conditions:
\begin{enumerate}
\item $\calS(\rho_0)<\pinfty$;
\item There exists a measurable family $(\bj_t)_{t\in [0,T]} \subset \Ms(E)$ such that $(\rho,\bj)\in \ECE0T$;
\item The pair $(\rho,\bj)$ complies with the the {\em $(\scrE,\scrR,\scrR^*)$ energy-dissipation balance} \eqref{R-Rstar-balance}  on any interval $[s,t]\subset [0,T]$.
\end{enumerate}

\end{definition}
\par
The following result sheds light into the properties of curves in the class $ \ECE 0T$ if in addition, we assume convexity of the Fisher information functional $\Fish$.

\begin{theorem}
\label{thm:ECE}
 Suppose that
\begin{equation}
\label{convexity-Fisher-info}
\text{the Fisher information functional $\Fish\colon\mathrm{dom}(\scrE)\to [0,+\infty]$ is convex.}
 \end{equation}
Then, the set $ \ECE 0T$ is convex, i.e., for any $(\rho^0,\jj^0), \,  (\rho^1,\jj^1) \in  \ECE 0T$ and $\theta \in (0,1)$:
 \begin{equation}
 \label{convexityRCE}
    (\rho^\theta,\jj^\theta): = (1{-}\theta) (\rho^0,\jj^0)+\theta (\rho^1,\jj^1)\in  \ECE 0T \,.
 \end{equation}
 \par
 Moreover, if $(\rho,\jj) \in  \ECE 0T$, then the chain rule from Definition \ref{def:CR} holds along $(\rho, \jj)$. 
 In particular, if $(\rho,\jj) \in \ECE 0T$ is a $\Dissipative$ solution, then it is a $\Reflecting$ solution, fulfilling additionally the \emph{pointwise energy-dissipation balance}
\eqref{enh-EDB}. 
\end{theorem}

\begin{remark}
\label{rmk:atomic-vs-cvx}
\sl
It will be apparent from 
the
proof of Theorem \ref{thm:ECE}  that the required convexity of $\scrD$ becomes redundant if, in addition,  the density $u$ enjoys a lower bound as well, cf.\
\eqref{additional-lower-bound} below. 
 \end{remark}

\begin{proof}
$\vartriangleright$ First of all, we address the proof of \eqref{convexityRCE}. 
 By convexity of $\calS$, $\scrR$, and $\Fish$,
\[
\begin{cases}
\displaystyle \sup\nolimits_{t\in [0,T]} \scrE(\rho_t^\theta) \leq (1{-}\theta)   \sup\nolimits_{t\in [0,T]} \scrE(\rho_t^0) + \theta  \sup\nolimits_{t\in [0,T]} \scrE(\rho_t^1) <\pinfty,
\smallskip
\\
\displaystyle  \int_0^T \scrR(\rho_t^\theta,  \bj_t^\theta)\, \dd t \leq (1{-}\theta)\int_0^T \scrR(\rho_t^0, \bj_t^0)\, \dd t + \theta \int_0^T \scrR(\rho_t^1, \bj_t^1)\, \dd t  <\pinfty\,,
\smallskip
\\
\displaystyle  \int_0^T \scrD(\rho_t^\theta)\, \dd t \leq (1{-}\theta)\int_0^T \scrD(\rho_t^0)\, \dd t  +\theta \int_0^T \scrD(\rho_t^1)\, \dd t  <\pinfty\,.
\end{cases}
\]
Hence, the pair  $(\rho^\theta, \jj^\theta) $ has uniformly bounded entropy, and finite action and Fisher information and, evidently,
$u^\theta = \frac{\dd \rho^\theta}{\dd \pi} =  (1{-}\theta) u^0 +\theta u^1 \in  \rmL^\infty (0,T;\rmL^\infty(V;\pi))$. 
 Clearly, $(\rho^\theta,\jj^\theta)$ satisfies the continuity equation for all
test functions $\varphi 
\in \mathcal{X}_2$.  Thus, $ (\rho^\theta,\jj^\theta)\in \ECE 0T$. 
\par
$\vartriangleright$
We will split the proof of the chain rule into two steps: in {\bf \emph{Step $1$}}, we claim the validity of the chain rule \eqref{chain-rule-gen} along curves $(\rho,\jj) \in  \ECE 0T$ satisfying,
 in addition, 
the lower bound 
 \begin{equation}
\label{additional-lower-bound}
 \exists\,  \underline U>0 \, : \quad   u_t(x) \geq \underline U   \text{ for $\pi$-a.e.\ $x \in V$ and for all } t\in [0,T].
\end{equation}
in {\bf \emph{Step $2$}}, 
we exploit the convexity of $\Fish$  to 
 remove the lower bound \eqref{additional-lower-bound}. \EEE 
\medskip 

\par
{\bf \emph{Step $1$}}: Suppose that  \eqref{additional-lower-bound} holds. 
Then, we have that 
    \begin{equation}
    \label{needed-4-rigour}
     \upalpha(u_t(x),u_t(y)) \geq \underline{\upalpha}: = \min_{(u,v)\in [\underline U, \overline U]{\times} [\underline U, \overline U]} \upalpha(u,v)>0
     \quad \text{for } \pi\text{-a.e. } x,y \in V \quad \text{for all } t \in [0,T]\,.
         \end{equation}
\par
 We begin by showing the absolute continuity of the function  $[0,T]\ni t\mapsto \calS(\rho_t)$. 
     Fix any $t\in (0,T]$.  Note that without loss of generality we can assume that $\Fish(\rho_t)<\infty$.
      Then, by convexity of $\upphi$, we obtain for all $0\leq s \leq t$,
    \begin{align*}
       &  \calS(\rho_t) - \calS(\rho_s)
       \\
        &\le \int_V \upphi'(u_t(x)) [u_t(x)-u_s(x)]\,\pi(\dd x) \\
        &\stackrel{(1)}{=} \int_s^t \iint_{E} (\dnabla \upphi'{\circ} u_t)(x,y)\,\jj_r(\dd x \dd y) \dd r \\
        &\stackrel{(2)}{=} \int_s^t \iint_{E}    \frac1{\upalpha(u_r(x),u_r(y))}  (\dnabla \upphi'{\circ} u_t)(x,y) 
         \upalpha(u_r(x),u_r(y)) 
        \,\jj_r(\dd x \dd y) \dd r \\ 
        &\stackrel{(3)}\le \int_s^t \scrR(\rho_r,\jj_r)\,dr +\frac12  \int_s^t \iint_{E}  \uppsi^*({-}\dnabla (\upphi'{\circ} u_t)(x,y))\,\upalpha(u_r(x),u_r(y))\,\tetapi(\dd x \dd y) \dd r \\
        &\le \int_s^t \scrR(\rho_r,\jj_r)\,dr +  \frac12    \int_s^t \iint_{E}  \uppsi^*((\dnabla \upphi'{\circ} u_t)(x,y))\, \upalpha(u_t(x),u_t(y)) \frac{\upalpha(u_r(x),u_r(y))}{\upalpha(u_t(x),u_t(y))}
        \,\tetapi(\dd x \dd y) \dd r  \\
        & \leq \int_s^t \scrR(\rho_r,j_r)\,dr +  \frac{C_{\overline{U}}}{ \underline{\upalpha}} \EEE  \Fish(\rho_t)|t-s|\,.
    \end{align*}
    Here,
    {\footnotesize (1)} follows from choosing the test function $\varphi =  \upphi'{\circ} u_t $ in the continuity equation: 
   in fact, due to $u \in  \rmL^\infty (0,T;\rmL^\infty(V;\pi))$ and the additional lower bound \eqref{additional-lower-bound}, we have $\upphi'{\circ} u_t \in \rmL^\infty (V;\pi)$. Furthermore, $\Fish(\rho_t)<\infty$  yields
   \[
   \underline{\upalpha}  \iint_{E}
    \Psi^*( {-}\overline\nabla \upphi'(u_t))\,\tetapi(\dd
    x,\dd y)  \leq
    \iint_{E}
    \Psi^*( {-}\overline\nabla \upphi'(u_t))\upalpha(u_t(x),u_t(y))\,\tetapi(\dd
    x,\dd y) <+\infty,
   \]
   and thus $\upphi'{\circ} u_t \in  \mathcal{X}_2$ by Lemma \ref{l:worthwhile-highlight}.  
    Estimate {\footnotesize (2)} is justified by \eqref{needed-4-rigour}, while {\footnotesize (3)} ensues from Young's inequality (cf.\ \eqref{conseq-CR}) and from the fact that $\uppsi^*$ is an even function. All in all, for the place-holder $\mathcal{W}(t): = \int_0^t \scrR(\rho_r,j_r)\,dr -\calS(\rho_t)$ we deduce the estimate
\[
    \mathcal{W}(s)-\mathcal{W}(t) \leq \frac{C_{\overline{U}}}{ \underline{\upalpha}} \EEE  \Fish(\rho_t)|t-s|\,.
\]
Then, we are in a position to apply \cite[Lemma 2.9]{Ambrosio-Gigli-Savare14} and infer that the mapping $[0,T]\ni t \mapsto \calS(\rho_t)$ is in $\AC([0,T])$.  

\par
Let us  now fix a point $t \in (0,T)$, out of a negligible set, where $\frac{\dd}{\dd t}  \calS(\rho_t)$ exists.
With the same calculations as above 
we obtain for $h>0$
\begin{align*}
    \frac{1}{h}\bigl[\calS(\rho_t) - \calS(\rho_{t+h})\bigr] &  \leq \frac1h \int_V \upphi'(u_t(x)) [u_t(x)-u_{t+h}(x)]\,\pi(\dd x) 
    \\
    &
    =    \int_{t}^{t+h} \iint_{E} {-}\dnabla (\upphi'{\circ} u_t)(x,y)\,\jj_r(\dd x \dd y) \dd r \EEE 
    \end{align*}
    and, analogously,
    \[
     \frac{1}{h}\bigl[\calS(\rho_t) - \calS(\rho_{t+h})\bigr] \geq   \int_{t}^{t+h} \iint_{E} {-}\dnabla (\upphi'{\circ} u_{t+h}))(x,y)\,\jj_r(\dd x \dd y) \dd r 
    \]
    Letting $h\down0$, we conclude  the desired \eqref{chain-rule-gen}. 
\medskip

\par
{\bf \emph{Step $2$}}:
In order to remove the additional condition \eqref{additional-lower-bound}  on the pair $(\rho,\jj)$, we may argue as follows. 
Let $(\bar\rho, \bar\jj)$ be given by  $\bar\rho :[0,T] \to  \calM^+(V)$, $\bar{\rho}_t \equiv \pi$ for all $t\in [0,T]$, and $\bar\jj :=0 $ . 
It is immediate to check that $(\bar\rho, \bar\jj) \in \ECE 0T$.   
 Now consider the pair $(\rho^\theta, \jj^\theta)$ 
\begin{equation}
\label{clever-trick-lowerbound}
(\rho^\theta, \jj^\theta):= (1{-}\theta) (\rho, \jj) +\theta (\bar\rho, \bar\jj) \ \theta \in [0,1]\,. 
\end{equation}
By the previously proved convexity of $\ECE 0T$, we have that $(\rho^\theta, \jj^\theta) \in \ECE 0T$ for all $\theta \in [0,1]$. 
Furthermore, by construction, $u_t^\theta =\frac{\dd \rho_t^\theta}{\dd \pi} $ satisfies the lower bound  
 \[
 u_t^\theta(x) \geq \theta \qquad \text{for $\pi$-a.e.\ $x\in V$, \ for all $t\in [0,T]$.} 
 \]
Thus, the chain-rule 
\eqref{chain-rule-gen} holds along $(\rho^\theta, (1{-}\theta) \jj)$, yielding for all $0 \leq s \leq t \leq T$
\begin{equation}
\label{jasper}
\begin{aligned}
 \calS(\rho_t^\theta) - \calS(\rho_s^\theta)    = 
   \int_s^t\iint_{\Ed} ({-}\dnabla \upphi'{\circ} u_r^\theta)(x,y)\, \jj^\theta_r(\dd x \dd y) 
     \dd r  \,.
    \end{aligned}
\end{equation}
\par
Let us now send $\theta \downarrow 0$ in the first line of \eqref{jasper}. We clearly have 
\[
\calS(\rho_r)  \leq \lim_{\theta \down 0} \calS(\rho_r^\theta)  
\leq \lim_{\theta \down 0}  (1{-}\theta) \calS(\rho_r)=  \calS(\rho_r) \qquad \text{for all } r \in [0,T]\,.
\]
As for the right-hand side, using that  $\jj^\theta_t = \tfrac12 w_t^\theta \tetapi  = \tfrac12(1{-}\theta) w_t \tetapi$, with 
$w= \frac{\dd (2\jj)}{\dd\tetapi}$, 
we rewrite the integral as
\[
  \int_s^t\iint_{\Ed} ({-}\dnabla \upphi'{\circ} u_r^\theta)(x,y)\, \jj^\theta_r(\dd x \dd y) \dd r =  \int_s^t \iint_{\Ed} \frac12 ({-}\dnabla \upphi'{\circ} u_r^\theta)(x,y) w_r^\theta(x,y)\, \tetapi(\dd x \dd y)  \dd r
\]
The integrand on the right-hand side pointwise converges to 
$\tfrac12 ({-}\dnabla \upphi'{\circ} u_t) w_t$ for almost every $t\in (0,T)$ and 
$\tetapi$-a.e.\ in $\Ed$. Moreover, mimicking the calculations in Step $1$ we obtain the 
 estimate for $\tetapi$-a.e.\ $(x,y) \in \Ed$ and a.e.\ $r\in (0,T)$,
\[
\begin{aligned}
&
\frac12|(\dnabla \upphi'{\circ} u_r^\theta)(x,y) w_r^\theta(x,y)| \\
        &\quad
        \leq
     \frac12   \uppsi \left( \frac{\frac12 w_r^\theta(x,y)}{\upalpha(u_r^\theta(x),u_r^\theta(y))}  \right) \upalpha(u_r^\theta(x),u_r^\theta(y)) + \frac12    \uppsi^*\left(( {-}
 \dnabla \upphi'{\circ} u_r^\theta)(x,y)) \right)  \upalpha(u_r^\theta(x),u_r^\theta(y)) \,,
 \end{aligned}
\]
where the estimate for the modulus derives from  the evenness of $\uppsi$ and $\uppsi^*$. Now, as $\theta \down 0$, the above dominants converge a.e.\ to their analogues for $(\rho,\jj)$. Moreover, by lower semicontinuity and convexity
\[
\begin{aligned}
\int_s^t  \scrR(\rho_r,\jj_r) 
\,\dd r \leq \liminf_{\theta\down 0} \int_s^t \scrR(\rho_r^\theta,\jj^\theta_r)\, \dd r \leq  \liminf_{\theta\down 0}(1{-}\theta)\int_s^t \scrR(\rho_r,\jj_r)\, \dd r = \int_s^t  \scrR(\rho_r,\jj_r)\,  \dd r\,.
\end{aligned}
\]
Analogously, we prove that 
\[
\lim_{\theta\down 0}\int_s^t \scrD(\rho_r^\theta)  \dd r = \int_s^t \scrD(\rho_r)\, \dd r \,.
\]
We may once again apply the dominated convergence theorem \cite[\S 2, Thm.\ 2.8.8]{Bogachev07} to conclude that 
\[
  \int_s^t\iint_{\Ed} ({-}\dnabla \upphi'{\circ} u_r^\theta)(x,y)\, \jj^\theta_r(\dd x \dd y)\, \dd r \longrightarrow 
   \int_s^t\iint_{\Ed} ({-}\dnabla \upphi'{\circ} u_r)(x,y)\, \jj_r(\dd x \dd y)\, \dd r \,.
\]
 All in all, taking the limit in \eqref{jasper} we have shown that 
\[
 \calS(\rho_t) - \calS(\rho_s)    = 
   \int_s^t\iint_{\Ed} ({-}\dnabla \upphi'{\circ} u_r)(x,y)\, \jj_r(\dd x \dd y) 
     \dd r   \qquad \text{for all } 0 \leq s \leq t \leq T.
\]
Therefore, the chain rule \eqref{chain-rule-gen} ensues. \EEE
\par
The very last part of the statement follows from Theorem \ref{th:diss.vs.bal}, which ensures that the pair $(\rho,\jj) \in \ECE 0T$ complies with the \emph{pointwise energy-dissipation balance} \eqref{enh-EDB}. Whence, it is a $\Reflecting$ solution. 
 \end{proof}
 
 As a straightforward corollary of Theorem \ref{thm:ECE} we even have convexity of the class of $\Reflecting$ solutions emanating from the 
 \emph{same} initial datum. 
 \begin{cor}
 \label{cor:cvx-refl}
 Suppose that the Fisher information functional $\Fish$ is convex. 
\par
 Then, the family of $\Reflecting$ solutions of the $(\scrE,\scrR,\scrR^*)$ evolution system emanating from a given initial datum $\rho_0 = u_0 \pi $, with $u_0 \in \rmL^\infty (V;\pi)$, is convex.
 \end{cor}
 \begin{proof}
 Let $(\rho^i,\jj^i)$, $i = 0,1$, be $\Reflecting$ solutions, and consider their convex combination 
 $(\rho^\theta, \jj^\theta): = (1{-}\theta)(\rho^0,\jj^0)+\theta(\rho^1,\jj^1) $, which is in $\ECE 0T$. It remains to prove that the pair $(\rho^\theta, \jj^\theta)$ complies with the 
 energy-dissipation balance. For this,  we use that $\scrL_t(\rho_i,\jj_i) =0$ for $i=0,1$. 
 By the convexity  of the trajectory functional $(\rho, \jj) \mapsto \scrL_t (\rho, \jj)+\calS(\rho(0))$, we then have 
\[
\begin{aligned}
\scrL_t (\rho^\theta, \jj^\theta) + \calS(\rho^\theta(0))  & \leq (1{-}\theta)\scrL_t(\rho^0,\jj^0)
+
(1{-}\theta) \calS(\rho^0(0))
+\theta\scrL_t(\rho^1,\jj^1)
 +\theta \calS(\rho^1(0)) 
 \\
 & = 0+  \calS(\rho_0) =  \calS(\rho^\theta(0))
\qquad \text{for all } t \in [0,T],
\end{aligned}
\]
where we have used that $\rho^0(0) = \rho^1(0) = \rho^\theta(0) =\rho_0$. All in all, we have $\scrL_t (\rho^\theta, \jj^\theta)  =0$, 
 giving that $ (\rho^\theta, \jj^\theta)\in \ECE 0T $ is a $\Dissipative$ solution. By Theorem \ref{thm:ECE} it is a $\Reflecting$ solution.
 \end{proof}

\par
We conclude this section addressing uniqueness in the case the entropy density is strictly convex.  In the case of \emph{bounded} kernels, a uniqueness result for $\Balanced$ solutions was proved in \cite[Thm.\ 5.9]{PRST22}. 
Instead, in the present context, uniqueness holds, as an immediate consequence of Corollary \ref{cor:cvx-refl}, only within the subclass of $\Balanced$ solutions precisely given by $\Reflecting$ solutions. The argument for the proof of Proposition \ref{prop:uniqueness-too} mimicks that for \cite[Thm.\ 5.9]{PRST22}.

\begin{prop}
\label{prop:uniqueness-too}
Suppose that the Fisher information functional  $\Fish$ is convex and the entropy density $\upphi$ is \emph{strictly} convex. 
Let $(\rho^1,\jj^1), \, (\rho^2,\jj^2) \in \ECE 0T$ be two $\Reflecting$ solutions of the $(\scrE,\scrR,\scrR^*)$ evolution system, emanating from the same initial datum $\rho_0$.
\par
Then, $\rho_t^1=\rho_t^2$ for every $t\in [0,T]$.
\end{prop} 
\begin{proof}
By Corollary \ref{cor:cvx-refl}, we have that the pair $(\widehat\rho,\widehat\jj) = \left(\frac12(\rho^1{+}\rho^2) ,\frac12(\jj^1{+}\jj^2) \right)$ is a $\Reflecting$ solution. Therefore,
\[
0= \scrL_t(\widehat{\rho},\widehat{\jj}) \qquad \text{for all } t \in [0,T]\,.
\]
and, a fortiori $\calS(\widehat{\rho}_t) \equiv 0$.
Since $\calS$ is strictly convex, we gather $\rho_t^1=\rho_t^2$ for all $t\in [0,T]$. 
\end{proof}

\subsection{Enhancing the continuity equation via a density property}
\label{ss:ASIDE}
In this section we explore a situation in which all curves $(\rho,\jj)\in \CE 0T$ such that the curve  $u_t = \frac{\dd \rho_t}{\dd \pi}$ is bounded from above \emph{and below}, upgrade to $(\rho,\jj)\in \ECE 0T$. As we will show, this holds as soon as the space of bounded Lipschitz functions is dense in the space $\mathcal{X}_2$ \eqref{eq:test_extend-quadr}. 

\par
Let us emphasize that all of our existence  results for the singular $(\calS, \scrR,\scrR^*)$ evolution system are independent of such density property and the forthcoming discussion does not directly contribute to the existence of $\Reflecting$ solutions. Yet, we deem Theorem \ref{thm:density-vindicated} ahead noteworthy because it conveys the gap in generality that occurs when passing from evolution systems in $V=\R^d$, where the validity of density condition is ensured by convolution kernels, to the analysis of systems in more general ambient spaces, cf.\ also Remark \ref{rmk:density-gap} ahead. 

\par
Prior to stating this density property, we we need to settle some preliminary definitions.
 \subsubsection*{\bf Orlicz spaces and the density condition}
  Let $\Young : \R \to [0,+\infty)$ be a  \emph{Young} function, i.e.\  even, convex, with superlinear growth at infinity  and $\Young (0)=0$.  \EEE
  On the measure space 
  $(E, \mathfrak{B}(E), \tetapi  )$
  we introduce the \emph{Orlicz space}    associated with the \emph{Young} function $\Young$, namely
\[
 \rmL^{\Young}(E;\tetapi): = \left\{ \zeta \in \rmL^1(E;\tetapi)\, : \ \exists\, \ell> 0 \ \iint_{E} \Young\left(\frac{\zeta(x,y)}{\ell}\right) \, \tetapi(\dd x \dd y) <+\infty\right\},
 \]
with the associated Luxemburg norm
\[
\|\zeta \|_{\rmL^{\Young}(E;\tetapi)} : = \inf\left\{ \ell > 0 \, : \   \iint_{E} \Young\left(\frac{\zeta(x,y)}{\ell}\right) \, \tetapi(\dd x \dd y)  \leq 1\right \}.
\]
  While referring to, e.g.,  \cite{RaoRen} for a comprehensive  presentation of Orlicz spaces, let us pin down a key fact for later use: 
 for any $(\zeta_n)_n \subset \rmL^{\Young} (E; \tetapi)$ we have that 
\begin{equation}
\label{characterization-of-Orlicz}
\|\zeta_n \|_{ \rmL^{\Young} (E; \tetapi)} \to 0 \quad \text{iff} \quad \lim_{n\to\infty}\iint_{E} \Young(\beta   \zeta_n(x,y)) \, \tetapi (\dd x \dd y) =0 \text{ for all } \beta >0\,,
 \end{equation} 
  see e.g.\ \cite[Lemma 1.16]{Leo-unpubl}. \EEE
\par
 We will also consider the space
 \begin{equation}
 \label{X-L-Orli}
 \rmX^{\Young}: = \{ \varphi \in \rmL^\infty (V;\pi)\,: \ \overline{\nabla}\varphi \in \rmL^{\Young}(E;\tetapi)\}\,.
 \end{equation}
 Furthermore, let us  introduce the 
 \emph{small Orlicz} space
\[
 \mathcal{M}^{\Young}(E;\tetapi): = \left\{ \varphi \in \rmL^1(E;\tetapi)\, : \ \forall\, \ell> 0 \ \iint_{E} \Young\left(\frac{\varphi(x,y)}{\ell}\right) \tetapi(\dd x \dd y) <+\infty\right\},
 \]
\noindent and, accordingly,  
  \begin{equation}
 \label{X-M-Orli}
 \mathcal{X}^{\Young}: = \bigl\{ \varphi \in \rmL^\infty (V;\pi)\,: \ \overline{\nabla}\varphi \in \mathcal{M}^{\Young}(E;\tetapi)\bigr\}\,.
 \end{equation}
Clearly, $\mathcal{M}^{\Young}(V;\pi)$ is a subspace of $\rmL^{\Young}(E;\tetapi)$, so that $ \mathcal{X}^{\Young}\subset \rmX^{\Young}$. The space $\mathcal{X}^{\uppsi^*}$ from \eqref{eq:test_extend} falls within this class, as does $\mathcal{X}_2$ from \eqref{eq:test_extend-quadr}, corresponding to the \emph{quadratic} Young function $ \Young(\xi)= \frac12 |\xi|^2$. Our density condition \eqref{Ass:F-bis} below involves $\Lipb(V)$ and $\mathcal{X}_2$. 
\par
We are now in a position to  state the main result of this section.
To avoid overburdening the exposition, we  postpone its \emph{proof} 
to Appendix \ref{app:density-vindicated}. 
\begin{theorem}
\label{thm:density-vindicated}
Assume that $\Lipb(V)$ is dense in $\mathcal{X}_2$ in the following sense: 
 \begin{equation}
  \label{Ass:F-bis}
  \begin{aligned}
  &
 \text{for every  $\varphi \in \mathcal{X}_2$ there exists a sequence $(\varphi_n)_n \subset \Lipb(V)$ such that as $n\to \infty$,}
\\
& \text{$\varphi_n\to\varphi$ in $\rmL^1(V;\pi)$,}
\\
& \text{$\overline\nabla \varphi_n \to \overline\nabla \varphi$ in $\rmL^2(E;\tetapi)$.}
\end{aligned}
 \end{equation}
 Let $(\rho,
 \jj) \in \ACE0T$ with $u =  \frac{\dd \rho}{\dd \pi} \in \rmL^\infty (0,T; \rmL^\infty(V;\pi))$. 
Then, $(\rho,
 \jj) \in \ECE0T$. 
\end{theorem}
 As a direct consequence of Theorems  \ref{thm:ECE} and 
\ref{thm:density-vindicated}, we have the following.

\begin{cor}
\label{cor:density-vindicated}
Assume \eqref{Ass:F-bis} and suppose that the Fisher information functional  $\Fish\colon\mathrm{dom}(\scrE)\to [0,+\infty]$ is convex. Then, any $\Dissipative$ solution $(\rho, \jj)$ with $u = \frac{\dd \rho}{\dd \pi} \in \rmL^\infty (0,T; \rmL^\infty(V;\pi))$ is a $\Reflecting$ solution.
\end{cor}

\begin{remark}[Broader character of property \eqref{Ass:F-bis}]
    In fact, by resorting to the assumed condition \eqref{quadratic-at-0} on $\uppsi^*$, in Lemma \ref{cor:Oliver} in Appendix \ref{app:density-vindicated} we will show that \eqref{Ass:F-bis} guarantees the validity of the same approximation property, with the space $\mathcal{X}_2$ replaced by the space $\mathcal{X}^{\uppsi^*}$ induced by the Young function $\Young = \uppsi^*$, and with the convergence $\overline\nabla \varphi_n \to \overline\nabla \varphi$ in $\rmL^2(E;\tetapi)$ replaced by that in $\rmL^{\uppsi^*}(E;\tetapi)$. 
\end{remark}

\begin{remark}[Density gap]
\label{rmk:density-gap}
    In Appendix \ref{app:examples-density} we will exhibit  examples of setups in which property \eqref{Ass:F-bis} holds: besides the obvious choices of $V$ as $\R^d$ or as a bounded (Lipschitz) domain $\Omega\subset \R^d$, we will show its validity in the $d$-dimensional torus.
\par
    The well-groundedness of \eqref{Ass:F-bis} in ambient spaces more general than $\R^d$ remains an open problem. We have chosen not to explore this direction further, because, indeed, all of our existence results for the singular $(\calS, \scrR,\scrR^*)$ evolution system do not, in fact, \emph{depend} on \eqref{Ass:F-bis}. 
\end{remark}

\section{Approximation via regularized kernels and robustness of Reflecting solutions}
\label{s:4}
\noindent
As mentioned in the Introduction, we will prove our existence results  for the evolution system associated with $(V,\kappa,\upphi,\uppsi,\upalpha)$, by approximating the singular
kernels $(\kappa(x,\cdot))_{x\in V}$ via a family of regularized kernels (see Theorems \ref{th:main-0}, \ref{th:main} \& \ref{thm:robustness} ahead, corresponding to \textsc{Main Theorems} \ref{thm:mdissex}, \ref{thm:mrexist} \& \ref{thm:mrobust} respectively).
\par
We set up the approximate systems in Section \ref{ss:4.1} below, and state our convergences results in Sec.\ \ref{ss:4.2}. The cornerstone of their proofs will be a suitable robustness result for $\Reflecting$ solutions, stated in Sec.\ \ref{ss:robustness} ahead. 
\par
  Hereafter, in addition to our standing Assumptions \ref{Ass:E},  \ref{Ass:D} \& \ref{Ass:flux-density}, we shall require the following.
  \begin{assumption}
  \label{ass:dupphi} The function 
   $\mathrm{D}_\upphi: \R_+\times \R_+ \to [0,+\infty]$
   from \eqref{dupphi}
    is convex.
  \end{assumption} 
  As we will see, on the one hand, Assumption~\ref{ass:dupphi} will ensure a pivotal lower semicontinuity property for the Fisher information functionals, which, in this context, depend on the approximating parameter since we will approximate the kernel $\kappa$, cf.\ \eqref{eps-Fisher} below. On the other hand, Assumption~\ref{ass:dupphi} clearly implies the convexity condition \eqref{convexity-Fisher-info} in Theorem \ref{thm:ECE} which, in turn, will be at the core of the proof of Theorem \ref{th:main}.

\subsection{Setup of the approximation}
\label{ss:4.1}
Let us consider a sequence $(a_{n})_{n}$  of nonnegative \emph{symmetric} bounded functions
$a_n :E\to \R_{+}$  such that 
\begin{equation}
\label{eq:an-1}
\sup_{n} \|a_{n}\|_{\rmL^\infty(\tetapi)}< \infty, \quad \mbox{$a_{n}(x,y) \to 1 $ for  $\tetapi$-a.e.\ $(x,y) \in E$.}
\end{equation}
Henceforth we will assume 
\[\sup_{n}  \|a_{n}\|_{\rmL^\infty(\tetapi)} \leq 1,
\]
but all the statements contained in this section are still valid, upon to minor modification of the estimates involved, when $\sup_{n}  \|a_{n}\|_{\rmL^\infty(\tetapi)}$ is estimated by a generic positive constant. 
\par
For each $n\in \N$ we define the  kernels $(\kappa_n(x,\cdot))_{x\in V} \subset \Ms^+(V)$ via 
\begin{equation}
\label{kappa-eps-reg-n}
	\kappa_n(x,\dd y) := a_{n}(x,y)\kappa(x,\dd y) \qquad \text{for all } x \in V,
\end{equation}
which in turn induce the couplings  $(\tetapin)_n$  on $E$  given by
\begin{equation}
\label{tetapin}
\tetapin (\dd x \dd y) :=\kappa_n(x, \dd y  ) \pi (\dd x )\,.
\end{equation}
Observe that, by the symmetry of $a_n$ each $\tetapin$ still  satisfies the detailed balance \eqref{DBC}. 
\par
It is immediate to check that, by their very definition, the kernels $(\kappa_n)_n$ satisfy 
\begin{equation}
\label{mitigated-bounde}
\forall\, n \in \N\, : 
\quad
\begin{cases}
\displaystyle
\kappa^{n}\leq \kappa, 
\\
\displaystyle 
 \sup_{x\in V} \int_{V}  (1{\wedge}\dV^2(x,y)) \kappa_{n}(x,\dd y) \leq c_\kappa
 \end{cases}
\end{equation}
with $c_\kappa$ the very same constant as in \eqref{mitigation of singularity},  as  we have for convenience supposed $\sup_{n \in \N}\|a_n\|_{\infty}\leq 1$.
Moreover, 
 $(\kappa_n)_n$  is  a measurable family of measures, cf.\ Lemma \ref{eq:knass} ahead. 
\par
We now aim to consider the evolution systems $(\scrE,\scrR^n,(\scrR^n)^*)_n$ arising in connection with the structures $(V,\mathsf{d}, \kappa_n)$ and use their solutions
to approximate a $\Dissipative$/$\Reflecting$ solution to the original system $(\scrE,\scrR,\scrR^*)$ associated with  the singular kernel $\kappa$. 
For this, we need to ensure that the systems 
$(\scrE,\scrR^n,(\scrR^n)^*)_n$ do possess solutions. That is why, in Section 
\ref{ss:4.2} ahead
we will additionally suppose that 
the kernels $\kappa_n$ are \emph{bounded}, namely
\begin{equation}
\label{kan-back-2-bounded-kernel}
\tag{\textrm{bnd}\,$\kappa_n$}
\forall\,n \in \N\, : \quad  \sup_{x \in V} \kappa_n(x,V) <+\infty\,,
\end{equation}
so that the couplings $\tetapin$ are now \emph{finite} measures  on $E$. 
To fix ideas, one may think of the following prototypical example.
\begin{example}
\label{kappa-eps-reg}
Let $(\eps_n)_n$ be a null sequence and   set 
\[
a_{\eps_n}(x,y) =  \frac{1\wedge\dV^2(x,y)}{\varepsilon_n + (1{\wedge}\dV^2(x,y))} \qquad \text{for all } (x,y) \in E\,.
\]
Consider the kernels $\kappa_{\eps_n}(x,\dd y): = a_{\eps_n}(x,y)  \kappa (x, \dd y)$. For every  $x
\in V$, $  \kappa_{\eps_n}(x,\dd y)$ is a \emph{finite} measure on the whole of $V$, and by \eqref{mitigation of singularity}   there holds
\begin{equation}
\label{back-2-bounded-kernel}
\sup_{x \in V} \int_{V} \kappa_{\eps_n}(x, \dd y) 
= \sup_{x\in V} \int_V \frac1{\varepsilon_n + 1{\wedge}\dV^2(x,y)} (1{\wedge}\dV^2(x,y))  \kappa(x,\dd y) 
\leq \frac1{\eps_n} c_\kappa\,.
\end{equation}
\end{example}

\par
Let $\scrR^n: \calM^+(V) \times \calM(E) \to [0,+\infty]$ and $(\scrR^n)^*  : \calM^+(V)\times \Cb(E) \to [0,+\infty]$ the primal and dual dissipation potentials associated with the couplings $\tetapin$ via Definition \ref{def:primal-and-dual} (with  $\Ed$ replaced  by $E$), and let $\scrD^n$ the corresponding Fisher information functionals
\begin{equation}
    \label{eps-Fisher}
    \Fish^n(\rho): = 
\begin{cases}
\frac12 \iint_{E}  \mathrm{D}_\upphi(u(x),u(y)) \, \tetapin(\dd x \dd y) \qquad \text{if  } \rho = u \pi \text{ and  $ \mathrm{D}_\upphi(u(x),u(y)) \in L^1(E,\tetapin)$.}
\\
+\infty \text{ otherwise.}
\end{cases} 
\end{equation}
\par
Under the standing Assumptions  \ref{Ass:E}, \ref{Ass:D}, \ref{Ass:flux-density}, \ref{ass:dupphi}, and in view of the boundedness property \eqref{kan-back-2-bounded-kernel}, the  regularized system $(V,\kappa_n,\upphi,\uppsi,\upalpha)$ complies with the conditions of \cite[Thm.\ 5.7]{PRST22}. The latter ensures 
 \begin{itemize}
 \item[-]
the existence of 
a $\Balanced$ $(\rho,\jj)$ solution to the $(\scrE,\scrR^n,(\scrR^n)^*)$ evolution system,
\item[-]
that fulfills the continuity equation along any sub-interval $ [s,t]\subset [0,T]$, with arbitrary test functions $\varphi \in \Bb(V)$. 
\end{itemize}
Thus, a fortiori we have the existence of a $\Reflecting$ solution as soon as the density  $u = \frac{\dd \rho}{\dd \pi}$  is bounded. 
More precisely, we have  the following statement, which also encompasses the maximum principle proved in \cite[Thm.\ 6.5]{PRST22} and the uniqueness
result from \cite[Thm.\ 5.9]{PRST22}. 
\begin{prop}{\cite[Thms\ 5.9, 5.10, \& 6.5]{PRST22}}
\label{prop:fromPRST}
 For $n\in \N$  fixed, let the kernel $\kappa_n$ be defined by \eqref{eq:an-1} and \eqref{kappa-eps-reg-n}. The following statements hold:
 \begin{enumerate}
 \item For every $\rho^0\in \dom(\scrE)$, there is a $\Balanced$ solution $(\rho,\jj)$ of the $(\scrE,\scrR^n,(\scrR^n)^*)$ system, which is \emph{unique} if $\upphi$ is strictly convex, fulfilling the continuity equation for all $\varphi \in \Bb(V)$. 
%
%
 \item If $\rho^0 = u^0 \pi$ and the initial density  $u^0 \in L^\infty(V,\pi)$ fulfills 
 \begin{equation}
 \exists\, 0 \leq  \underline{U} < \overline{U} \ \text{ for $\pi$-a.e.\ $x\in V$:} \qquad 
  0 \leq \underline{U} \leq u^0(x) \leq \overline{U}\,,
  \end{equation}
  then 
  $\rho = u \pi$ satisfies
  \begin{equation}
  \label{boundedness}
   0 \leq \underline{U} \leq u_t(x) \leq \overline{U} \qquad \text{for } (\Lebone{\times}\pi)\text{-a.e. } (t,x) \in (0,T){\times}V\,.
  \end{equation}
  As a consequence, $(\rho, \jj)$ is also a $\Reflecting$ solution  of the 
   $(\scrE,\scrR^n,(\scrR^n)^*)$ system.
 \end{enumerate}
\end{prop}

\subsection{Our convergence results}
\label{ss:4.2}
Now, let  $(\rho^{n},\jj^n)_n$ be a sequence of ($\Balanced$ solutions of the $(\scrE,\scrR^n,(\scrR^n)^*)$ evolution systems, starting from initial data $(\rho_0^n)_n \subset \calM^+(V)$ satisfying suitable conditions---their existence is ensured by Proposition \ref{prop:fromPRST}.
\par
Our first main result  addresses the convergence of (a subsequence of)  $(\rho^{n},\jj^{n})_n$, to a $\Dissipative$ solution, in the sense of Definition \ref{def:weak-solution}, of the $(\scrE,\scrR,\scrR^*)$ system. 
In the following statement and thereafter, we  will use this notation: denoting $\Lebone$ the Lebesgue measure $\Lebone|_{(0,T)}$, a measurable family $\nu = (\nu_t)_{t\in [0,T]} \in \Ms(Y)$, with $Y\in \{ V,E\}$, induces a measure in $\Ms([0,T]{\times} Y)$ by setting 
\begin{equation}
\label{Lebintegrated}
\nu_{\!\Lebone} (\dd t \dd x \dd y): = \nu_t(\dd x \dd y) \,\Lebone (\dd t).
\end{equation}
\begin{theorem}[Convergence to $\Dissipative$ solutions]
\label{th:main-0} 
 Under Assumptions  \ref{Ass:E},  \ref{Ass:D}, \ref{Ass:flux-density} \& \ref{ass:dupphi}, consider a family of kernels $(\kappa_n)_n$ defined from functions $(a_n)_n$ as in \eqref{eq:an-1} \& \eqref{kappa-eps-reg-n}.
 Moreover, suppose that  the kernels satisfy the boundedness condition \eqref{kan-back-2-bounded-kernel}. 
 
 Further, let $(\rho_0^n)_n, \rho_0 \in \dom(\scrE)$ satisfy 
 \begin{equation}
  \label{initial-thm:main}
 \begin{cases}
 &
 \rho_0^n \to \rho_0 \text{ setwise in } \calM^+(V) \text{ as } n \to \infty\,,
 \\
 & 
 \scrE(\rho_0^n)  \to \scrE(\rho_0)  \text{ as } n \to \infty\,.
 \end{cases}
 \end{equation}
 Let $(\rho^{n},\bj^{n})_n$ be $\Balanced$ solutions to the $(\scrE,\scrR^{n},(\scrR^{n})^*)$ systems emanating from $( \rho_0^n)_n$. 
 \par
 Then,
  there exist $(\rho,\bj) \in \CE 0T$ and (a not relabeled) subsequence such that the following convergences hold as $n\to\infty$
  \begin{subequations}
  \label{cvg-thm1}
\begin{align}
\label{DBL}
&
\rho_t^{n}\to \rho_t  && \text{setwise in } \calM^+(V) \quad \text{for all } t\in [0,T];
\\
& 
\label{cvg-j}
\bj_{\!\Lebone}^{n}\to \bj_{\!\Lebone}  =\bj_t \Lebone   &&  
\text{$\sigma$-setwise in } \Ms([0,T]{\times}E);
\end{align}
\end{subequations}
the measure $\rho$ satisfies
$\rho_t  = u_t \pi$  for all $t\in [0,T]$, 
for some $u\in \rmL^1(0,T;\rmL^1(V;\pi))$ and the pair
$(\rho,\bj)$ is a  $\Dissipative$ solution \EEE of the $(\scrE,\scrR,\scrR^*)$ evolution system.
\end{theorem}

The proof of Theorem~\ref{th:main-0} is a straightforward consequence of the robustness result stated in Theorem \ref{thm:robustness} ahead. 
\begin{remark}
\label{rmk:atomic-vs-convexII}
The convexity of $\mathrm{D}_\upphi$ will play  a pivotal role in the proof of Theorem~\ref{th:main-0}. Nonetheless, let us observe that 
it could be replaced by 
 the  condition that $\pi$ is atomic.
\end{remark}
 \par
 Let us now additionally assume a uniform upper bound on the  initial densities for the approximate systems, which in particular ensures that the approximate $\Balanced$ solutions $(\rho^{n},\jj^{n})_n$ obtained in \cite{PRST22} are also $\Reflecting$ (cf.\ Proposition \ref{prop:fromPRST}). With our second main result we will prove their convergence to a $\Reflecting$ solution of the $(\scrE,\scrR,\scrR^*)$ system.
 

\begin{theorem}[Convergence to a $\Balanced$ solution]
\label{th:main}
Under the assumptions of Theorem \ref{th:main-0}, suppose moreover that  $(\rho_0^n)_n \subset  \dom(\scrE)$ satisfy
  \begin{equation}
  \label{strong-initial}
  \exists\, \overline{U}>0  \quad \forall\, n \in \N \text{ and  for $\pi$-a.e.\ $x\in V$:} \qquad 
  0 
 \leq   u_0^n(x) \leq \overline{U}\,.
\end{equation}
Let $(\rho^n,\bj^n)_n $ be $\Reflecting$ solutions of the $(\scrE,\scrR^n,(\scrR^n)^*)$ systems, with initial data $(\rho_0^n)_n$.
\par
Then,  convergences \eqref{cvg-thm1} hold, up to a subsequence, to  a  pair
$(\rho,\bj)  \in \CE0T$ such that 
\begin{enumerate}
\item $(\rho,\bj)$ is a $\Reflecting$ solution (the \underline{unique} $\Reflecting$ solution, if $\upphi$ is strictly convex), of the $(\scrE,\scrR,\scrR^*)$ evolution system and,
  in addition, 
\item it satisfies the energy-dissipation balance in the pointwise form \eqref{enh-EDB}.
 \end{enumerate}
Furthermore, we have the \emph{enhanced} convergences
\begin{subequations}
\label{enh-cvg-thm4.3}
\begin{gather}
\int_s^t  \scrR^n(\rho_r^n, \bj_r^n)\, \dd r \longrightarrow \int_s^t  \scrR(\rho_r, \bj_r)\, \dd r \,,
\\
\int_s^t 
 \Fish^n(\rho_r^n)\, \dd r  \longrightarrow \int_s^t   \Fish(\rho_r)\, \dd r \,,
 \\
  \calS(\rho_t^n)  \longrightarrow   \calS(\rho_t) 
\end{gather}
\end{subequations}
for all $0\leq s \leq t \leq T$. 
\end{theorem}

The proof of Theorem~\ref{th:main} will be again directly deduced from the ensuing Theorem \ref{thm:robustness}.

\subsection{Robustness of Dissipative and  Reflecting solutions}
\label{ss:robustness}
As already anticipated, 
Theorems \ref{th:main-0} \& \ref{th:main} will be a consequence of the following  more general result, addressing the 
robustness of the notions of $\Dissipative$ and $\Reflecting$ solutions 
when the singular kernel $\kappa$ is approximated by a family  $(\kappa_n)_n$,  related to $\kappa$ via \eqref{eq:an-1}
\eqref{kappa-eps-reg-n}. 
\par
Let us emphasize that, in Theorem  \ref{thm:robustness} ahead, we will drop the requirement that the kernels $\kappa_n$ be bounded in the sense of 
\eqref{kan-back-2-bounded-kernel}: they will only satisfy \eqref{mitigated-bounde}, which again allows for the `tamed' singularity of $\kappa$. In fact, 
\eqref{kan-back-2-bounded-kernel} was required  in Section 
\ref{ss:4.2} only to ensure the applicability of the existence result from \cite{PRST22}, and thus the existence of $\Balanced$ (and a fortiori $\Reflecting$) solutions to the
$(\scrE,\scrR^n,(\scrR^n)^*)$ systems. 
In the more general setup of  \eqref{eq:an-1} \& \eqref{kappa-eps-reg-n}, we are nonetheless in a position to prove the following convergence result for a family of $\Dissipative$/$\Reflecting$ solutions (whose existence is not addressed) to the $(\scrE,\scrR^n,(\scrR^n)^*)$ system.

\begin{theorem}
\label{thm:robustness}
Under Assumptions \ref{Ass:E},  \ref{Ass:D}, \ref{Ass:flux-density} \& \ref{ass:dupphi}, consider a family of kernels $(\kappa_n)_n$ defined from functions $(a_n)_n$ as in \eqref{eq:an-1} \& \eqref{kappa-eps-reg-n}.
 
 Let $(\rho_0^n)_n, \rho_0 \in \dom(\scrE)$ satisfy the boundedness and well-preparedness condition \eqref{initial-thm:main}.
\begin{enumerate}
\item Let $(\rho^{n},\bj^{n})_n$ be $\Dissipative$ solutions to the $(\scrE,\scrR^{n},(\scrR^{n})^*)$ systems emanating from $( \rho_0^n)_n$. 
 \par
 Then, there exist a (not relabeled) subsequence, and a curve $(\rho,\jj) \in \ACE0T$, such that convergences \eqref{DBL}
 hold, and $(\rho,\jj)$ is a  $\Dissipative$ solution to the $(\scrE,\scrR,\scrR^*)$ system.
 \item 
 Suppose that  $(\rho^{n},\bj^{n})_n$ are $\Reflecting$ solutions to the $(\scrE,\scrR^{n},(\scrR^{n})^*)$ systems emanating from data $( \rho_0^n)_n$ such that, in addition, 
\eqref{strong-initial} holds. Moreover, suppose that
\begin{equation}
\label{eq:cb2-required}
\sup_{n\in \N} \sup_{t\in [0,T]} \|u_t^n\|_{\rmL^\infty(\pi)} <\infty
\end{equation}
Then, there exist a (not relabeled) subsequence and a pair $(\rho,\jj) \in \ACE0T$, such that convergences \eqref{DBL} \& \eqref{enh-cvg-thm4.3}
 hold, and $(\rho,\jj)$ is a $\Reflecting$ solution to the $(\scrE,\scrR,\scrR^*)$ system.
 \par
 Finally, if $\upphi$ is strictly convex, then the above convergences hold for the entire sequence $(\rho^{n},\bj^{n})_n$.
\end{enumerate}
\end{theorem}
The proof of Theorem~\ref{thm:robustness} will be carried out in Section \ref{s:5}.

\medskip
Whenever the approximating kernels $(\kappa_n)_n$ are bounded, \eqref{eq:cb2-required} is a redundant requirement. In fact, the maximum principle proven in \cite{PRST22} ensures that the uniform bound \eqref{strong-initial} for the initial densities $(u_0^n)_n$ transfers along the flow to the densities $(u_t^n)_n$ of $\Balanced$ ($=\Reflecting$) solutions. Nonetheless, we emphasize that the kernels $(\kappa_n)_n $ in the statement of Theorem~\ref{thm:robustness} need not be bounded. Therefore, \eqref{eq:cb2-required} has to be assumed in the general case.

However, for strictly convex $\upphi$, assumption \eqref{eq:cb2-required} may be alleviated via a maximum principle for the reflecting solutions themselves.
\begin{cor}[Maximum principle]
\label{cor:max}
Under the assumptions of Theorem \ref{thm:robustness} and strict convexity of $\upphi$, the maximum principle holds for $(u_t)_{t\in [0,T]}$ for any $\Reflecting$ solution $(\rho,\jj)$. In particular, condition \eqref{eq:cb2-required} is redundant.
\end{cor}
\begin{proof}
Suppose that $u_0\leq M$, $\pi$-a.e., i.e., $\|u_0\|_{\rmL^\infty(\pi)}\leq M$, and consider the sequence of approximating solutions $(\rho^n)_{n\in \N}$ of Theorem \ref{th:main-0} corresponding to \emph{bounded} kernels, with same initial data. By the maximum principle proven in \cite{PRST22}, we find that $u_t^n \leq M$ for all $t\in [0,T]$, $\pi$-a.e.\ in $V$. By uniqueness, we obtain via Theorem \ref{thm:robustness} the convergences $u_t^n\rightharpoonup u_t$ weakly in $\rmL^1(\pi)$ for all $t\in [0,T]$. Consequently, by lower semicontinuity, we deduce that $\|u_t\|_{\rmL^\infty(\pi)} \leq M$ for all $t\in [0,T]$ as well. 
\end{proof}

\EEE
\section{An application in  configuration spaces}
\label{s:8}
In this section we introduce a possible setting for the application of our results to the (singular) dynamics of \emph{weakly} interacting particle systems. Indeed, the systems we focus on, which arise in models of theoretical biology and ecology involving birth, death, and mutation processes, are such that the interaction of a single particle with the other particles in a \emph{trait space} $\trait$ only depends on the empirical measure, a rescaled discrete measure describing the position of all particles. The dynamics of these systems consist of particles being created, dispersed, or annihilated, and can be described in terms of jump processes on the space of probability measures on $\trait$.

\par
The variational and gradient flow representation for the law of these processes has been discussed, under suitable regularity assumptions, including boundedness,  on the measure-dependent birth, death, and jump kernels, in \cite{HT2023,HHT2024}
\par
In this section, we will demonstrate how the results of our current work can apply to \emph{singular} jump kernels, on the configuration space $V$ over a generic metric space $\trait$, that are defined via a suitable lifting procedure. In particular, we will show that the singular kernel given in \eqref{widehat-kappa} below does comply with \eqref{mitigation of singularity}.

\medskip
\par
To set forth our construction, let us consider 
\begin{itemize}
\item[-] a separable metric measure space $(\trait, \dV)$, also satisfying the \emph{Radon property} (i.e., every finite Borel measure on $\trait$ is inner regular), endowed with a reference measure $\pi$;
\item[-] $(\kappa(x,\cdot))_{x \in \trait} \subset \Ms^+(\trait)$ a Borel family of measures satisfying the assumptions of Section \ref{s:3}: Without loss of generality $\kappa(x,\{x\})=0$ for every $x\in \trait$, $\kappa$ satisfies detailed balance \eqref{DBC} with respect to $\pi$, and both 
\begin{equation}\label{true-measurabilityc}
\mbox{the map $\displaystyle \trait\ni x \mapsto \int_{\trait\backslash\{x\}} f(y) (1{\wedge}\dV^2(x,y)) \, \kappa (x,\dd y)$ is measurable}
 \end{equation}
 for any $f\in \Bb(\trait)$, and 
 \begin{equation}
\label{mitigation of singularityc}
\sup_{x\in \trait}\int_{\trait}  (1{\wedge}\dV^2(x,y)) \,\kappa(x,\dd y) =: c_\kappa <+\infty.
\end{equation}
\end{itemize}
These objects can then be lifted to the space $\boldsymbol{V}:=  \mathcal{P}(\trait) $ in the following sense (in what  follows, we will use $\boldsymbol{boldsymbol}$ for the `lifted objects' and $\,\widehat{\cdot}\,$ to denote functionals/measures over them):
\begin{itemize}
\item[--] For $N\geq 2$,  we introduce the mapping
\[
\Lambda_N : \trait^N \to  \mathcal{P}(\trait), \qquad  \boldsymbol{\mathfrak{z}} = (\mathfrak{z}_1,\mathfrak{z}_2,\ldots, \mathfrak{z}_N) \mapsto \mu = \Lambda_N(\boldsymbol{\mathfrak{z}})
:= \frac1N \sum_{i=1}^N \delta_{\mathfrak{z}_i};
\]
\item[-] On $\boldsymbol{V}= \mathcal{P}(\trait)$ we consider the finite measure 
\[
\widehat{\boldsymbol{\pi}}: = (\Lambda_N)_{\#} \pi^N \qquad \text{with } \pi^N \text{ the product measure } \pi{\otimes}\ldots{\otimes}\pi \text{ on } \trait^N\,,
\]
where the above push-forward measure is defined by 
$(\Lambda_N)_{\#} \pi^N  (A) : = \pi^N (\Lambda_N^{-1}(A))$ for any Borel subset $A \subset V$. 
\item[-] For every $\nu \in \boldsymbol{V} $ we define $\widehat{\boldsymbol{\kappa}}(\nu,\cdot)$ as the Borel measure
 \begin{equation}
 \label{widehat-kappa}
\widehat{\boldsymbol{\kappa}}(\nu,\dd \eta) : =
\begin{cases}\displaystyle
 \int_\trait \left(  \int_{\trait}  \delta_{\nu{+}\tfrac1N (\delta_y {-}\delta_x)} (\dd \eta) \,\kappa(x, \dd y) 
 \right)  \nu(\dd x) 
& \text{if } \nu \in \Lambda_N(\trait^N),
 \\
 0 & \text{otherwise}\,.
 \end{cases}
\end{equation}
\end{itemize}
The measure $\widehat{\boldsymbol{\kappa}}(\nu,\cdot)$ is non-trivial if and only if $\nu$ is itself an empirical measure,
 i.e., 
 \begin{equation}
 \label{nu-empirical}
 \nu = \frac1N \sum_{i=1}^N \delta_{\mathfrak{z}_i} \qquad \text{ for some $\mathfrak{z}_1,\ldots, \mathfrak{z}_N \in \trait$.}
 \end{equation}
  In that case, 
 \begin{equation}
 \label{explicit-representation}
 \begin{aligned}
 \widehat{\boldsymbol{\kappa}}(\nu,\dd \eta)  
   & = \frac1N \sum_{i=1}^N   \int_{\trait{\backslash} \{\mathfrak{z}_i\}}  \delta_{\tfrac1N \sum_{j \in \{1,\ldots, N\} {\backslash}\{ i\}} \delta_{\mathfrak{z}_j} {+}\tfrac1N \delta_y} (\dd \eta) \, \kappa(\mathfrak{z}_i, \dd y) 
 \end{aligned}
 \end{equation}
\par
We now equip $V = \mathcal{P}(\trait)$ with the Wasserstein distance $W_2$ induced by $\dV$: For any $\mu_1,\, \mu_2 \in \mathcal{P}(\trait)$,  
 \begin{equation}
 \label{Wasserstein2}
 \begin{gathered}
 W_2^2(\mu_1,\mu_2) := \min \left\{ \iint_{\trait{\times}\trait} \dV^2(x,y)  \, \boldsymbol{\gamma}(\dd x \dd y )\, : \  \boldsymbol{\gamma} \in \Gamma (\mu_1,\mu_2)\right\}\qquad \text{with }
 \\
  \Gamma (\mu_1,\mu_2) = \{  \boldsymbol{\gamma}\in  \mathcal{P}(\trait{\times}\trait)\, : \ (\mathsf{p}_i)_{\#} \boldsymbol{\gamma} = \mu_i \text{ for } i=1,2\}\,
  \end{gathered}
 \end{equation}
 where $ \mathsf{p}_i: \trait \times\trait \to \trait$ denotes the projection on the $i$th-component. For notational consistency,  hereafter 
 we will write $\boldsymbol{\dV}$ in place of $W_2$.

 \begin{prop}
 The collection $(\boldsymbol{V}, \boldsymbol{\dV} ,\widehat{\boldsymbol{\pi}},\widehat{\boldsymbol{\kappa}})$ satisfies the assumptions of Section \ref{s:3}. Moreover, 
 \begin{equation}\label{mitigation of singularity-lift}
\sup_{\nu\in \boldsymbol{V}}\int_{\boldsymbol{V}} (1{\wedge} 
 \boldsymbol{\dV}^2(\nu,\eta)) \,\widehat{\boldsymbol{\kappa}}(\nu,\dd \eta) \leq  c_{\kappa}.
\end{equation}
 \end{prop} 
 \begin{proof}
 It follows from \cite[Prop.\ 7.1.5]{AGS08} that  the metric space $(\boldsymbol{V}, \boldsymbol{\dV})$ is separable. In order to check \eqref{mitigation of singularity-lift}, let us preliminarily introduce the following subset of $ \mathcal{P}(\trait){\times}\mathcal{P}(\trait)$: 
 \[
 \begin{aligned}
 &
(\nu,\eta) \in 
 \mathfrak{S}_N \quad \text{if and only if }
 \\
 &    \begin{cases}
 \nu = \frac1N  \sum_{i=1}^N \delta_{\mathfrak{z}_i} &\text{for  } \mathfrak{z}_1,\ldots, \mathfrak{z}_N \in \trait,
 \\
 \eta = \eta_{\mathfrak{y}}^i = \frac1N  \sum_{j \in \{1,\ldots, N\} {\backslash}\{ i\}}  \delta_{\mathfrak{z}_j} +  \frac1N \delta_{\mathfrak{y}}  & 
\text{for some } i =1,\ldots, N \text{ and } {\mathfrak{y}} \in \trait {\backslash}\{ \mathfrak{z}_1,\ldots, \mathfrak{z}_N\}.
 \end{cases}
 \end{aligned}
 \]
 We have the key estimate
 \begin{equation}
 \label{key-est}
  \boldsymbol{\dV}^2(\nu,\eta) \leq 
 \frac1N  \dV^2(\mathfrak{z}_i,\mathfrak{y})
  \qquad \text{for all }   (\nu,\eta) \in 
 \mathfrak{S}_N \,.
 \end{equation} 
Indeed,  it is sufficient to choose as competitor in the formula for $ W_2(\nu,\eta)$ the plan given by 
\[
\boldsymbol{\gamma}(\dd x \dd y) := \frac1N \sum_{j \in \{1,\ldots, N\} {\backslash}\{ i\}}  \delta_{\mathfrak{z}_j}(
\dd x) \delta_{\mathfrak{z}_j}(
\dd y)  + \frac1N  \delta_{\mathfrak{z}_i} (\dd x)  \delta_{\mathfrak{y}}(\dd y)\,.
\]
Then, we have 
 \[
 \begin{aligned}
  \boldsymbol{\dV}^2(\nu,\eta) &  \leq \frac1N  \sum_{j \in \{1,\ldots, N\} {\backslash}\{ i\}}  \iint_{\trait{\times}\trait} \dV^2(x,y) \,  \delta_{\mathfrak{z}_j}(
\dd x) \delta_{\mathfrak{z}_j} (\dd y) 
 + \frac1N   \iint_{\trait{\times}\trait} \dV^2(x,y)  \, \delta_{\mathfrak{z}_i} (\dd x)  \delta_{\mathfrak{y}}(\dd y)
\\
&=  \frac1N  
\dV^2(\mathfrak{z}_i, \mathfrak{y}) 
\end{aligned}
 \]
Estimate \eqref{key-est} and the assumed \eqref{mitigation of singularityc} imply that, for every $\nu \in \Lambda_N(\trait^N)$, there holds
\[
\begin{aligned}
&\int_{\boldsymbol{V}} 
(1{\wedge} 
\boldsymbol{\dV}^2(\nu,\eta)) \,\widehat{\boldsymbol{\kappa}}(\nu,\dd \eta)
\\
&\qquad = \frac1N  
 \sum_{i=1}^N   \int_{\trait{\backslash} \{\mathfrak{z}_i\}} \left( 
  \int_{\boldsymbol{V}} 
 (1{\wedge} \boldsymbol{\dV}^2(\nu,\eta))\,
  \delta_{\tfrac1N \sum_{j \in \{1,\ldots, N\} {\backslash}\{ i\}} \, \delta_{\mathfrak{z}_j} {+}\tfrac1N \delta_y} (\dd \eta) \right) \kappa(\mathfrak{z}_i, \dd  y) 
  \\
  &\qquad\leq \frac1N 
   \sum_{i=1}^N 
   \int_{\trait{\backslash} \{\mathfrak{z}_i\}} (1{\wedge}\dV^2(\mathfrak{z}_i, y))  \, \kappa(\mathfrak{z}_i, \dd  y) 
   \leq c_{\kappa}\,,
  \end{aligned}
\]
establishing the bound \eqref{mitigation of singularity-lift}.

It remains to prove the measurability property \eqref{true-measurability}. Having obtained \eqref{mitigation of singularity-lift}, we may now resort to Lemma \ref{lm:kmeas}, which guarantees that it is equivalent to prove that for every $H \in \B^+(\boldsymbol{E}; \overline{\R}) $ (with $\boldsymbol{E}  = \boldsymbol{V}{\times} \boldsymbol{V}$), the map 
\begin{equation}
\label{measurab-config}
 \boldsymbol{V}\ni \nu \mapsto \int_{\boldsymbol{V}} H(\nu,\eta)\, \widehat{\boldsymbol{\kappa}} (\nu,\dd \eta)  \quad \text{
is in $\B^+( \boldsymbol{V}; \overline{\R})$}.
\end{equation}
Now, for all $\nu  = \frac1N \sum_{i=1}^N \delta_{\mathfrak{z}_i}   \in  \Lambda_N(\trait^N) $ we have 
\[
 \int_{\boldsymbol{V}} H(\nu,\eta)\, \widehat{\boldsymbol{\kappa}} (\nu,\dd \eta)  = 
  \sum_{i=1}^N   \int_{\trait{\backslash} \{\mathfrak{z}_i\}} H(  \mathfrak{z}_i, \eta) \,  \delta_{\tfrac1N \sum_{j \in \{1,\ldots, N\} {\backslash}\{ i\}} \delta_{\mathfrak{z}_j} {+}\tfrac1N \delta_y} (\dd \eta) \, \kappa(\mathfrak{z}_i, \dd y) 
  \]
and exploiting this representation, and the fact that  $\{\kappa(x,\cdot)\}$ is a Borel family, 
 we may prove  \eqref{measurab-config}.

 Finally, the reversibility \eqref{DBC} follows from the same calculation as \cite[Lemma 3.6]{HHT2024}.
 \end{proof}

\par
In the functional-analytic setup of $(V =\mathcal{P}(\trait) ,W_2,\widehat{\boldsymbol{\pi}},\widehat{\boldsymbol{\kappa}})$,  we may now consider triples $(\upphi, \uppsi, \upalpha)$ of entropy, dissipation, and flux densities  complying with Assumptions \ref{Ass:E}, \ref{Ass:D}, and \ref{Ass:flux-density}, respectively, (cf.\ Example \ref{admissible-triples-intro} for concrete choices of $(\upphi, \uppsi, \upalpha)$), and define the entropy functional
\[
	\mathcal{M}^+(V) \ni \rho\mapsto \widehat{\scrE}(\rrho) = 
 \int_{\mathcal{P}(\trait)} \upphi\left(\frac{\dd\rrho}{\dd \widehat{\boldsymbol{\pi}}} \right)\, \dd \widehat{\boldsymbol{\pi}} \qquad \text{if } \rrho \ll \widehat{\boldsymbol{\pi}}
\]
with $\widehat{\scrE}(\rrho) = +\infty$ otherwise. Likewise, we may introduce the primal and dual dissipation potentials $\widehat{\scrR}$ and $\widehat{\scrR}^*$ via Definition \ref{def:primal-and-dual}. The results from Sections \ref{s:3-NEW} and \ref{s:4} then apply to the evolution system $(\widehat{\scrE},\widehat{\scrR}, \widehat{\scrR}^*)$: in particular, there exist $\Dissipative$ and $\Reflecting$ solutions to the corresponding Forward Kolmogorov equation \eqref{FKEm}.

\section{Proof of Theorem \ref{thm:robustness}}
\label{s:5}
Recall that the kernels $\kappa_n(x, \dd y) = a_n(x,y) \kappa (x, \dd y)$ involve functions
$(a_n)_n$ satisfying \eqref{eq:an-1}.   In addition to the couplings
$(\tetapin)_n$ from \eqref{tetapin}, we introduce the couplings associated with
$(\kappa_n)_n$ and a given  sequence $(\rho^n)_n\subset \mathcal{M}^+(V)$ via
\begin{equation}
\label{kan-def:tetarho}
\begin{aligned}
  \teta_{\!\rho^n}^-(\dd x\,\dd y) := \rho^n(\dd x)\kappa_n(x,\dd y), \qquad
  \teta_{\!\rho^n}^+(\dd x\,\dd y) := \rho^n(\dd y)\kappa_n(y,\dd x)= s_{\#}\teta_{\!\rho^n}^-(\dd x\,\dd y)
\end{aligned}
\end{equation}
(with slight abuse of notation). 

\medskip
With our first result, which is the analogue of Lemma \ref{l:3.4}, we are going to check that the couplings $(\tetapin)_n$ and 
$ \teta_{\!\rho^n}^{\pm}$ (associated with a  converging sequence $(\rho^n)_n$), converge appropriately. 

\begin{lemma}
\label{eq:knass}
Let the kernels $(\kappa_n)_n$ satisfy \eqref{eq:an-1} \& \eqref{kappa-eps-reg-n}. Let us also consider
a sequence $(\rho^{n})_{n}\subset \dom(\scrE)$ with $\rho^{n}\ll \pi$ for all $n \in \N$, 
and $\rho^n \to \rho$ setwise in $\ \calM^+(V)$. 
 \par
 Then,  $(\kappa_n)_n$ is a measurable family in the sense of \eqref{ass:kappa}, and 
 the following convergences hold:
\begin{equation}
\label{not-negligible-n}
\begin{cases}
(1{\wedge} \dV^2) \tetapin \to (1{\wedge} \dV^2) \tetapi
\\
(1{\wedge} \dV^2) \teta_{\!\rho^{n}}^\pm \to (1{\wedge} \dV^2) \teta^{\pm}_{\!\rho} 
 \end{cases}
\qquad   \text{setwise in $\calM^+(E)$ as $n\to \infty$.}
\end{equation}
In particular, 
\begin{equation}
\label{not-negligible-2n}
\begin{cases}
\tetapin \to \tetapi
\\
 \teta_{\!\rho^n}^\pm \to \teta_{\!\rho} 
 \end{cases}
\qquad   \text{$\sigma$-setwise in $\Ms^+(E)$ as $n\to \infty$.}
\end{equation}
\end{lemma}
\begin{proof}
As for the measurability assertion, notice that the property that 
\begin{equation*}
\mbox{the map $\displaystyle V\ni x \mapsto \int_{V} f(y) (1{\wedge}\dV^2(x,y)) \, \kappa (x,\dd y)$ is measurable for any $f\in \Bb(V)$,}
 \end{equation*}
implies that 
\begin{equation*}
\mbox{$\displaystyle V\ni x \mapsto \int_{V} f(y) a_n(x,y) (1{\wedge}\dV^2(x,y)) \, \kappa (x,\dd y)$}
 \end{equation*}
is measurable, by choosing $g_n(x,y)=f(y)a_n(x,y)$ in Lemma \ref{lm:kmeas}.
\par
It is obviously sufficient to prove 
\eqref{not-negligible-n}
for, e.g., the sequence $(\teta_{\!\rho_n}^-)_n$. With this aim,  as in the proof of Lemma \ref{l:3.4}, 
let us fix  $\phi\in \Bb(E)$, and observe that for the  bounded Borel mappings
 \[
 \begin{cases}
 \displaystyle
 x\mapsto f_n(x) :=  \int_V \phi(x,y)(1{\wedge} \dV^2(x,y))\kappa_n(x,\dd y)
  =  \int_V a_n(x,y) \phi(x,y)(1{\wedge} \dV^2(x,y))\kappa(x,\dd y)
  \\
 \displaystyle
 x\mapsto f(x) :=  \int_V \phi(x,y)(1{\wedge} \dV^2(x,y))\kappa(x,\dd y)
  \end{cases}  
  \]
there holds $f_n \to f$ pointwise in $V$ $\pi$-almost everywhere  as $n\to \infty$, with 
\[
|f_n(x) | \leq \int_V |\phi(x,y)|(1{\wedge} \dV^2(x,y))\kappa(x,\dd y) \qquad \text{for $\pi$-a.e.\ $x\in V$.}
\] 
Then, we may combine this with the setwise convergence of $(\rho^n)_n$ thanks to  \eqref{eq:pointset}, and conclude that 
$\langle f_n, \rho^n\rangle \to \langle f, \rho\rangle$. Namely,
\[
\begin{aligned}
&
\iint_{E} \phi(x,y) \, (1\wedge \dV^2(x,y)) \teta_{\!\rho^n}^- (\dd x \dd y) = \int_{V} \left(  \int_{V} a_n(x,y)\phi(x,y) \,(1\wedge \dV^2(x,y))\kappa(x,\dd y) \right) \, \rho^n(\dd x) 
\\
&
\longrightarrow  \int_{V} \left(  \int_{V} \phi(x,y) \,(1\wedge \dV^2(x,y))\kappa(x,\dd y) \right) \, \rho(\dd x)  = \iint_{E} \phi(x,y) \, (1\wedge \dV^2(x,y))\teta_{\!\rho}^- (\dd x \dd y) \,.
\end{aligned}
\]
and \eqref{not-negligible-n} follows. A fortiori we have \eqref{not-negligible-2n} by the same argument as in the proof of Lemma \ref{l:3.4}. 
\end{proof}

The core argument for the proof of the robustness Theorem \ref{thm:robustness} will be 
developed in the upcoming Section \ref{ss:compactness}, dealing with the compactness properties of a 
sequence $(\rho^{n},\bj^{n})_n  $ with 
\[
(\rho^n,\jj^n) \in \ACEn0T  \qquad \text{for every $n\in \N$},
\]
i.e., with bounded action and entropy for the $(\scrE,\scrR^{n},(\scrR^{n})^*)$ system. In Section \ref{ss:5.2}, we will take $(\rho^n,\jj^n)_n$ to be a sequence of $\Dissipative$/$\Reflecting$ solutions, as in the statement of Theorem  \ref{thm:robustness}, and then conclude its proof.

\subsection{Compactness}
\label{ss:compactness}
Consider a sequence $(\rho^{n},\bj^n) \in \ACEn 0T$. 
We will use that
\begin{equation}
\label{absolute-continuity-n}
\rho^{n} = u^{n} \pi \quad 
\text{ and  } \quad  2\bj^{n} = w^{n} \tetapin 
\end{equation}
 (since $\rho_t^{n} \ll \pi$ and $\bj_t^{n} \ll \tetapin$
for $\Lebone$-a.e.\ $t\in (0,T)$, by 
 \eqref{nice-representation})
for some $(u^{n})_n  \subset
\rmL^1(0,T; \rmL^1(V;\pi)) $ and  \ $(w^{n})_n \subset \rmL^1(0,T; \rmL^1(E,\tetapin)) $.

For later convenience we remark that the action integrals $ \int_0^T \scrR^{n}(\rho^{n}, \bj^{n}) \dd t $ rewrite
taking into account that, by  \eqref{nice-representation}
 we have for $\Lebone$-a.e.\ $t\in (0,T)$
 \begin{equation}
 \label{def-set-Eeps}
  |\bj_t^{n} |(E{\backslash} E_t^{n}) =0 \qquad \text{with }  E_t^{n} =  \{ (x,y)\in E\, : \ \upalpha(u_t^{n}(x),u_t^{n}(y))>0  \},
 \end{equation}
 and hence for $\Lebone$-a.e.\ $t\in (0,T)$
 \begin{equation}
 \label{starting-point-later}
 \scrR^n(\rho_t^{n},\bj_t^{n})  = 
    \displaystyle
    \frac12\iint_{E_t^{n}}
    \uppsi\Bigl(\frac{w_t^{n}(x,y)}{\upalpha(u_t^{n}(x),u_t^{n}(y))}\Bigr)\upalpha(u_t^{n}(x),u_t^{n}(y))\,\tetapin(\dd
    x,\dd y)\,. \EEE 
 \end{equation}

We are now in a position to state the main result of this section, collecting the compactness properties of the sequence $(\rho^n,j^n)_n$; in the second part of the statement,
we will additionally suppose that the pairs $(\rho^n,j^n)$ satisfy the reflecting continuity equation for the systems $(\scrE,\scrR^n,(\scrR^n)^*)$ and accordingly write
 $(\rho^n,j^n)\in \ECEn 0T$.
\begin{prop}[Compactness]
\label{prop:compactness2}
Consider a sequence $(\rho^n,j^n)\in \ACEn 0T$ such that 
$(\rho^n(0))_n$ complies with \eqref{initial-thm:main}, and satisfying
\begin{equation}
\label{eq:compa1}
\begin{aligned}
\sup_{n\in \N} \sup_{t\in [0,T]} \scrE(\rho^n_t) <\infty,\ \quad \sup_{n\in \N} \int_0^T  \scrR^n(\rho^n_t,\bj^n_t)\, \dd t <\infty.
\end{aligned}
\end{equation}
Then,
\begin{enumerate}
\item
 there exist  
a curve $(\rho,j) \in \ACE 0T$ 
(i.e., with finite entropy and action for the system $(\scrE,\scrR,\scrR^*)$ associated with the kernel $\kappa$),
such that, along a (not relabeled) subsequence, 
\begin{subequations}
  \label{cvg-thm1mod}
\begin{align}
&
\rho_t^{n}\to \rho_t  && \text{setwise in } \calM^+(V) \quad \text{for all } t\in [0,T];
\\
& 
(1{\wedge} \dV )\bj_{\!\Lebone}^{n}\to (1{\wedge} \dV )\bj_{\!\Lebone}  =(1{\wedge} \dV )\bj_t \Lebone   && 
\text{setwise in } \calM([0,T]{\times}E)
\end{align}
(cf.\ \eqref{Lebintegrated} for the notation $\bj_{\!\Lebone}^{n}$, $\bj_{\!\Lebone}$). 
\end{subequations}
In particular, 
\[\bj_{\!\Lebone}^{n}\to \bj_{\!\Lebone} \quad 
\text{$\sigma$-setwise in } \Ms([0,T]{\times}E).\]
\item
If, in addition, $(\rho^n,j^n)\in \ECEn 0T$ for all $n\in \N$ and the densities $u^n = \frac{\dd \rho^n}{\dd \pi}$ satisfy
\begin{equation}
\label{eq:cb2}
\sup_{n\in \N} \sup_{t\in [0,T]} \|u_t^n\|_{\rmL^\infty(\pi)} <\infty
\end{equation}
then there exists a curve $(\rho,j)\in \ECE 0T$ such that, along a (not relabeled) subsequence, convergences \eqref{cvg-thm1mod} hold and, moreover,
\begin{equation}
\label{eq:comce1}
\begin{aligned}
\lim_{n\to \infty} 
\iiint_{[0,T]\times E} \xi(x,y) \, \dd \bj_{\!\Lebone}^{n}  (\dd x \dd y \dd t) & = \iiint_{[0,T]\times E} \xi(x,y) \,  \bj_{\!\Lebone} (\dd x \dd y \dd t) 
\\
& \quad \mbox{for every $\xi \in \Bb(E)$ with} \iint_{E}\xi^2 \, \dd \tetapi < \infty.
\end{aligned}
\end{equation}
\end{enumerate}
\end{prop}

Preliminarily, for the reader's convenience we recall a refined version of the Ascoli-Arzel\`a theorem in metric spaces that will be used in the proof of Proposition \ref{prop:compactness2}.
\begin{prop}{\cite[Prop.\ 3.3.1]{AGS08}}
\label{thm:Ascoli}
Let $(\mathscr{S}, d)$ be a complete  metric space, also endowed with a topology $\serifsigma$ compatible with $d$ in the sense that
for all $(x_n)_n,\, (y_n)_n \subset \mathscr{S}$ there holds
\begin{equation}
\label{compatibility-top-metr}
(x_n,y_n) \stackrel{\serifsigma}{\longrightarrow} (x,y) \ \Longrightarrow  \ \liminf_{n\to\infty} d(x_n,y_n) \geq d(x,y)\,. 
\end{equation}
Let $\mathsf{K}$ be a $\serifsigma$-sequentially compact subset of  $\mathscr{S}$, and let
$(v_n)_n$
be a sequence of curves  $v_n : [0,T]\to \mathscr{S}  $ such that 
\begin{subequations}
\label{conditions-for-Ascoli}
\begin{align}
\label{compactness}
& v_n(t) \in \mathsf{K} && \text{for all } t \in [0,T], \ n \in \N,
\\
& 
\label{equicontinuity}
\limsup_{n\to\infty} d(v_n(t),v_n(s)) \leq \omega(s,t) && \text{for all } s,\, t \in [0,T],
\end{align}
\end{subequations}
where  $\omega : [0,T]{\times}[0,T]\to [0,\infty)$ is a   symmetric function for which there exists an (at most) countable subset $\mathscr{C}$ of $[0,T]$ such that
\[
\lim_{(s,t) \to (r,r)}\omega(s,t) = 0 \quad \text{for all } r\in [0,T]{\backslash}\mathscr{C}.
\]
Then, there exist an increasing subsequence $(n_k)_k$ and a limit curve $v: [0,T]\to \mathscr{S}$ such that 
\[
v_{n_k}(t) \stackrel{\serifsigma}{\to} v(t) \text{ for all } t \in [0,T], \ \text{ and } \ v: [0,T]\to \mathscr{S} \text{ is } d\text{-continuous on } [0,T]{\backslash}\mathscr{C}\,. 
\]
\end{prop}





\begin{proof}[Proof of Proposition \ref{prop:compactness2}, Part $(1)$.]
Let us 
consider a sequence 
$(\rho^n,\bj^n)_n$ fulfilling \eqref{eq:compa1}.
\par
By \eqref{initial-thm:main} we have that $\sup_n  \rho_0^n(V)  \leq M$ for some $M>0$
and thus, by 
the mass conservation property $\rho_t^{n}(V) = \rho_0^{n}(V)$ for all $t\in [0,T]$ we gather that 
 \[
 \sup_n \sup_{t\in[0,T]} \rho_t^n(V)\leq M\,.
 \]
 We gain additional bounds by resorting to the estimates from Lemma \ref{lm:proprr}. In particular, 
 \eqref{eq:disspb}, combined with  $\kappa^n \leq  \kappa $, give 
for every $n \in \N$, $t\in [0,T]$,  and $\xi\in \Bb(E)$ with $\|\xi\|_{\infty}\leq 1$
\begin{equation}
\label{eq:disspbn}
\begin{aligned}
 (\scrR^n)^*(\rho^n_t,\beta \xi) &\leq \frac{c_{\upalpha}}{2} \psih^*(\beta) \iint_{E} \xi^2(x,y) \, (\teta_{\rho_t^n}^+{+}\teta_{\rho_t^n}^-{+}\tetapin)(\dd x\dd y)\\
&\leq \frac{c_{\kappa}\psih^*(\beta)c_{\upalpha}(2\rho_t^n(V){+}\pi(V))}{2}\left\|\frac{\xi}{1{\wedge}\dV}\right\|_{\infty}^2 .
\end{aligned}
\end{equation}
In turn, from \eqref{eq:disspbd} and
again $\kappa_n\leq \kappa$ we obtain
\begin{equation}
\label{eq:disspbdn}
\begin{aligned}
\scrR(\rho_t^n,\bj_t^n)&\geq  \frac12 \calF_\psih\left(2(1{\wedge}\dV)\bj_t^n\Big|c_{\upalpha}(1{\wedge}\dV)^2 (\teta_{\rho_t^n}^+{+}\teta_{\rho_t^n}^-{+}\tetapin\right)\\
&\geq \frac12  \hat \psih\Big(2\|(1{\wedge}\dV)\bj_t^n\|_{\mathrm{TV}}\,,\, c_{\upalpha}c_{\kappa}(2\rho_ t^n(V){+}\pi(V))\Big).
\end{aligned}
\end{equation}
Finally, let $S:=\sup_{n} \sup\nolimits_{t\in [0,T]} \scrE(\rho^n_t)$
and consider the energy sublevel
\begin{equation}
\label{energy-sublevel-jasp}
\mathsf{K}:=\left\{ \rho \in \calP(V) : \scrE(\rho) = \int_V \upphi \left( \frac{\dd\rho}{\dd \pi}\right) \dd \pi \leq S\right\}.
\end{equation}
Because of the superlinear growth of $\upphi$, and of the characterization of setwise compactness provided  by \eqref{eq:73}, $\mathsf{K}$ is sequentially compact with respect to setwise convergence, in particular with respect to narrow convergence, and any narrowly converging sequence in $\mathsf{K}$ converges setwise as well.

We will now split the proof into a number of claims. \medskip

\paragraph{\bf Claim $1$:} \emph{There exists a narrowly continuous curve $(\rho_t)_{t\in [0,T]}\subset \mathsf{K}$ such that, along a (not relabeled) subsequence}
\begin{equation}
\rho_t^{n}\to \rho_t \quad \text{narrowly and setwise in } \calM^+(V) \quad \text{for all } t\in [0,T].
\label{compactness-jasper}
\end{equation}
To prove \eqref{compactness-jasper} we will resort to Proposition \ref{thm:Ascoli}. In fact,
\eqref{compactness} follows from the fact that $\rho_t^n \in \mathsf{K}$ for every $n \in \N$ and $t\in [0,T]$. 
\par
In order to verify the equicontinuity condition \eqref{equicontinuity}, we argue as follows:
Let $h_t^n:=\|(1{\wedge}\dV)\jj_t^n\|_{\mathrm{TV}}$. By  estimate \eqref{eq:disspbdn} we obtain 
\[\limsup_{n\to \infty} \int_0^T \frac12  \hat \psih\Big(2h_t^n\,,\, c_{\upalpha}c_{\kappa}(2 M+\pi(V))\Big)\, \dd  t < \infty.
\]
Due to the superlinearity of $\psih$ and Lemma 
\ref{l:crucial-F}, 
 up to non-relabeled subsequence  the measures $(h_t^n \Lebone)_n$ converge setwise in $\calM^+([0,T])$ to some measure $H\ll \Lebone$, or equivalently, $h_t^n$ converges weakly in $L^1([0,T])$ to some $h\in L^1([0,T])$. In particular, 
\begin{equation}
\label{cvg-hn}
\lim_{n\to \infty}\int_s^t h^n(r) \, \dd  t = \int_s^t h(r) \, \dd t, \qquad \mbox{for $[s,t]\in [0,T]$.}
\end{equation}
We set $\omega(s,t):=\int_s^t h(r)\, \dd t$. Clearly, $\omega(s,t)$ is symmetric, bounded and satisfies $\lim_{s\to t}\omega(s,t)=0$. Both $\mathsf{K}$ and $\omega$ will play their corresponding roles in applying Proposition \ref{thm:Ascoli} on $\calM^+(V)$ with the bounded Lipschitz metric (see \eqref{nbl-norm}). What remains to show is that 
\begin{equation}
\limsup_{n\to \infty} \| \rho_t^n{-}\rho_s^n\|_{\BL} \leq \omega(s,t), \quad \mbox{for all $[s,t]\subset[0,T]$,}  
\end{equation}
but this follows directly from combining the estimate 
\[
 \| \rho_t^n{-}\rho_s^n\|_{\BL} \leq      \int_s^t 
 \| (1{\wedge}\dV)   \jj_r^n \|_{\mathrm{TV}} \dd r 
 \] 
 (cf.\   \eqref{j-controls-rho})
 with \eqref{cvg-hn}.
\medskip

\paragraph{\bf Claim $2$:} \emph{
There exists a measurable family $(\bj_t)_{t\in [0,T]}\subset \Ms(E)$ with $\bj_t(E\setminus \Ed)=0$ for all $t\in [0,T]$, and 
    \begin{equation}
    \label{jasper-j-bound}
\int_0^T \iint_E  (1{\wedge}\dV(x,y))\, |\bj_t|(\dd x \dd y)\,\dd t<+\infty,
    \end{equation}
such that, along a (not relabeled) subsequence}
\begin{equation}
\label{Jasper-jn}
(1{\wedge} \dV )\bj_{\!\Lebone}^{n}\to (1{\wedge} \dV )\bj_{\!\Lebone}  =(1{\wedge} \dV )\bj_t \Lebone  \quad 
\text{setwise in } \calM([0,T]{\times}E). 
\end{equation}

From \eqref{eq:disspbdn} we obtain the upper bound 
\begin{align*}
\limsup_{n\to \infty} \frac12  \calF_\psih\left(2(1{\wedge}\dV)\bj_t^n \Lebone\Big|c_{\upalpha}(1{\wedge}\dV)^2 (\teta_{\rho_t^n}^+{+}\teta_{\rho_t^n}^-{+}\tetapin)\Lebone\right)\leq \limsup_{n\to \infty}  \int_0^T \scrR(\rho_t^n,\bj_t^n)\, \dd t < \infty\,.
\end{align*}
Now, by Lemma \eqref{eq:knass} we have 
\[(1{\wedge}\dV)^2 (\teta_{\rho_t^n}^+{+}\teta_{\rho_t^n}^-{+}\tetapin)\Lebone \,\longrightarrow \,(1{\wedge}\dV)^2 (\teta_{\rho_t}^+{+}\teta_{\rho_t}^-{+}\tetapi) \Lebone \quad \text{setwise in } \calM([0,T]{\times}E). \]
Then, recalling 
 Lemma \ref{l:crucial-F} we obtain 
\[(1{\wedge}\dV) \bj_{t}^{n}\Lebone \to (1{\wedge}\dV) \nu, \, \mbox{setwise in $\calM([0,T]\times E)$ for some }
\begin{cases}
\nu\in \Ms([0,T]\times E) \text{ with}
\\
\nu\ll (\teta_{\rho_t}^+{+}\teta_{\rho_t}^-{+}\tetapi) \Lebone \ll \Lebone
\end{cases}
.\]
 Therefore,  $\nu=\jj_t \Lebone$ for some measurable family $(\jj_t)_{t\in [0,T]}\subset \Ms(Y)$, and this establishes 
 convergence \eqref{Jasper-jn}.
\medskip

\paragraph{\bf Claim $3$:}
\emph{We have $(\rho,j)\in \CE 0T$ and
part $(1)$ of Proposition \ref{prop:compactness2} is proven.}
\par\noindent
We fix 
$\varphi \in \Lipb(V)$ and 
  take the limit as $n\to\infty$ in 
\[
	\int_V \varphi(x)\, \rho_t^{n}(\dd x)  - \int_V \varphi(x)\, \rho_s^{n}(\dd x)   = \int_s^t \iint_E \dnabla\varphi\,(x,y)\,\jj_r^{n}(\dd x\dd y)\,\dd r \qquad \text{for all } 0 \leq s \leq t \leq T\,.
	\]\
	It suffices to comment the limit passage on the right-hand side, based on the fact that $\dnabla\varphi$ can be rewritten as $\zeta (1{\wedge}\dV)$, with $\zeta 
	\in \Bb(E)$,
    and on convergence \eqref{Jasper-jn}.
\end{proof}

\begin{proof}[Proof of Proposition \ref{prop:compactness2}, Part $(2)$.]
Le us now additionally suppose that 
$(\rho^n,j^n)\in \ECE 0T$, with  the bound \eqref{eq:cb2} for the densities
$(u_n)_n$. We will 
 establish the following 
\medskip

\paragraph{\bf Claim $4$:} \emph{
There exists a measurable family $(\bj_t)_{t\in [0,T]}\subset \Ms(E)$ 
satisfying \eqref{jasper-j-bound} such that, in addition, for the aforementioned subsequence, convergence \eqref{eq:comce1} holds.}

\par\noindent
Without loss of generality, we may confine ourselves to proving \eqref{eq:comce1}
for all test functions
$\xi \in \Bb(V)$ with $\|\xi\|_\infty\leq 1$
and $\iint_{E}\xi^2 \, \dd \tetapi < \infty$. 
\par
We combine the  intermediate estimate in \eqref{eq:disspbn}, i.e.,
\[
(\scrR^n)^*(\rho^n_t,\beta \xi) \leq \frac{c_{\upalpha}}{2} \psih^*(\beta) \iint_{E} \xi^2(x,y) \, (\teta_{\rho_t^n}^+{+}\teta_{\rho_t^n}^-{+}\tetapin)(\dd x \dd y)
 \qquad \text{for all } \beta>0\,,
\]
 with the bound
$ \sup_{n} \sup\nolimits_{t\in [0,T]} \|u_t^n\|_{\infty} =: C <\infty$
and the estimate $\kappa_n \leq \kappa$, to infer 
\begin{equation}
\label{eq:disspbn2}
 (\scrR^n)^*(\rho^n_t,\beta \xi) \leq \frac{c_{\upalpha}(2C{+}1) \psih^*(\beta)}{2}  \iint_{\Ed} \xi^2(x,y)\, \tetapi(\dd x \dd y) \qquad \text{for all } \beta>0\,.
\end{equation}
For fixed $m\geq 1$, 
Let 
\[
A_m=\{(x,y) \subset E : 
\dV(x,y)>\tfrac{1}{m}\}, \qquad  B_m:=[0,T]{\times} A_m, \qquad \xi_m:=\boldsymbol{1}_{A_m} \xi
\]
with 
 $\xi\in \Bb(E)$
 the aforementioned test function. 
  Since  the sequence of measures 
  $((1{\wedge}\dV) \bj_{t}^{n}\Lebone)_n$ converges setwise to $(1{\wedge}\dV)\bj_{t}\Lebone$
  and the function $1{\wedge}\dV$ is 
  is bounded from above and below on 
$B_m$, we have that
\begin{equation*}
\begin{aligned}
\forall\, m \geq 1\, : \quad 
\lim_{n\to \infty} \iiint_{[0,T]\times E]} \xi_m(x,y) \,  \bj_{\!\Lebone}^{n}(\dd t \dd x \dd y) & =
\lim_{n\to \infty} \iiint_{B_m} \frac{\xi(x,y)}{1{\wedge}\dV(x,y)}(1{\wedge}\dV(x,y)) \,  \bj_{\!\Lebone}^{n}(\dd t \dd x \dd y)
\\ &
= \iiint_{B_m} \frac{\xi(x,y)}{1{\wedge}\dV(x,y)}(1{\wedge}\dV(x,y)) \,  \bj_{\!\Lebone}(\dd t \dd x \dd y) \\ & = 
\iiint_{[0,T]\times E} \xi_m(x,y) \,  \bj_{\!\Lebone} (\dd t \dd x \dd y)\,.
\end{aligned}
\end{equation*}
Therefore, to establish the desired convergence 
\eqref{eq:comce1} it is sufficient to show that 
\begin{equation}
\label{needed-1-jasp}
\lim_{m\to \infty} \iiint_{[0,T]\times E} \xi_m(x,y) \,  \bj_{\!\Lebone}( \dd t \dd x \dd y)=\iiint_{[0,T]\times E} \xi(x,y)  \, \bj_{\!\Lebone} (\dd t \dd x \dd y)
\end{equation}
and the vanishing estimate on the complementary 
$B_m^c = [0,T]{\times} A_m^c$
\begin{equation}
\label{needed-2-jasp}
\lim_{m\to \infty} \limsup_{n\to \infty} \iiint_{B_m^c} |\xi|(x,y) \, |\bj_{\!\Lebone}|(\dd t \dd x \dd y) = 0.
\end{equation}
We start from 
 the latter estimate,
Now, by \eqref{eq:disspbn2} we have for any $\beta>0$ and $m,n\geq 1$,
\begin{equation*}
\begin{aligned}
\forall\, \beta>0\,: \qquad 
&\iiint_{B_m^c} |\xi_m|(x,y)  \, |\bj_{\!\Lebone}|(\dd t \dd x \dd y) 
\\
&\leq \frac{1}{\beta}\left(\int_0^T (\scrR^n(\rho_t^n,\bj_t^n)\, \dd t + \int_0^T (\scrR^n)^*(\rho_t^n, \boldsymbol{1}_{A_m^c}|\xi_m|)\, \dd t \right)\\
&\leq \frac{1}{\beta}\left(\int_0^T (\scrR^n(\rho_t^n,\bj_t^n)\, \dd t +\frac{c_{\upalpha}(2C{+}1) \psih^*(\beta)}{2} \int_0^T\!\!\iint_{E} \boldsymbol{1}_{A_m^c}\xi^2(x,y) \, \tetapi(\dd x \dd y) \, \dd t\right).
\end{aligned}
\end{equation*}
Taking successive limits as $n\to\infty$
(note that the rightmost term in the above estimate no longer depends on $n$),
and then as  $m\to\infty$,  we then obtain 
\[
\forall\, 
\beta>0\,: \qquad
\limsup_{m\to\infty}\limsup_{n\to \infty} \iiint_{B_m^c} |\xi_m|(x,y) \,  |\bj_{\!\Lebone}| 
(\dd t \dd x \dd y)
\leq \frac{1}{\beta} \limsup_{n\to \infty}\int_0^T (\scrR^n(\rho_t^n,\bj_t^n)\, \dd t \leq \frac{C}{\beta},
\]
the last estimate due to the assumed bound on the action.
Since $\beta$ is  arbitrary,  estimate \eqref{needed-2-jasp} follows.  Finally,  \eqref{needed-1-jasp}
 since the former follows by dominated convergence, again exploiting  the fact that $(\rho,\bj)$ has finite action. 
\medskip

\paragraph{\bf Conclusion of the proof.} Relying on
convergence \eqref{eq:comce1},
 one can take the limit as $n\to \infty$ in the reflecting continuity equation \eqref{eq:RCE} fulfilled by the curves $(\rho^n,\jj^n)$. In fact,
 \eqref{eq:comce1} allows us to take the limit in the integral
 \[
 \int_s^t \!\!\iint_E \xi(x,y)\,\jj_r^{n}(\dd x\dd y)\,\dd r 
 \]
 for any $\xi \in \Bb(E)$ such that
\[\iint_{E}\xi^2 \, \dd \tetapi < +\infty, \qquad
\text{which  implies}  \qquad \iint_{E}\xi^2 \, \dd \tetapin < +\infty\,.
\]
In particular, we may choose $\xi:=\dnabla \varphi$ with $\varphi $ in the space $ \calX_2$ associated with the kernel $\kappa$. This ensures that that $(\rho,j)\in \ECE 0 T$.
\end{proof}

\subsection{Proof of Theorem \ref{thm:robustness}}
\label{ss:5.2}
We are in a position to carry out the 
\begin{proof}[Proof of Theorem \ref{thm:robustness}, Part $(1)$.]
Let $(\rho^n,\jj^n)_n$ be a sequence of
$\Dissipative$ solutions as in the statement.
From the upper energy-dissipation estimate 
\begin{equation}
\label{UEDE-7n}
\int_0^t \left( \scrR^n(\rho_r^n, \bj_r^n) + \Fish^n(\rho_r^n) \right) \dd r+ \calS(\rho_t^n)   \leq \calS(\rho_0^n)  
\end{equation}
and the assumed bound $\sup_n \calS(\rho_0^n) <+\infty$ we immediately deduce that estimate \eqref{eq:compa1} holds for $(\rho^n,\jj^n)_n$. Therefore, we are in a position to apply the first part of Proposition \ref{prop:compactness2} and deduce that there exists $(\rho,j) \in \ACE 0T$ (i.e., with finite entropy and action for the limiting system $(\scrE,\scrR,\scrR^*)$), such that, along a (not relabeled) subsequence, convergences \eqref{cvg-thm1mod} hold. 
  \par
  In order to conclude that $(\rho,\jj)$ is itself a $\Dissipative$ solution, it remains to take the limit as $n\to\infty$ in \eqref{UEDE-7n}. 
  For the energy terms, the well-preparedness condition \eqref{initial-thm:main} guarantees that $\calS(\rho_0^n)\to \calS(\rho_0)$, while on the left-hand side, we have by \eqref{DBL}:
 \begin{equation}
\label{lsc-energies}
\scrE(\rho_t) \leq \liminf_{n\to\infty} \scrE(\rho_t^{n}) \qquad \text{for all } t \in [0,T]
\end{equation}  
\par
Let us now fix $t\in (0,T)$, let $\Lebone_t$ be the Lebesgue measure restricted to $(0,t)$, and let us introduce the vector measures
\[
\bbeta_{\!\Lebone_t}^{n} : = (\teta_{\!\rho^{n}}^-,\teta_{\!\rho^{n}}^+ ) \Lebone_t = 
\int_0^t (\teta_{\!\rho^{n}}^-(r),\teta_{\!\rho^{n}}^+(r) )\, \dd r, 
\quad
\bbeta_{\!\Lebone_t}: = 
(\teta_{\!\rho}^-,\teta_{\!\rho}^+ ) \Lebone_t 
=
\int_0^t (\teta_{\!\rho}^-(r),\teta_{\!\rho}^+(r) )\, \dd r 
\]
(namely, we consider the construction \eqref{Lebintegrated}, with the Lebesgue measure now restricted to $(0,t)$). 
It is not difficult to derive from Lemma \ref{eq:knass}
that 
\begin{equation}
\label{setwise4beta}
\bbeta_{\!\Lebone_t}^{n} \to \bbeta_{\!\Lebone_t} \quad \sigma\text{-setwise in } \calM([0,t]{\times}E;\R^2)\,,
\end{equation}
and we have the analogous convergence  for  the measures $\tetapin {\otimes}\Lebone_t $ to  $\tetapi {\otimes}\Lebone_t$. 
Now, let us resort to  the representation formula \eqref{equivalent-upsilon} for $\scrR$, rewritten, with slight abuse of notation, in terms of the measures $(\bbeta_{\!\Lebone_t}^{n} )_n$
and $(\bj_{\!\Lebone_t}^{n})_n$:
thus, we have
\[
\begin{aligned}
\liminf_{n\to\infty}\int_0^t \scrR^{n}(\rho_r^{n}, \jj_r^{n})\, \dd r &=\frac12 \liminf_{n\to\infty}  \calF_\Upsilon \left((\bbeta_{\!\Lebone_t}^{n} ,2 \bj_{\!\Lebone_t}^{n})|\tetapin{\otimes}\Lebone_t\right)
\\
& \stackrel{(1)}\geq 
\frac12 \calF_\Upsilon \left((\bbeta_{\!\Lebone_t} ,2 \bj_{\!\Lebone_t})|\tetapi{\otimes}\Lebone_t\right)
 = \int_0^t \scrR(\rho_r, \jj_r)\, \dd r\,,
\end{aligned}
\]
where {\footnotesize (1)} follows from convergences \eqref{cvg-j}, 
\eqref{not-negligible-2n}, and 
\eqref{setwise4beta}
for $(\bj_{\!\Lebone_t}^{n})_n$, $(\tetapin{\otimes}\Lebone_t)_n$, and  $\bbeta_{\!\Lebone_t}^{n})_n$, respectively, and from the lower semicontinuity result in Lemma \ref{l:crucial-F}. 

Finally, in order to take the limit as $n\to\infty$ in the Fisher information term, we use that, by 
  Definition \ref{Def:Fisher}, 
 \begin{equation}
 \label{lsc-Fisher-eps-dep}
\begin{aligned}
\liminf_{n\to\infty}\int_0^t \Fish(\rho_r^{n})\, \dd r  &= \frac12 \liminf_{n\to\infty}\int_0^t 
 \iint_{E'}  \mathrm{D}_\upphi(u_r^{n}(x),u_r^{n}(y)) \, \tetapin(\dd x \dd y) \dd r
 \\
 & = 
\frac12  \liminf_{n\to\infty}  \iiint_{[0,t]{\times}E'}   \mathrm{D}_\upphi\left( \frac{\dd \bbeta_{\!\Lebone_t}^{n}}{\dd (\tetapin{\otimes}\Lebone_t)} \right) 
\, (\tetapin{\otimes}\Lebone_t)
(\dd r \dd x \dd y) 
\\
& \doteq \frac12  \liminf_{n\to\infty}  \calF_{\Xi} (\bbeta_{\!\Lebone_t}^{n}|(\tetapin{\otimes}\Lebone_t))  \stackrel{(2)}{\geq} \frac12  \calF_{\Xi} (\bbeta_{\!\Lebone_t}|(\tetapi{\otimes}\Lebone_t)) 
= \int_0^t \Fish(\rho_r)\, \dd r
\end{aligned}
\end{equation}
where we have used the functional $\calF_\Xi$ induced by the (\emph{convex}, by Assumption \ref{ass:dupphi}) mapping  $\Xi(w,z): = \mathrm{D}_\upphi(w,z)$,  and deduced estimate  {\footnotesize (2)} from convergences \eqref{setwise4beta} and Lemma \ref{l:crucial-F}. 
\par
All in all, we have completed the passage to the limit in \eqref{UEDE-7n}, and the upper energy-dissipation inequality \eqref{UEDE} follows. Thus, $(\rho,\jj)$ is a 
$\Dissipative$ solution.
\end{proof}
\par
Finally, let us address the
\begin{proof}[Proof of Theorem \ref{thm:robustness}, Part $(2)$.]
Let us now suppose that $(\rho^{n},\bj^{n})_n$ are $\Reflecting$ solutions to the $(\scrE,\scrR^{n},(\scrR^{n})^*)$ systems such that, in addition, \eqref{eq:cb2-required} holds. Then, we may apply
the second part of Theorem \ref{thm:robustness} and deduce that there exist a curve
$(\rho,\jj) \in \ECE 0T$ and a 
not relabeled subsequence along which convergences \eqref{cvg-hn} holds.
\par
By Part (1) of Theorem \ref{thm:robustness},
$(\rho,\jj)$ is a $\Dissipative$ solution.
Then, thanks to Theorem \ref{thm:ECE}
$(\rho,\jj)$ is also a $\Reflecting$ solution. 
\par
The enhanced convergences \eqref{enh-cvg-thm4.3}
ensue from a standard argument based on the comparison between the approximate energy-dissipation balance fulfilled by the curves $(\rho^n,\jj^n)_n$ and  the balance
for $(\rho,\jj)$.
We choose to omit the details, referring, e.g., to the proof of \cite[Thm.\ 3.11]{MRS13}. 
\par
Finally, the very last claim in the statement, under the strict convexity of $\upphi$, is an immediate consequence of the uniqueness result from Proposition 
\ref{prop:uniqueness-too}. 
\end{proof}

\appendix 

\section{Proofs of some results from Section \ref{s:prelims}}
\label{app:preliminary}
\begin{proof}[Proof of Lemma \ref{lm:sigmap}]
\emph{(1)} Consider such pairs $\mu^{\pm},\nu^{\pm}$, and construct an increasing sequence $(A_n)_{n\in \N}$ such that $\lim_{n\to\infty} A_n=Y$ and all measures are finite on $A_n$ for every $n\in \N$. Now, fix $n\in \N$, and consider all Borel sets $B$ contained in $A_n$. Then the restrictions 
$(\mu^+\mres A_n)-(\mu^-\mres A_n)$ and $(\nu^+\mres A_n)-(\nu^-\mres A_n)$ are both finite signed measures over $A_n$ that coincide for all $B\subset A_n$, with $(\mu^+\mres A_n) \perp (\mu^-\mres A_n)$ and $(\nu^+\mres A_n)\perp ( \nu^-\mres A_n)$. By the Jordan decomposition theorem it follows that $\mu^{\pm} \mres A_n=\nu^{\pm} \mres A_n$. Since for any measurable set $B$ we have the equality 
\[\mu^{\pm}(B)=\lim_{n\to \infty} \mu^{\pm}(B\cap A_n) = \lim_{n\to \infty} (\mu^{\pm}\mres A_n)(B \cap A_n) = \lim_{n\to \infty} (\nu^{\pm}\mres A_n)(B \cap A_n) = \nu^{\pm}(B),\]
we obtain that $\mu^\pm = \nu^\pm$. 

\emph{(2)} Suppose we  have a sequence $(A_n)_{n}$ and a signed set-function $\nu$ with the aforementioned properties. Since for every $n$ the signed set-function $\nu\mres A_n$ is a finite signed Borel measure,
 considering its Hahn decomposition 
we obtain positive finite and mutually singular measures $\nu^{\pm}_n$ defined on $\mathfrak{B}(A_n)$, satisfying the consistency relation $\nu^{\pm}_{n'}(B)=\nu^{\pm}_{n}(B)$ if $B\in A_{n'}$ and $n'\leq n$. We define the set-functions
\[\mu^{\pm}(B):=\lim_{n\to \infty} \nu_n^{\pm}(B\cap A_n).  \]
It is clear the 
 the sequence $(\nu_n^{\pm}(B\cap A_n)_n$
is monotone,   allowing us to show that $\mu^{\pm}$ are $\sigma$-additive and are indeed Borel measures, and in addition $\sigma$-finite, due to the finiteness on every $A_n$. Now, suppose that $\nu^+,\nu^-$ are not mutually singular. Then there exists some $B$ with positive measure for both $\mu^{\pm}$. However, by their construction, this implies that there exists some $n\in \N$ with both $\nu^{\pm}_n(B\cap A_n)$ positive, leading to a contradiction. Hence, $\mu^+\perp \mu^-$. 
\end{proof}

\begin{proof}[Proof of Lemma \ref{l:basic-props}]
\emph{(1)} 
 Let $(A_n)_n$ be an exhaustion of $Y$ for which all the measures $\mu^{\pm},\nu^{\pm}$ are finite, consider $\cup_{n=1}^{\infty} \mathfrak{B}(A_n)$ and define 
the set function $\bar \eta: \cup_{n=1}^{\infty} \mathfrak{B}(A_n) \to \R$ 
via $\bar \eta = \mu + \nu$. By 
statement (2) from Lemma \ref{lm:sigmap}, $\bar 
\eta$ extends to an element $\eta \in \calM_{\sigma}(Y)$.  

\emph{(2)} Let $(A_n)_n$ be a exhaustion of $Y$ corresponding to $\nu$. 
Since $|f|$ is
$|\nu|$-a.e.\ finite, it is also 
$\nu^{\pm}$-a.e.\ finite, and the sets $B_0:=\{x : |f|=+\infty\}$, $B_m:=\{x : |f|<m\}$ for $m\in \N$
provide an exhaustion for $\nu^+$ and $\nu^-$. 
Constructing $C_{n,m}:=A_n \cap B_m$,  we find that $|f| |\nu|\in \calM_{\sigma}^+(Y)$, and the mutual singularity of $(f \nu)^{\pm}$ is straightforward to verify. 


\emph{(3)} Due to the mutual singularity of $\nu^{\pm}$ we can find disjoint sets $P,Q$ such that $Y=P\cup Q$, $\nu^-(P)=0$, $\nu^+(Q)=0$. Since $g_n$ converges $|\nu|$-a.e.\  to $g$ and $\int g_n \, \dd |\nu|$ converges to $\int g\, \dd |\nu|$, we find by Fatou's Lemma that  $\int_B g_n \, \dd |\nu|$ converges to $\int_B g \, \dd |\nu|$ for any measurable set $B$. Setting $B=P$ and $B=Q$, we deduce that $\int_Y g_n \, \dd \nu^{\pm}$ converge to $\int_Y g \, \dd \nu^{\pm}$. 

By a modified version of the dominated convergence theorem, \cite[\S 2, Thm.\ 2.8.8]{Bogachev07}, we can then conclude that the integrals $\int_Y f_n \, \dd \nu^{\pm}$ converge as well, and the desired result now follows after taking limits in \eqref{eq:fsint}.
\end{proof}

\begin{proof}[Proof of Lemma \ref{lm:sigmasetwise}]
Suppose we have two exhaustions $(A_k)_{k\in \N}$ and $(B_m)_{m\in \N}$ corresponding respectively to $\mu,\nu\in \Ms(Y)$, with $\mu_n {\mres} A_k$ converging setwise to $\mu {\mres} A_k$ and $\mu_n {\mres} B_m$ converging setwise to $\nu {\mres} B_m$. Then, $\mu$ and $\nu$ coincides on the sets $A_k\cap B_m$, ensuring the equality $\mu=\nu$ due to Lemma~\ref{lm:sigmap}(2). 

Next, suppose that for some
$f$ satisfying our assumptions
(yielding that  $f\mu_n$ are \emph{finite} measures on $Y$), we have that $f \mu_n$ converges setwise to $\nu$. Then clearly 
\[\nu(\{x:f=0\}) = \lim_{n\to \infty} \int_{f=0} f\mu_n =0,\]
allowing us to write $\nu=f \mu$ with $\mu(\dd x) :=\boldsymbol{1}_{f>0} (x) f^{-1}(x) \nu(\dd x )$. The previous setwise convergence yields 
$\sup_n \int_Y f \dd |\mu_n|< \infty $.
Now, let us consider 
 the sets $A_k:=\{ x : f(x)\geq k^{-1} \}$:  we have that 
 \[
|\mu|(A_k) \leq \int_{A_k} k f(x) |\mu|(\dd x) \leq  k  | f \mu|(Y)\,.
 \]
 Therefore, the sets $(A_k)_k$ provide an exhaustion for $\mu$,
 and we observe that  for any  $g\in \B_b(Y)$ 
\[ \lim_{n\to\infty} \int_{A_k} g \,  \dd \mu_n = \lim_{n\to\infty} \int_{Y} \boldsymbol{1}_{A_k}(x) \frac{g(x)}{f(x)} (f(x)\,\mu_n)(\dd x) \stackrel{(\star)}{=}  \int_{A_k} g(x)  \,  \dd \mu,\]
where {\footnotesize $(\star)$} follows from the fact that on $A_k$  both $f$ and $1/f$ are bounded.
Then, $\mu_n$ converges 
$\sigma$-setwise to $\mu$. 

Vice versa, suppose that $(\mu_n)_{n\in \N}\subset \Ms(Y)$ converges $\sigma$-setwise to $\mu\in \Ms(Y)$, with again the exhaustion $(A_k)_{k\in \N}$. Let us define 
\begin{equation}
\label{ck-constants}
c_k:=\sup_{n} |\mu_n|(A_k)
\end{equation}
and note 
$|\mu|(A_k)\leq c_k$. We can then construct the strictly positive function
\[ f(x):=\sum_{k=1}^{\infty} \frac{2^{-k}}{1+c_k} \boldsymbol{1}_{A_k}(x).\]
Clearly, $\sup_n \int_Y f \dd |\mu_n|,\int_Y f \dd |\mu| < \infty $, and 
\[\boldsymbol{1}_{A_k^c}(x) f(x) \leq  \sum_{\ell=k+1}^{\infty} \frac{2^{-\ell}}{1+c_{\ell}} \boldsymbol{1}_{A_{\ell}}(x),\]
and, in particular,
\[\lim_{k\to \infty} \sup_{n} \int_{A_k^c} f |\dd \mu_n|\leq \lim_{k\to \infty} \sup_{n} \sum_{\ell=k+1}^{\infty} \frac{2^{-\ell}}{1+c_{\ell}} |\mu_n|(A_{\ell}) \leq \lim_{k\to \infty} \sum_{\ell=k+1}^{\infty} 2^{-\ell} = 0.  \]
We have, for any measurable set $B$, and $k,n\in \N$,
\begin{equation}\label{eq:sigmalim}
\int_B f \dd \mu^n =\int_Y \boldsymbol{1}_{B\cap A_k} f \dd \mu^n + \int_Y \boldsymbol{1}_{B\cap A_k^c} f \dd \mu^n
\end{equation}
Note that since $f$ is bounded and $B\cap A_k \subset A_k$, by $\sigma$-setwise convergence of $\mu^n$ to $\mu$, and a dominated convergence argument (via  Lemma \ref{lm:sigmap}),
\[\lim_{k\to \infty} \lim_{n\to \infty} \int_Y \boldsymbol{1}_{B\cap A_k} f \dd \mu^n = \int_Y \boldsymbol{1}_{B} f \dd \mu\]
Moreover, by the previous bounds, 
\[\limsup_{k\to \infty} \limsup_{n\to \infty} \left| \int_Y \boldsymbol{1}_{B\cap A_k^c} f \dd |\mu|^n \right| = 0.\]
Since the choice of $k$ in \eqref{eq:sigmalim} was arbitrary, we find 
\[\lim_{n\to \infty}\int_B f \dd \mu^n = \lim_{k\to \infty} \lim_{n\to\infty} \int_Y \boldsymbol{1}_{B\cap A_k} f \dd \mu^n = \int_B  f \dd \mu.\]
We have thus shown that $(f\mu^n)_n$ setwise converges to $f\mu$.
\end{proof}

\section{Measurability properties of the kernels}
\label{ss:appB}
Our construction of the coupling measure $\tetapi$ \ref{teta-coupling} relies on the following property: that for every (Borel) measurable set $C\subset E$
\begin{equation}
\label{equiv-measu}
V \ni x \mapsto \int_V {\boldsymbol 1}_{C}(x,y) \kappa(x,\dd y ) \in [0,+\infty]
\end{equation}
is a nonnegative extended Borel mapping. If that holds, then 
\[\tetapi(C):=\int_V \left(\int_V {\boldsymbol 1}_{C}(x,y) \kappa(x,\dd y )\right) \pi(\dd x)\]
defined a (countably additive) measure by a monotone convergence argument.  In passing, we also note that, by our definition $\tetapi $ satisfies
\begin{equation}
\label{null-property}
\text{for all } N,B \in \mathfrak{B}(V) \text{ with } N \cap B = \emptyset\,: 
\quad \pi(N) =0 \ \Rightarrow \ \tetapi(N{\times}B) = \tetapi(B{\times}N)=0\,.
\end{equation}
\par
In the following result we gain further insight into the measurability condition involving \eqref{equiv-measu}, and in particular show that it is equivalent to property \eqref{true-measurability}, required throughout the paper. 
\begin{lemma}\label{lm:kmeas}
Let $\kappa(x,\cdot)$ be a family of Borel measures
with $\kappa(x,\{x\})=0$ for all $x\in V$ and satisfying  \eqref{mitigation of singularity}, i.e.\ 
 \begin{equation}
 \label{mitigation of singularity-APP}
\sup_{x\in V}\int_{V} (1{\wedge} \dV^2(x,y)) \,\kappa(x,\dd y) =: c_\kappa <+\infty.
\end{equation}
 Then, $\kappa(x,\cdot)\in \Ms^+(V)$ for all $x \in V$, and  the following measurability properties are equivalent 
 \begin{enumerate}
\item Property \eqref{true-measurability} holds for any $f \in \Bb(V)$,
 now integrated on the whole of $V$, i.e.,
 \begin{equation}
\mbox{the map $\displaystyle V\ni x \mapsto \int_{V} f(y) (1{\wedge}\dV^2(x,y)) \, \kappa (x,\dd y)$ is a bounded Borel mapping.}
 \end{equation}
 \item The family $\left((1{\wedge}\dV^2)\kappa(x,\cdot)\right)_{x\in V}$ is a Borel family.
 \item For any $g\in \Bb(E)$ 
 \begin{equation}
\mbox{the map $\displaystyle V\ni x \mapsto \int_{V} g(x,y)(1{\wedge}\dV^2(x,y)) \, \kappa (x,\dd y) $ is a bounded Borel mapping.}
 \end{equation}
 \item  For any $h\in \B^+(E ; \overline{\R})$
 \begin{equation}
\mbox{the map $\displaystyle V\ni x \mapsto \int_{V} h(x,y)\, \kappa (x,\dd y) $ is in  $\B^+(V;  \overline{\R})$}.
 \end{equation}
 \item  For any  $C\in \mathfrak{B}(E)$ 
 \begin{equation}
\mbox{the map $\displaystyle V\ni x \mapsto \int_{V} {\boldsymbol 1}_{C}(x,y) \, \kappa (x,\dd y) $ is in $\B^+(V;  \overline{\R})$}.
 \end{equation}
 \item The family $\left(\kappa(x,\cdot)\right)_{x\in V}$ is a Borel family. 
 \end{enumerate}
\end{lemma}
\begin{proof}
 First, fix any $x\in V$. Since the function $V\ni y\mapsto (1{\wedge}\dV^2(x,y))$ is measurable and positive $\kappa(x,\cdot)$-a.e., the bound \eqref{mitigation of singularity-APP} implies that $\kappa(x,\cdot)\in \Ms^+(V)$. Next, let us sketch the necessary equivalences.
\\
\noindent
$(1)\implies (2)$ This follows via setting $f(y)={\boldsymbol 1}_B(y)$ for any Borel set $B\subset V$. \\
$(2)\implies (3)$ The family $\left((1{\wedge}\dV^2)\kappa(x,\cdot)\right)_{x\in V}$ is a Borel family such that 
\[\sup_{x\in V} (1{\wedge}\dV^2)\kappa(x,\cdot)(V) < \infty,\]
and hence the arguments of  \cite[Sec.\ 2.4]{PRST22}  apply. \\
$(3)\implies (4)$ One can take 
\[h_{\eps}:=\frac{1{\wedge}\dV^2}{\eps+1{\wedge}\dV^2} (h\wedge \eps^{-1}), \qquad g_{\eps}:=\frac{1}{\eps+1{\wedge}\dV^2} (h\wedge \eps^{-1})\]
and let $\eps\to 0$. \\
$(4)\implies (1)$ Substitute $h(x,y)=(1{\wedge}\dV^2)(x,y)f(y)$ and use \eqref{mitigation of singularity-APP}. \\
$(4)\iff (5)$ This follows from the fact that any measurable nonnegative $h$ is a monotone limit of simple functions, and vice versa, taking $h={\boldsymbol 1}_C$. \\
$(5) \iff (6)$ The direction $(5)\implies (6)$ is clear. Now, suppose that $\left(\kappa(x,\cdot)\right)_{x\in V}$ is a Borel family with $\kappa(x,\{x\})=0$ for every $x\in V$ and \eqref{mitigation of singularity-APP} holds. In particular, $V\ni x\mapsto \kappa(x,\{x\})$ is Borel, so with $C_0:=E\setminus E'$, we obtain that the mapping $x\mapsto \int_V {\boldsymbol 1}_{C_0}(x,y) \kappa(x,\dd y)$ is measurable, with 
\[\int_V {\boldsymbol 1}_{C_0}(x,y) \kappa(x,\dd y) \equiv 0.\]
 Moreover, for any $C=A\times B$, the mapping 
\[x\mapsto \int_V {\boldsymbol 1}_{C}(x,y) \kappa(x,\dd y)={\boldsymbol 1}_A(x)\kappa(x,B),\]
is measurable. We now lift from product sets to general sets, using the bounds on $\kappa(x,\dd y)$, similar as for simple products of $\sigma$-finite measures. Let us introduce
\[\calA := \left\{ C\subset \mathfrak{B}(E) \,: \,  x\mapsto \int_V {\boldsymbol 1}_{C}(x,y) \kappa(x,\dd y) \mbox{ is measurable}\right\},\]
which, at least, will contain any countable union of product sets. 

Now, by separability of $(V,\dV)$ we can construct product sets $B_{r_i}(x_j)\times B_{r_k}(x_l)$ using open balls $B_{r_i}(x_j)$ centered at $x_j$ with radius $r_i$, for some dense sequences $(x_j)_{i\in N}$ in $V$ and $(r_j)_{n\in \N}$ in $(0,+\infty)$, that cover the whole of $E$. Moreover, for every $n\geq 1$, let 
\[C_n:=C_0 \cup \bigcup_{i,j,k,l} \left\{ B_{r_i}(x_j)\times B_{r_k}(x_l) \, : \, \dV(x_j,x_l)>\left(\frac1n{-}r_i{-}r_k\right)   \right\}.\]
Note that the sequence  $(C_n\setminus C_0)_n$ exhausts $E'$ as $n\to \infty$, with $\dV(x,y)>n^{-1}$ for every $x,y\in C_n\setminus C_0$. Note that $(C_n)_n$ exhausts $E$ as $n\to \infty$, $C_n\in \calA$ for every $n\in \N$, and 
\[\int_{V} {\boldsymbol 1}_{C_n}(x,y) \kappa(x,y) \leq n c_{\kappa} < \infty, \mbox{ for all $x\in V$}.  \]
Via a monotone class argument on $\calA \cap \{ C \in \mathfrak{B}(E): C\subset C_n \}$, similar to \cite[Theorem 14.5]{Schilling2017measures}, we derive that $\{ C \in \mathfrak{B}(E): C\subset C_n \} \subset \calA$ for every $n$, and this can be extended to arbitrary $C\in \mathfrak{B}(E)$ by taking $C\cap C_n$ and the limit $n\to \infty$. We conclude $\calA=\mathfrak{B}(E)$.
\end{proof}

\section{Proof of Theorem \ref{thm:density-vindicated}}
\label{app:density-vindicated}
First of all, we  improve the approximation property
required in   \eqref{Ass:F-bis}. Namely,  Lemma  \ref{l:cutoff} below we show that, if    \eqref{Ass:F-bis} holds,
we can indeed
construct a sequence $(\widehat{\varphi}_n)_n $ such that, in addition to the convergences
in Ass.\ \ref{Ass:F-bis}, we have that  $ \| \widehat{\varphi}_n \|_\infty$ and 
$\|\overline\nabla \widehat{\varphi}_n \|_\infty = \sup_{x,y\in V} |\widehat{\varphi}_n(x,y)|$ are \EEE  uniformly bounded. 
As it will be clear from its proof (in particular, the arguments of Claim 2 therein),  Lemma \ref{l:cutoff} could be extended to the space $\mathcal{X}^{\Young} $, for \emph{any} Young function $\Young$.  

\begin{lemma}
\label{l:cutoff} 
Assume property \ref{Ass:F-bis}. 
For every $\varphi \in \mathcal{X}_2$ there exists a sequence $(\widehat{\varphi}_n)_n \subset \Lipb(V)$ such that as $n\to \infty$
 \begin{enumerate}
 \item   $\sup_{n
\in \N} \| \widehat{\varphi}_n \|_\infty
 \leq 2\|\varphi\|_{\infty}$; 
  \item $ \sup_{n
\in \N} \|\overline\nabla \widehat{\varphi}_n \|_\infty
 \leq 4\|\varphi\|_{\infty}$;
 \item $\widehat{\varphi}_n\to\varphi$ in $\rmL^1(V;\pi)$;
 \item $\overline\nabla \widehat{\varphi}_n \to \overline\nabla \varphi$ in  $\rmL^{2}(E;\tetapi)$.
 \end{enumerate}
 \end{lemma}
 
 \begin{remark}
\label{rmk:a-fortiori}
Indeed, because of the uniform  bounds  on $\| \widehat{\varphi}_n \|_\infty$ and $ \|\overline\nabla \widehat{\varphi}_n \|_\infty$, the convergences in items (2) \& (3) improve to 
\[
\text{$\widehat{\varphi}_n\to\varphi$ in $\rmL^p(V;\pi)$ and  $\overline\nabla \widehat{\varphi}_n \to \overline\nabla \varphi$ in  $\rmL^{p}(E;\tetapi)$ \emph{for every} $p\in [1,\infty)$.}
\]
 \end{remark}
 \begin{proof}[Proof of Lemma \ref{l:cutoff}]
Obviously, we may suppose that $\varphi \not\equiv 0$. 
Let $(\varphi_n)_n \subset \Lipb(V)$ approximate $\varphi$ as in \eqref{Ass:F-bis}. 
Let us introduce the truncation operator
\[
\tau_\ell: \R \to \R, \quad 
\tau_\ell(x):= \begin{cases}
x & \text{if } |x|\leq \ell,
\\
\ell & \text{if } x >\ell,
\\
-\ell  &  \text{if } x <-\ell,
\end{cases}
\quad \text{with  the place-holder } \ell:= 2\|\varphi\|_{\infty},
\]
and let us  define 
\begin{equation}
\label{truncated-phin}
\widehat{\varphi}_n(x):= \tau_\ell(\varphi_n(x))\,.
\end{equation}
Clearly, $ \widehat{\varphi}_n \subset \Lipb(V)$, with   $ \sup_{n
\in \N} \|\overline\nabla \widehat{\varphi}_n \|_\infty
\leq 2\ell$. 
 Thus, properties $(1)$ and  $(2)$ follow.
 We now check the remaining items of the statement in separate claims.  
\smallskip

\par
\noindent
{\sl \textbf{Claim $1$:} we have  $\widehat{\varphi}_n\to\varphi$ in $\rmL^1(V;\pi)$.}
\\
Since $\ell =2 \|\varphi\|_{\infty}$, we have that 
\[
O_n:= \{ x\in V\, : \ |\varphi_n(x) | >\ell \} \subset\left \{ x\in V\, : \ |\varphi_n(x){-} \varphi(x) | >\frac{\ell}2 \right \}\,.
\]
Therefore, since $\varphi_n \to \varphi$ in $\rmL^1(V;\pi)$, we infer that 
\[
\lim_{n\to \infty} \pi (O_n) \leq \lim_{n\to \infty} \pi \left(\left\{ x\in V\, : \ |\varphi_n(x){-} \varphi(x) | >\frac{\ell}2 \right \}\right) =0,
\]
so that 
\[
 \int_{O_n} |\widehat{\varphi_n}{-}\varphi(x)| \, \pi(\dd x) \longrightarrow  0 \text{ as } n \to \infty\,.
\]
In this way, we have that 
\[
\lim_{n\to\infty}\| \widehat{\varphi}_n{-} \varphi\|_{\rmL^1(V;\pi)} = \lim_{n\to\infty} \int_{V{\backslash}O_n} |\varphi_n(x){-}\varphi(x)| \, \pi(\dd x) =0\,.
\]
\smallskip

\par
\noindent
{\sl \textbf{Claim $2$:} we have that}
\begin{equation}
 \label{cvg-small-Orl}
 \lim_{n\to\infty} \iint_{\Ed}  \big|\overline\nabla \widehat{\varphi}_n(x,y){-} \overline\nabla \varphi(x,y) \big|^2 \, \tetapi (\dd x \dd y) =0
 \end{equation}
\\
First of all, we observe that, since $\widehat{\varphi}_n\to \varphi$ $\pi$-almost everywhere in $V$, by \eqref{null-property}
we have  the pointwise convergence
\begin{equation}
\label{pointwise-cvg-gradients}
\overline\nabla \widehat{\varphi}_n  \to  \overline\nabla \varphi \qquad \tetapi\text{-a.e.\ in } E\,.
\end{equation}
A direct calculation shows that 
\begin{equation}
    \label{ptw-ineq-nablas}
|\overline\nabla \widehat{\varphi}_n(x,y)| \leq |\overline\nabla \varphi_n (x,y)| \qquad 
\text{for every $(x,y) \in E$,}
\end{equation}
therefore we have for all $(x,y) \in \Ed$
\begin{equation}
\label{domination}
\begin{aligned}
|\overline\nabla (\widehat{\varphi}_n {-} \varphi)(x,y))|^2 & \stackrel{(1)}{\leq} 
2|\overline\nabla \widehat{\varphi}_n(x,y)|^2+ 
2|\overline\nabla \varphi(x,y)|^2
\\   & \leq
2| \overline\nabla \varphi_n(x,y)|^2+ + 
2|\overline\nabla \varphi(x,y)|^2
\\
& \leq
4| \overline\nabla \varphi_n(x,y) {-} \ona\varphi(x,y) |^2+ 
4|\overline\nabla \varphi(x,y)|^2
\,,
\end{aligned}
\end{equation}
where the very last estimate follows from 
$|\overline\nabla \varphi_n | \leq 
|\overline\nabla \varphi_n {-} \overline\nabla\varphi| +|\overline\nabla\varphi| $.
Combining convergence \eqref{pointwise-cvg-gradients} with the estimate \eqref{domination}, taking into account that 
$\overline\nabla \varphi_n \to \ona\varphi$ in $\rmL^2(E;\tetapi)
$ and resorting to the 
previously mentioned  dominated convergence theorem  \cite[\S 2, Thm.\ 2.8.8]{Bogachev07}, we conclude 
\eqref{cvg-small-Orl}. 
This concludes the proof of Claim $2$.
\end{proof} 
\par
Our next result shows that 
the density property \eqref{Ass:F-bis}
in fact extends to $\mathcal{X}^{\uppsi^*}$. 
\begin{lemma}
\label{cor:Oliver}
There hold
\begin{align}
&
\label{e:coincidentX} 
  \mathcal{X}^{\uppsi^*} =  \rmX^{\uppsi^*}   =   \mathcal{X}_2\,,
\end{align}
and the following estimates 
\begin{equation}
\label{est-condensed}
\begin{aligned}
\forall\, M>0  \ 
\forall\, \beta>0 \ \  & \exists\, c_M,\,  C_{M,\beta} >0  \ \ \forall\,  \eta \in \rmL^\infty(E;\tetapi) \text{ with } \|\eta\|_{\rmL^\infty(\tetapi)}
\leq M 
 \\
 & 
 \begin{cases}
\displaystyle \iint_{E}\uppsi^*(\beta\eta(x,y))\, \tetapi (\dd x \dd y) \leq C_{M,\beta}
 \iint_E |\eta(x,y)|^2 \, \tetapi (\dd x \dd y) \,,
 \\
\displaystyle \iint_E |\eta(x,y)|^2 \, \tetapi (\dd x \dd y) \leq c_M  \iint_E \uppsi^*(\eta(x,y))  \, \tetapi (\dd x \dd y)
 \,. 
 \end{cases}
\end{aligned}
\end{equation}
In particular, 
the density property \eqref{Ass:F-bis} holds in 
$\mathcal{X}_2$ if and only if it holds in $\mathcal{X}^{\uppsi^*}$, 
 with the convergence $\overline\nabla \varphi_n \to \overline\nabla \varphi$ in 
  $\rmL^2(\Ed;\tetapi)$ replaced by that in 
 $\rmL^{\uppsi^*}(\Ed;\tetapi)$. 
\end{lemma}
\begin{proof}
Preliminarily, observe that 
\begin{equation}
\label{Olli-est}
\forall\, \beta>0 \ \forall\, M>0 \ \  \exists\, K_{\beta,M}>0  \ \  \forall\, \xi \in [{-}M,M]\,: \quad \uppsi^*(\beta\xi)\leq K_{\beta,M} \uppsi^*(\xi)\,.
\end{equation}
Indeed, 
if $|\beta|\leq 1$ then $\uppsi^*(\beta \xi) \leq \uppsi^*(\xi)$ since $\uppsi^* $ is non-decreasing. Now let $|\beta | >1$ and let $r>0$
be as in  \eqref{easy-conseq}.
We  estimate $\uppsi^*(\beta \xi)$ distinguishing two cases:
\begin{enumerate}
\item $|\beta \xi| \leq r$: then $|\xi| \leq \tfrac{r}{|\beta|} <r$, since $|\beta|>1$. By  \eqref{easy-conseq}  we have
\begin{equation}
\label{easy-case1}
\uppsi^*(\beta \xi) \leq  \frac32 c_0 |\beta\xi|^2  \leq  \frac32 c_0 |\beta|^2  \frac{2}{c_0} \uppsi^*(\xi) = 3  |\beta|^2 \uppsi^*(\xi) \,.
\end{equation}
\item $|\beta \xi| > r$: then, $|\xi|> \frac r\beta$. By monotonicity of $\uppsi^*$, we find $\uppsi^*(\tfrac r{\beta}) 
\leq \uppsi^*(\xi)$, and thus
\begin{equation}
\label{easy-case2}
\uppsi^*(\beta \xi) \leq  \frac{\uppsi^*(\xi)}{\uppsi^*(\tfrac r{\beta}) }\uppsi^*(\beta \xi) \leq \frac{\uppsi^*(\xi)}{\uppsi^*(\tfrac r{\beta}) }\uppsi^*(\beta M)
\,,
\end{equation}
where for the last estimate we have used that $|\xi| \leq M$.
\end{enumerate}
 Combining \eqref{easy-case1} and \eqref{easy-case2} we conclude estimate \eqref{Olli-est}
with
 $
K_{\beta,M}= \max\left\{ 3  |\beta|^2, \frac{\uppsi^*(\beta M)}{\uppsi^*(\tfrac r{\beta}) }\right\}.
$
\par
$\vartriangleright $  \eqref{est-condensed}: 
We may continue \eqref{Olli-est} by using that, thanks \eqref{needed-control-bound},
\[
\uppsi^*(\xi) \leq \frac1{M^2} \xi^2  \psih^*(M ).
\]
All in all, we have shown that 
\[
\forall\, \beta>0 \ \forall\, M>0 \ \ 
 \forall\, \xi \in [{-}M,M]\,:
\quad \uppsi^*(\beta\xi)\leq K_{\beta,M} \frac{\psih^*(M )}{M^2} \xi^2\,,
\]
and the first of estimates \eqref{est-condensed} follows.
\par
In turn,  the second of estimates \eqref{est-condensed}  is an immediate
consequence of 
\eqref{Olli-est-2}. 

\par
$\vartriangleright \, \mathcal{X}^{\uppsi^*} =  \rmX^{\uppsi^*}$: 
it clearly suffices to prove that $  \rmX^{\uppsi^*} \subset \mathcal{X}^{\uppsi^*}$.
With this aim, 
 let us fix 
$\varphi \in  \rmX^{\uppsi^*}$.
Since $\varphi \in \rmL^\infty(V;\pi)$, we have that $\overline\nabla \varphi 
\in \rmL^\infty(E;\tetapi)$  and thus we can apply estimate \eqref{Olli-est} with 
$
M=  \|\overline \nabla \varphi\|_{\rmL^\infty(\tetapi)},$
which then ensures that 
\[
\forall\, \beta>0: \ \ 
\iint_{\Ed} \uppsi^*(\beta \overline\nabla \varphi (x,y))
\tetapi(\dd x \dd y )\leq K_{\beta,M}
\iint_{\Ed} \uppsi^*(\overline\nabla \varphi (x,y))
\tetapi(\dd x \dd y )<\infty\,.
\]
Hence, $\overline \nabla \varphi \in \mathcal{M}^{\uppsi^*}(\Ed;\tetapi)$, and thus $\varphi \in  \mathcal{X}^{\uppsi^*}$.
\medskip

$\vartriangleright\, \mathcal{X}^{\uppsi^*} = \mathcal{X}_2$:
 let  $\varphi \in \mathcal{X}^{\uppsi^*}$.
From  the second of estimates \eqref{est-condensed}
we immediately deduce that 
$
 \iint_{\Ed}  |\overline \nabla \varphi(x,y)|^2
 \tetapi (\dd x \dd y ) <\infty$, 
 whence
 $\varphi \in \mathcal{X}_2$. 
 Analogously, relying on the first of  \eqref{est-condensed} we show that 
$ \mathcal{X}_2 \subset \mathcal{X}^{\uppsi^*} $. 
\par
The last part of the statement is, again, 
an immediate consequence of \eqref{est-condensed}. Indeed, 
the latter estimates
 apply 
to any approximating sequence $(\varphi_n)_n \subset \Lipb(V)$ since,
by Lemma  \ref{l:cutoff} (recall that the lemma  holds in 
$\calX^{\Young}$ for \emph{any} Young function $\Young$), we may suppose
that $\sup_n\|\varphi_n\|_\infty<+\infty$. 
\end{proof}

We are now in a position to carry out the 
\begin{proof}[Proof of Theorem \ref{thm:density-vindicated}]
 By the assumed density condition \eqref{Ass:F-bis}, for every 
$\varphi \in \mathcal{X}_2$
 there exists a sequence $(\varphi_n)_n \subset 
\Lipb(V)$ suitably
approximating $\varphi$, with $\sup_n \|\varphi_n\|_\infty<\infty$ thanks to Lemma
\ref{l:cutoff}. 
 We now send $n\to \infty$ in  the continuity equation tested by $\varphi_n$, i.e.\ in 
  \[
        \int_V\varphi_n(x)\,\rho_t (\dd x) - \int_V\varphi_n(x)\, \rho_s  (\dd x) = \iiint_{[s,t]{\times}E} \dnabla\varphi_n(x,y)\, \jj_{\!\Lebone}(\dd r \dd x \dd y) \   \text{ for all } 0\leq s \leq t\leq T\,.
            \]
            Since $\varphi_n \weakto \varphi$ in $\rmL^1(V;\pi)$, we pass to  the limit on the left-hand side. 
            As for the right-hand side, 
            we will show that 
            \begin{equation}
            \label{RHSto0}
             \iiint_{[s,t]{\times}E}  \dnabla\varphi_n(x,y)  \, \jj_{\!\Lebone}(\dd r \dd x \dd y)  \longrightarrow      \iiint_{[s,t]{\times}\Ed}   \dnabla\varphi(x,y)  \, \jj_{\!\Lebone}(\dd r \dd x \dd y) \,.
            \end{equation}
          For this, we use that 
            \[
            \begin{aligned}
          & \iiint_{[s,t]{\times}E}   \left|\dnabla\varphi_n(x,y){-} \dnabla\varphi(x,y) \right|  \, \jj_{\!\Lebone}(\dd r \dd x \dd y) 
         \\
         & =\frac1{\beta}\int_s^t  \iint_{E_\upalpha^r} \beta   \left|\dnabla\varphi_n(x,y){-} \dnabla\varphi(x,y) \right|
          \frac{\tfrac12 w_r(x,y)}{\upalpha(u_r(x), u_r(y))} \upalpha(u_r(x), u_r(y)) \tetapi(\dd x \dd y ) \, \dd r
         \\
         & \stackrel{(1)}\leq  \frac1{2\beta} \EEE \int_s^t  \iint_{E_\upalpha^r} \uppsi^*(\beta (\dnabla\varphi_n(x,y){-} \dnabla\varphi(x,y) )) \upalpha(u_r(x), u_r(y)) \tetapi(\dd x \dd y ) \, \dd r
        \\
        & \quad 
         +   \frac1{2\beta} \EEE \int_s^t  \iint_{E_\upalpha^r} \uppsi\left( \frac{w_r(x,y)}{\upalpha(u_r(x), u_r(y))} \right) \upalpha(u_r(x), u_r(y)) \tetapi(\dd x \dd y ) \, \dd r
         \end{aligned}
                     \]
          with the set   $E_\upalpha^r = \{ (x,y)\in E\, : \ \upalpha(u_r(x), u_r(y))>0 \} $ fulfilling $|\bj_r|(E{\backslash} E_\upalpha^r )=0$ \EEE
           for   $\Lebone$-a.e.\ $r\in (0,T)$, cf.\ \eqref{nice-representation}, and {\footnotesize (1)} due to Young's inequality.
       Now, for every \emph{fixed} $\beta>0$ we have
       \[
       \begin{aligned}
       &
       \int_s^t  \iint_{E_\upalpha^r} \uppsi^*(\beta (\dnabla\varphi_n(x,y){-} \dnabla\varphi(x,y) )) \upalpha(u_r(x), u_r(y)) \, \tetapi(\dd x \dd y ) \, \dd r
       \\
       &\qquad 
       \leq \overline{C} (t{-}s)  \iint_{E} \uppsi^*(\beta (\dnabla\varphi_n(x,y){-} \dnabla\varphi(x,y) )) \, \tetapi(\dd x \dd y ) \, \dd r \doteq  \overline{C}  (t{-}s) I_n
       \end{aligned}
       \]
      with $ \overline{C} : = \sup_{t\in [0,T]} \sup_{(x,y) \in E} \upalpha(u_t(x), u_t(y))<+\infty$ as $u \in \rmL^\infty(0,T;\rmL^\infty(\pi))$.
      
       All in all, we conclude that for every $\beta>0$
       \[
       \begin{aligned}
       &
   \limsup_{n\to \infty}     \iiint_{[s,t]{\times}E}   \left|\dnabla\varphi_n(x,y){-} \dnabla\varphi(x,y) \right|  \, \jj_{\!\Lebone}(\dd r \dd x \dd y) \\ 
&\hspace{4em} \leq  \frac{\overline{C}}{2\beta}  (t{-}s) \lim_{n\to\infty} I_n + \frac1{\beta}\int_s^t   \scrR(\rho_r,\bj_r)\, \dd r
   =  \frac1{\beta}\int_s^t   \scrR(\rho_r,\bj_r)\, \dd r\,.
   \end{aligned}
       \]
       Indeed, $I_n \to 0$ since the density property also holds in $\mathcal{X}^{\uppsi^*}$
       by Lemma \ref{cor:Oliver}. \EEE
         By the arbitrariness of $\beta$
         we conclude \eqref{RHSto0}, which finishes the proof of the validity of the reflecting continuity equation. 
\end{proof}       
\EEE

\section{Examples for the density property}
\label{app:examples-density}
In this section, we revisit the prototypical example of the introduction to show how the density property \eqref{Ass:F-bis} can fail, and then proceed by providing various examples of functional analytic setups in which the density property does in fact hold. 

\subsection{Failure of the density property}
\label{ss:7.1}


Let us consider Example \ref{ex:Jasp-intro-2}, where for convenience we repeat the setting here. We set $V:=\Omega=[-1,1]$ with the usual Euclidean metric $\dV(x,y):=|x-y|$ and corresponding topology, take the uniform measure $\pi:=\Lebone$, and split up the space in $\Omega_1=[-1,0)$ and $\Omega_2=(0,1]$, and define the kernel
\[\kappa(x,\dd y) = a(x,y) |x{-}y|^{-(1{+}2s)}, \qquad a(x,y):=1_{\Omega_1\times \Omega_1 \cup \Omega_2 \times \Omega_2}(x,y).\]

Recall that the space $\mathcal{X}_2$ is given by 
\begin{equation}
\label{X2introA2}
    \mathcal{X}_2 = \left\{ \varphi \in \rmL^\infty( V;\pi)  \, : \ \text{Borel, bounded, with }  \iint_{V{\times}V} |\dnabla \varphi(x,y)|^2\ \tetapi(\dd x \dd y)<+\infty \right\},
\end{equation}
and in particular contains all functions that are constant on $\Omega_1$ and $\Omega_2$ but possibly discontinuous at the boundary, i.e.
\[ \left\{f : \, f=a \boldsymbol{1}_{[-1,0)}+b \boldsymbol{1}_{\{0\}} + c \boldsymbol{1}_{(0,1]}, \quad (a,b,c)\in \R^3. \right\}\]

We will establish that any closure of Lipschitz functions will automatically consist of continuous functions, for sufficiently singular kernels, showing that the density property does not hold. 

\begin{prop}
 \label{failure}
Let $s>\tfrac{1}{2}$ and consider any sequence of Lipschitz functions $(\varphi_n)_{n} \subset \mathrm{Lip}_b(V)$ converging to a Lipschitz function  $\varphi \in \mathcal{X}_2$ in the sense of \eqref{Ass:F-bis}. Then $\varphi\in \mathrm{C}_{\mathrm{b}}(\Omega)$.  
 \end{prop}

 \begin{proof}
Our notion of convergence implies that, in particular, 
\[\sup_{n\in \N} \iint_{\Omega_i} |\dnabla \varphi_n(x,y)|^2\ \tetapi(\dd x \dd y) < +\infty, \quad \sup_{n\in \N} \|\varphi_n\|_{\infty} < +\infty.\]
Applying fractional Morrey-type inequalities (see \cite[Theorem 8.2]{DiNezza2012}) and the uniform bound on $\varphi_n$, we find that there exists a constant $C>0$ such that 
\[ \sup_{n\in \N} \|\varphi_n\|_{\rmC^{0,\alpha}(\Omega_i)} \leq C,\]
with $\alpha:=s-\tfrac{1}{2}$. But since the functions $\varphi_n$ are continuous at $0$, the uniform continuity over $[0,1)$ and $(0,1]$ is enough to obtain uniform continuity over $x=0$ as well (with a slightly larger constant $2C$ as upper bound), from which we can deduce that the limiting function $\varphi$ is continuous. 
 \end{proof}

As hinted at in the Introduction, the singularity $s>\tfrac{1}{2}$ is precisely where certain nonlocal analogues of partial integration (but still \emph{local} boundary terms) and traces for fractional kernels play, as seen in \cite{Guan2006} and subsequent works. Moreover, from a stochastic perspective, even though a pure jump process can not `hit' a boundary, for singular enough kernels the infinite amount of  small jumps allow it to still come infinitesimally close to the boundary (see \cite{Applebaum_2009}): this is why   boundary conditions to come into play. A more physical perspective of the various notions of fractional processes in bounded domains (restricted, censored, conditioned, taboo, etc.) and notions of reflection can be found in \cite{Garbaczewski2019}. 

\medskip

Finally, it is clear from the above example that using a different metric over $\Omega$, namely 
\[\dV_2(x,y) := \left\{ \begin{aligned}
&|x-y|, &\qquad &\mbox{if either $(x,y)\in \Omega_1\times \Omega_1$ or $(x,y)\in \Omega_2\times \Omega_2$, } \\
&+\infty, & \qquad &\mbox{otherwise,}
\end{aligned}\right.\]
leads to a separable metric space $(V,\dV_2)$ where Lipschitz functions \emph{are} dense in $\calX_2$, and in that sense the metric $\dV_2$ better connects to the reflecting setting, revealing any hidden reflecting interfaces. 

Unfortunately, even though any Dirichlet form can be made regular (see \cite{Fukushima1994}) by changing the topology, which connects to density of continuous functions, there do not exist, as far as the authors are aware, any results that show that for any active reflected Dirichlet form there exists a metric such that corresponding Lipschitz functions are dense. 

This contrasts with the local setting, i.e.\ the case of metric measure spaces, where under suitable assumptions there exists a unique metric associated with any Dirichlet form, and is partly due to the fact that for any two metrics $\dV_1$ and $\dV_2$ satisfying the fundamental estimate \eqref{eq:mainm}, their maximum $\dV_1{\vee} \dV_2$ need not satisfy \eqref{eq:mainm}, as  discussed in \cite{Frank2014}.

 \subsection{Assumption   \ref{Ass:F-bis} in a bounded domain of $\R^d$}
  \label{ss:7.2}
 %
 
 \noindent
 Firstly, we recall some basics on \emph{nonlocal integrodifferential operators of L\'evy type} on \emph{bounded} domains,
  in particular following 
the PhD dissertation \cite{GuyDiss} (see also the references therein) and the paper \cite{FogKas24}. 
We start by   considering a positive function $\lds: \R^d{\backslash} \{0\} \to [0,\infty) $
(recall  that  $\Lebone^d$ is the Lebesgue measure on $\R^d$)
fulfilling the following conditions:
\begin{enumerate}
\item \emph{L\'evy-integrability}:
\begin{equation}
\label{lds1}
\tag{$\lds.1$} 
\int_{\R^d} (1{\wedge}|h|^2) \lds(h) \, \Lebone^d(\dd h) \doteq  C_{\nu} < +\infty;
\end{equation}
\item symmetry:
\begin{equation}
\label{lds2}
\tag{$\lds.2$} 
\lds(h) = \lds({-}h) \quad \text{for all } h \in \R^d\,.
\end{equation}
\end{enumerate}
Namely, the  function $\lds$ is the density of the \emph{symmetric L\'evy measure}  $\lds \, \Lebone^d$. 
Let us now consider
\[
\text{a bounded Lipschitz domain } \Omega \subset \R^d\,.
\]
It can be shown 
(cf.\ the discussion in \cite[Chap.\ 2]{GuyDiss}),
that 
for every $v,\, w \in \rmC^\infty (\Omega)$ the integral
\[
\DRd_{\lds}(v,w) := \frac12 \iint_{\Omega{\times}\Omega}   (v(x){-}v(y)) (w(x){-}w(y))\, \lds(x{-}y)  \Lebone^d(\dd y)  \Lebone^d(\dd x) 
\]
is finite.  It is thus natural to introduce the space 
\[
\mathrm{H}_{\lds}(\Omega) : = \{ v \in \rmL^2(\Omega)\, : \ \DRd_{\lds}(v,v)<+\infty\}\,,
\]
and endow it with the norm
\[
\|v\|_{\mathrm{H}_{\lds}(\Omega)}^{2}: = \| v \|_{\rmL^2(\Omega)}^2 + \DRd_{\lds}(v,v)\,.
\]
 The space $\mathrm{H}_{\lds}(\R^d)$ is analogously defined. 
\par
As shown in 
\cite[Thm.\ 3.46]{GuyDiss} (cf.\ also
\cite[Rmk.\ 2.8]{FogKas24},  $(\mathrm{H}_{\lds}(\Omega), \| {\cdot}\|_{\mathrm{H}_{\lds}(\Omega)} )$ is a separable Hilbert space.
Moreover,  \cite[Lemma 3.47]{GuyDiss}, 
$H^1(\Omega) \subset \mathrm{H}_{\lds}(\Omega)$ continuously. The prototypical example  is given the fractional Sobolev space $H^{\mathrm{s}}(\Omega)$. 
\begin{example}[The fractional space and the fractional Laplacian]
\upshape
Starting from the  even  function $\lds_{\mathrm{s}}:\R^d{\backslash}\{0\} \to  [0,\infty) $ given by 
\[
\lds_{\mathrm{s}}(h) = C_{d,\mathrm{s}} \frac{1}{\displaystyle  |h|^{d{+}2\mathrm{s}}}, \qquad  \mathrm{s} \in (0,1),
\]
with $ C_{d,\mathrm{s}}$  a suitable constant, one obtains the 
Sobolev-Slobodeckij space $\mathrm{H}^{\mathrm{s}}(\R^d)$ via this construction. 
The corresponding energy form is 
\[
\DRd_{\lds_{\mathrm{s}}}(v,w) := \frac12 \iint_{\Omega{\times}\Omega}  C_{d,\mathrm{s}} \frac{1}{\displaystyle  |x{-}y|^{d{+}2\mathrm{s}}} \,
(v(x){-}v(y)) (w(x){-}w(y))\,  \Lebone^d(\dd y)  \Lebone^d(\dd x) \,.
\]
 The fractional Laplacian 
$\mathrm{A}_{\mathrm{s}}:  \mathrm{H}^{\mathrm{s}}(\Omega) \to \mathrm{H}^{\mathrm{s}}(\Omega)^*$ is the associated operator
\[
\pairing{\mathrm{H}^{\mathrm{s}}(\Omega)^*}{\mathrm{H}^{\mathrm{s}}(\Omega)}{\mathrm{A}_{\mathrm{s}} v}{w}: = \DRd_{\lds_{\mathrm{s}}}(v,w) \qquad \text{for all } v,\,w \in \mathrm{H}^{\mathrm{s}}(\Omega) \,.
\]
\end{example}
\noindent A Meyer-Serrin density type result holds for $\mathrm{H}_{\lds}(\Omega)$, 
\cite[Thm.\ 3.67]{GuyDiss}, \cite[Thm.\ 2.12]{GuyDiss}. 
 \begin{prop}
 \label{meyers-serrin}
 Let $\lds$ fulfill \eqref{lds1} \& \eqref{lds2}. Then,
 the space $\mathrm{C}^\infty(\Omega)$ is dense in $\mathrm{H}_{\lds}(\Omega)$. 
 \end{prop}
\noindent A fortiori, one indeed has that $\DRd$ is a \emph{Dirichlet form} on $\mathrm{H}_{\lds}(\Omega)$, cf.\ \cite[Prop.\ 2.25]{FogKas24}. 
%
\medskip

\paragraph{\bf Density property in a bounded domain of $\R^d$.}
Let us consider 
\begin{subequations}
\label{Rdsetup}
\begin{itemize}
\item[-]
 the measure space
\[
 (V,\pi) = (\Omega, \Lebone^d \restr{\Omega} )
\]
 \item[-]
 the kernels $(\kappa(x,\cdot))_{x \in \Omega} \subset  \Ms^+(\Omega{\backslash} \{x\})$ \EEE  defined by 
  \[
 \kappa(x, \dd y  ) = \lds (x{-}y) \, \Lebone^d \restr{\Omega} (\dd y)
  \]
   and then extended to the whole of $\Omega$ via $\kappa(x,\{x\})=0$.
 \end{itemize}
 \end{subequations}
 Observe that they comply with \eqref{mitigation of singularity}, since
 for every $x\in \Omega$ we have that 
 $\int_{\Omega} (1{\wedge}|x{-}y|^2)  \kappa(x, \dd y  ) = \int_{\Omega} (1{\wedge}|x{-}y|^2)   \lds (x{-}y)
 \, \Lebone^d (\dd y) = C_{\lds}<+\infty$ by \eqref{lds1}, and that the detailed balance condition \eqref{DBC} holds, since the function $\lds$ is symmetric.
 \par
Now, it is immediate to check that 
the space $\mathcal{X}_2$ from \eqref{eq:test_extend-quadr}  is indeed
\[
\mathcal{X}_2 = \rmL^\infty(\Omega) \cap \mathrm{H}_{\lds}(\Omega) : = \bigl\{ \varphi \in \rmL^\infty(\Omega)\, : \ \DRd_{\lds}(\varphi,\varphi)<+\infty \bigr\}\,.
\]
Thanks to Proposition 
\ref{meyers-serrin}, 
for every  $\varphi \in \mathcal{X}_2 \subset  \mathrm{H}_{\lds}(\Omega)$ there exists $(\varphi_n)_n \subset \mathrm{C}^\infty(\Omega) \subset \Lipb(\Omega)$ such that 
$\|\varphi_n{-}\varphi\|_{ \mathrm{H}_{\lds}(\Omega)} \to 0 $, namely
\[
\begin{cases}
\displaystyle \|\varphi_n{-}\varphi\|_{ \rmL^2(\Omega)} \longrightarrow 0,
\\
\displaystyle  \DRd_{\lds}(\varphi_n{-}\varphi,\varphi_n{-}\varphi) := \frac12 \iint_{\Omega{\times}\Omega}   |\varphi_n(x){-}\varphi(x){-}\varphi_n(y){+}\varphi(y) |^2
\, \lds(x{-}y)  \Lebone^d(\dd y)  \Lebone^d(\dd x)  \longrightarrow 0,
\end{cases}
\]
which in fact grants the approximation properties (1) \& (2) of \eqref{Ass:F-bis}.  We have thus shown the following
\begin{prop}
\label{prop:ex-Rd}
	The triple $(V,\pi,\kappa)$ above satisfies the density property  \eqref{Ass:F-bis}.
\end{prop}

 \subsection{The density  property  in $\R^d$}
  \label{ss:7.2-bis}

The 
 validity of \eqref{Ass:F-bis}
 can be extended to the case where 
 \begin{itemize}
\item[-]
$V=\R^d$, endowed with  a finite measure $\pi$  that
 is bounded from above and 
 below by the Lebesgue measure   $\Lebone^d $ on any ball in $\R^d$;
 \item[-]
 the kernels $(\kappa(x,\cdot))_{x \in \R^d} \subset   \Ms^+(\R^d{\backslash} \{x\})$ \EEE are a weighted version of  the L\'evy density $\lds$ from \eqref{lds1}--\eqref{lds2}. 
 \end{itemize}
Indeed, we have the following result.
\begin{prop}
\label{prop:ex-Rdu}
Consider on  $V=\R^d$ a finite measure $\pi  \in \mathcal{M}^+(\R^d)$ given by 
\begin{equation}
\label{eq:rdu1}
\begin{gathered}
 \pi (\dd x) =\gamma(x) \Lebone^d (\dd x), \qquad \text{with } 
 \\
 \gamma
\in \rmL^1(\R^d; \Lebone^d) {\cap}\rmL^\infty(\R^d; \Lebone^d) \text{  such that } 
 \text{for every ball $B\subset \R^d$} \ 
 \exists\, c_B>0 \ \forall\, x \in B\, : \,
 \gamma(x) \geq c_B\,.
 \end{gathered}
\end{equation} 
Moreover, set
\[\kappa(x,\dd y):=\gamma(y) \,\lds(y-x) \Lebone^d(\dd y).\]
Then, the density property \ref{Ass:F-bis} holds.
\end{prop}
Note that, in this setup, 
\[
\tetapi (\dd x \dd y )=  
\kappa(x,\dd y) \pi(\dd x)=\gamma(x)\gamma(y) \,\lds(y-x) \Lebone^d(\dd x)\Lebone^d(\dd y).\]
 Before carrying out the proof,  \EEE we introduce 
the weighted fractional Sobolev space
\[
\begin{aligned}
\mathrm{H}_{\lds,\gamma} (\R^d): &= \{ v \in \rmL^2(\R^d;\pi)\, : \ \DRd_{\lds,\gamma}(v,v)<+\infty\}\,
\qquad \text{with } 
 \\
\DRd_{\lds,\gamma}(v,w) &:= \frac12 \iint_{\R^d{\times} \R^d}   (v(x){-}v(y)) (w(x){-}w(y))\, \gamma(x)\gamma(y)\lds(x{-}y)  \Lebone^d(\dd y)  \Lebone^d(\dd x) 
\end{aligned}
\]
and endow it with the norm
\[
\|v\|_{\mathrm{H}_{\lds,\gamma}(\R^d)}^{2}: = \| v \|_{\rmL^2(\R^d;\pi)}^2 + \DRd_{\lds,\gamma}(v,v)\,.
\]
\begin{proof}[Proof of Proposition \ref{prop:ex-Rdu}]
Let $(B_m)_{m\in \mathbb{N}}$ be the sequence of closed balls  centered at $0$   with radius $m\geq 1$, and $(f_m)_m$ a corresponding sequence of smooth and compactly supported multipliers, with the property of being $2$-Lipschitz, and $f_m\equiv1$ on $B_m$, $f_m\equiv0$ on $B_{m+1}^c= \R^d{\backslash}B_{m+1}$ for all $m\in \mathbb{N}$. Moreover, let $c_m>0$ 
be the corresponding lower bound of $\gamma$ on $B_m$, cf.\ \eqref{eq:rdu1}. 
\par
Take any $\varphi\in \mathrm{H}_{\lds,\gamma}(\R^d)$.  As in the proof of Lemma \ref{l:cutoff} (see Appendix 
\ref{app:density-vindicated}), \EEE
 via a scaling and truncation argument, we can assume without loss of generality that $\|\varphi\|_{\rmL^\infty(\pi)}\leq 1$.
 Setting $\varphi_m:=f_m \varphi$, it is clear that the sequence $(\varphi_m)_m$ converges to $\varphi$ in $\mathrm{H}_{\lds,\gamma}(\R^d)$ as $m\to\infty$:  indeed, to show that 
$ \DRd_{\lds,\gamma}(\varphi_m{-}\varphi,\varphi_m{-}\varphi) \longrightarrow 0$, it suffices to combine the pointwise convergence 
$\overline\nabla (\varphi_m{-}\varphi)(x,y) 
\to 0$ for $(\Lebone^d{\otimes}\Lebone^d)$-a.e.\ $(x,y) \in \R^d{\times} \R^d$ with the estimate \EEE
\begin{equation}\label{eq:rdu2}
\begin{aligned}
 |\overline{\nabla} \varphi_m(x,y)|^2  \EEE  &\leq 2\varphi^2 (x)|\overline{\nabla} (f_m)|^2+f_m^2 (y) |\overline{\nabla} \varphi|^2 \\
&\leq 4 \left(1{\wedge} |x{-}y|^2\right)+2|\overline{\nabla} \varphi(x,y) |^2 \quad \text{for $(\Lebone^d{\otimes}\Lebone^d)$-a.e.\ $(x,y) \in \R^d{\times} \R^d$} 
\end{aligned}
\end{equation}
 which provides an integrable majorant. By dominated convergence we have $ \DRd_{\lds,\gamma}(\varphi_m{-}\varphi,\varphi_m{-}\varphi) \to~0$. 

\EEE By a diagonal argument it is therefore sufficient to find for every $m\in \N$ a sequence of smooth $(\varphi_{m,n})_n$ converging
as $n\to \infty$
 to $\varphi_{m}$ in $\mathrm{H}_{\lds,\gamma}(\R^d)$.  Thus, fix $m \in \mathbb{N}$  and observe that 
 \[
\varphi_m \in  \mathfrak{C}_m: = \{ g\in \rmL^\infty(\R^d;\pi) \, : \  \|g\|_{\\rmL^\infty(\pi)}\leq 1, \ g \equiv 0 \text{ in }\R^d{\backslash} B_{m+1}\}\,.
 \]
 \par
 First, note that finiteness of and the convergence in the various norms used above are equivalent in 
 the class $\mathfrak{C}_m$. 
Clearly, it suffices to argue for 
convergence: thus, we will  show that,
for any $(g_n)_n,\, g $ in $\mathfrak{C}_m$,
\begin{equation}
\label{equivalence}
g_n \to g \quad \text{in }   \mathrm{H}_{\lds}(\R^d) \ \Longleftrightarrow \ g_n \to g \quad \text{in }   \mathrm{H}_{\lds,\gamma}(\R^d)\,.
\end{equation}
For, if $g_n \to g$ in $ \mathrm{H}_{\lds}(\R^d)$, then \EEE
$g_n(x)$ converges to $g(x)$ for a.e. $x\in V$.
 Let us now set $h_n:=(g-g_n)$. 
Via conditioning on the distance we write 
\[|\overline{\nabla} h_n(x,y)| ^2= \boldsymbol{1}_{d\leq 1}(x,y)|\overline{\nabla} h_n(x,y)|^2+\boldsymbol{1}_{d>1}(x,y)|\overline{\nabla} h_n(x,y)|^2\,,\]
where, recall, 
$\boldsymbol{1}_A$ denotes the indicator function of a set $A$ and 
 we have used the short-hand notation
\[
\begin{cases}
\boldsymbol{1}_{d\leq 1} = \boldsymbol{1}_{\{ (x,y)\in \R^d\, : \ |x{-}y|\leq 1\}}\,,
\\
\boldsymbol{1}_{d>1} = \boldsymbol{1}_{\{ (x,y)\in \R^d\, : \ |x{-}y|> 1\}}\,.
\end{cases}
\]
Now, taking into account the fact that 
$\boldsymbol{1}_{d\leq 1}|\overline{\nabla} h_n|^2$ is supported on $B_{m+2}{\times}B_{m+2}$, we estimate
\[
\begin{aligned}
&
\iint_{\R^d{\times}\R^d}  \EEE \boldsymbol{1}_{d\leq 1}(x,y)|\overline{\nabla}h_n(x,y)|^2 \tetapi(\dd x,\dd y)
\\
&\quad
\leq
 \|\gamma\|_{\rmL^\infty}^2 \iint_{B_{m+2}{\times} B_{m+2}} |\overline{\nabla} h_n(x,y)|^2 \lds(x{-}y)\Lebone^d(\dd y)  \Lebone^d(\dd x) \leq  \|\gamma\|_{\rmL^\infty}^2(\Lebone^d) \|h_n\|_{\mathrm{H}_{\lds}(\R^d)}^{2}
 \longrightarrow 0 
 \,.
 \end{aligned}
 \]
 
In turn, we have that 
 \[
 \iint_{\R^d{\times}\R^d}  \boldsymbol{1}_{d> 1}(x,y)|\overline{\nabla}h_n(x,y)|^2 \tetapi(\dd x,\dd y) \longrightarrow 0 
 \]
by dominated convergence,
 as the integrand  pointwise converges to zero, and 
 relying on   
\[
\begin{aligned}
&\iint_{\R^d{\times}\R^d}   \boldsymbol{1}_{d> 1}(x,y)\gamma(x)\gamma(y)\lds(x{-}y)  \Lebone^d(\dd y)  \Lebone^d(\dd x) 
\\
&\qquad = 
 \iint_{\R^d{\times}\R^d}   (1{\wedge}|x{-}y|^2)\lds(x{-}y) \gamma(x)\gamma(y)  \Lebone^d(\dd y)  \Lebone^d(\dd x)
 \leq \|\gamma\|_{\rmL^\infty}(\Lebone^d)
C_{\lds} \pi(\R^d).
\end{aligned}
\]
 Conversely, \EEE  with $g_n$ now to converging to $g$ in $\mathrm{H}_{\lds,\gamma}(\R^d)$, we have
\[
\begin{aligned}
&
\iint_{\R^d{\times}\R^d} \boldsymbol{1}_{d> 1}(x,y)|\overline{\nabla}h_n(x,y)|^2 \lds(x{-}y)\Lebone^d(\dd y)  \Lebone^d(\dd x)
\\
&
 \leq  c_{m+2}^{-2}\EEE \iint_{B_{m+2}{\times}B_{m+2}}
 |\overline{\nabla} h_n(x,y)|^2 \tetapi(\dd x,\dd y) \leq   c_{m+2}^{-2} \EEE  \|h_n\|_{\mathrm{H}_{\lds,\gamma}(R^d)}^{2},
 \end{aligned}
 \]
while 
\[
\iint_{B_{m+1}{\times} B_{m+1}^c} \boldsymbol{1}_{d> 1}(x,y) \lds(x{-}y)  \Lebone^d(\dd y)  \Lebone^d(\dd x) \leq C_{\lds} \Lebone^d(B_{m+1}) . \]
 Hence, we conclude that $g_n\to g $ in $\mathrm{H}_{\lds}(\R^d)$ as $n\to\infty$, and \eqref{equivalence} is proven. \EEE

 In order to approximate each $\varphi_m \in \mathfrak{C}_m {\cap} \mathrm{H}_{\lds}(\R^d)$ by a sequence  of smooth $(\varphi_{m,n})_n$, let us then argue in this way.  \EEE
Since $\rmC_{\mathrm{c}}^{\infty}(\R^d) {\cap}   \mathrm{H}_{\lds}(\R^d)$ is dense in $\mathrm{H}_{\lds}(\R^d)$ by \cite[Thm.\ 3.66]{GuyDiss},
 one can find  a sequence $(g_{m,n})_m$ converging to $\varphi_m$ in $\mathrm{H}_{\lds}(\R^d)$. The result now follows after setting $\varphi_{m,n}:=(f_{m+1}g_{m,n})\wedge 1$, which can be shown to converge to $\varphi_m$ in $\mathrm{H}_{\lds}(\R^d)$ as well, are bounded by $1$, and 
 indeed fulfill $(\varphi_{m,n})_n \subset \mathfrak{C}_m$. 
  By  the equivalence proved in the above lines, we have $\varphi_{m,n} \to \varphi_m $ in $\mathrm{H}_{\lds,\gamma}(\R^d)$ as $n\to\infty$.  
 \par
  This concludes the proof. \EEE
\end{proof}

\subsection{Assumption   \ref{Ass:F-bis} in the torus}
We take as $V$ the $d$-dimensional torus $\torus$ (i.e., the quotient space $\R^d/\mathbb{Z}^d$ w.r.t.\ to the relation $x \sim y$ iff $x-y \in \mathbb{Z}^d$), endowed with the
reference   measure $ \pi = \Lebone^d$. Let us consider two L\'evy densities $\nu$ and $\widehat \nu$ complying with \eqref{lds1} and \eqref{lds2}, and
the induced 
 kernels
\[
\kappa (x,\dd y ) = \nu(x{-}y) \Lebone^d(\dd y) \in \Mloc^{+}(\torus{\backslash}\{x\}),  \qquad \widehat\kappa (x,\dd y ) = \widehat{\nu}(x{-}y) \Lebone^d(\dd y) \in \Mloc^{+}(\torus{\backslash}\{x\})\,.
\]  
We consider the triple
$
(\VV,\ppi,\kkappa) 
$ with 
\[
\VV = \torus{\times}\torus, 
\qquad \ppi = \Lebone^d{\otimes}\Lebone^d, \qquad 
 \kkappa (\xx, \dd \yy):= \kappa (x, \dd y ) \delta_{\hat x}(\dd \hat y) +  \widehat{\kappa} (\hat x, \dd \hat y ) \delta_{x}(\dd y)\,.
\]
We will show the validity of the density condition in this context: note, in particular, that, in contrast to the setup of Section \ref{ss:7.2}, here the kernel
$\kappa$
is \emph{not} absolutely continuous w.r.t.\ the Lebesgue measure.
\begin{prop}
The density property \eqref{Ass:F-bis} holds for the triple $(\VV, \ppi, \TTheta_{\kkappa})$.
\end{prop}
\begin{proof}
Let $(\mu_n)_n,\, (\widehat{\mu}_n)_n,\  \subset \rmC_{\mathrm{c}}^\infty (\torus)$ be a standard family of mollifiers, fulfilling for every $k\in \N$
\begin{equation}
\label{torus-mollifiers}
\begin{aligned}
&
\mathrm{\supp}(\mu_n),\, \mathrm{\supp}(\widehat{\mu}_n) \subset B_{r_n}(0)  \text{ for a null sequence $r_n \downarrow 0$}\,,
\\
&
\int_{\torus} \mu_n(x) \dd x = \int_{\torus} \widehat{\mu}_n(x) \dd x  =1\,.
\end{aligned}
\end{equation}
Set
\[
\mmu_n (\xx) = \mu_n(x) \widehat{\mu}_n(\hat x) \qquad \text{for all } \xx = (x,\hat x) 
\in \VV\,.
\]
\par
Let us fix $\varphi \in \rmL^\infty(\VV;\ppi)$, and 
define
\[
\varphi_n(\xx): = \int_{\VV} \varphi(\xx{-}\zz) \mmu_n (\zz) \dd \zz = \iint_{\torus{\times}\torus}  \varphi(x{-}z,
\hat{x}{-}\hat{z}) \mu_n(z) \widehat{\mu}_n(\hat{z})  \dd z \dd \hat{z}\,. 
\]
It is well known that 
\begin{equation}
\label{cvg-L1-torus}
\varphi_n (\xx) \to \varphi(\xx) \quad \text{for } \ppi\text{-.a.e.} \xx \in \VV 
\qquad \text{ and }  \qquad 
\text{$\varphi_n \to \varphi $ in $\rmL^1(\VV;\ppi)$.}
\end{equation}
 We easily obtain that 
 \begin{equation}
 \label{pointwise-tetakk}
 \overline{\nabla} \varphi_n (\xx,\yy) \to \overline{\nabla}\varphi(\xx,\yy) \quad \text{for }  \TTheta_{\kkappa}\text{-a.e.} (\xx,\yy) \in \EE\,.
 \end{equation}
  Moreover, we observe that   for $ \TTheta_{\kkappa}\text{-a.e.} (\xx,\yy) \in \EE$ there holds, by Jensen's inequality,
 \begin{equation}
 \label{lessGn}
 \begin{aligned}
| \overline{\nabla} \varphi_n (\xx,\yy)|^2  & = \left| 
\iint_{\torus{\times}\torus} \varphi (\xx{-}\zz)  \mmu_n (\zz) \dd \zz  {-} \iint_{\torus{\times}\torus} \varphi (\yy{-}\zz)  \mmu_n (\zz) \dd \zz 
\right|^2
\\ & \leq  
\iint_{\torus{\times}\torus} |\varphi (\xx{-}\zz){-} \varphi (\yy{-}\zz)  |^2 \mmu_n (\zz) \dd \zz  \doteq G_n(\xx,\yy)\,.
\end{aligned}
 \end{equation}
Now, by the definition of $\kkappa$ and of the individual kernels $\kappa$ and $\widehat\kappa$ we have 
\[
\begin{aligned}
&
\iint_{\EE} G_n(\xx,\yy)  \TTheta_{\kkappa}(\dd \xx, \dd y) = I_n^1+I_n^2 \qquad \text{with }
\\
& 
I_n^1 = \int_{\torus} \iint_{\Ed} G_n(x,\hat x, y, \hat x) \nu(y{-}x)  \dd y\, \dd x\, \dd \hat x\,,
\\
& 
I_n^2 =  \int_{\torus} \iint_{\Ed} G_n(x,\hat x, x, \hat y) \widehat{\nu}(\hat{y}{-}\hat x) \dd \hat y \, \dd \hat x\,  \dd x\,,
\end{aligned}
\]
where in this context $\Ed = \torus{\times}\torus \backslash \{ (x,x)\, : \ x \in \torus\}$. 
We observe that, by Fubini's theorem, 
\[
\begin{aligned}
I_n^1 &  =  \iint_{\torus{\times}\torus} \left(  \int_{\torus} \iint_{\Ed} |\varphi (x{-}z, \hat x {-}\hat z){-} \varphi (y{-}z, \hat x {-}\hat z)  |^2 
 \nu(y{-}x) \dd y \, \dd x \dd \hat x \right) \, \mu_n (z)  \widehat{\mu}_n (\hat z)  \,  \dd z\, \dd \hat z 
 \\
 & 
 \stackrel{(1)}{=} \iint_{\torus{\times}\torus} \left( \int_{\torus} \iint_{\Ed} |\varphi (x', \hat x'){-} \varphi (y', \hat x')  |^2 
 \nu(y'{-}x') \dd y' \, \dd x' \,\dd \hat x' \right) \,  \mu_n (z)  \widehat{\mu}_n (\hat z) \, \dd z \,\dd \hat z 
 \\
 & 
 \stackrel{(2)}{=} \int_{\torus} \iint_{\Ed} |\varphi (x', \hat x'){-} \varphi (y', \hat x')  |^2 
 \nu(y'{-}x') \dd y' \, \dd x' \,\dd \hat x' 
 \end{aligned}
\]
Here {\footnotesize (1)} follows from the change of variariable $\xx'=\xx{-}\zz $, $\yy' = \yy{-}\zz$
(under which the domains of integration do not change, i.e.\ for all $\hat z \in 
\torus$ we have 
$\torus - \{\hat z\} = \torus$ and similarly for $\Ed$). Eventually,  {\footnotesize (2)}  is due to the fact that 
\[
 \iint_{\torus{\times}\torus}  \mu_n (z)\,\widehat{\mu}_n (\hat z) \, \dd z \,\dd \hat z  = 1
\]
by \eqref{torus-mollifiers}.  Clearly,  we may repeat the very same calculations for $I_n^2$, thus obtaining that 
\[
\iint_{\EE} G_n(\xx,\yy)  \TTheta_{\kkappa}(\dd \xx, \dd \yy) = 
\iint_{\EE}  |\varphi (\xx){-} \varphi (\yy)  |^2 \,
 \TTheta_{\kkappa} (\dd x \dd \yy)
\]
All in all, it follows from \eqref{lessGn} that 
\[
\limsup_{n\to\infty}  \iint_{\EE}  | \overline{\nabla} \varphi_n (\xx,\yy)|^2   \TTheta_{\kkappa} (\dd \xx \dd \yy) \leq \iint_{\EE}  |\varphi (\xx){-} \varphi (\yy)  |^2 \,
 \TTheta_{\kkappa} (\dd \xx \dd \yy)\,.
\]
Since on the other hand we have that, by the Fatou Lemma,
\[
\liminf_{n\to\infty}  \iint_{\EE}  | \overline{\nabla} \varphi_n (\xx,\yy)|^2   \TTheta_{\kkappa} (\dd \xx \dd \yy) \geq \iint_{\EE}  |\varphi (\xx){-} \varphi (\yy)  |^2 \,
 \TTheta_{\kkappa} (\dd \xx \dd \yy)
\]
we ultimately have $\iint_{\EE}   | \overline{\nabla} \varphi_n |^2\, \dd   \TTheta_{\kkappa} \to \iint_{\EE}   | \overline{\nabla} \varphi |^2 \, \dd   \TTheta_{\kkappa}  $.
Combining this with the pointwise convergence \eqref{pointwise-tetakk},
 a version of the dominated convergence theorem yields
the desired convergence 
\begin{equation}
\label{FINAL-torus}
\overline\nabla \varphi_n \longrightarrow \overline\nabla\varphi \qquad \text{in } \rmL^2(\EE; \TTheta_{\kkappa} )\,.
\end{equation}
By \eqref{cvg-L1-torus} and \eqref{FINAL-torus}, we conclude the validity of the density property. 
\end{proof}

\EEE

{\small

\markboth{References}{References}

\bibliographystyle{siam}
\bibliography{ricky_lit}

\begin{thebibliography}{10}

\bibitem{Adams-Dirr-Peletier-Zimmer}
{\sc S.~Adams, N.~Dirr, M.~A. Peletier, and J.~Zimmer}, {\em From a
  large-deviations principle to the {W}asserstein gradient flow: a new
  micro-macro passage}, Comm. Math. Phys., 307 (2011), pp.~791--815.

\bibitem{AmFuPa05FBVF}
{\sc L.~Ambrosio, N.~Fusco, and D.~Pallara}, {\em Functions of Bounded
  Variation and Free Discontinuity Problems}, Oxford University Press, 2005.

\bibitem{AGS08}
{\sc L.~Ambrosio, N.~Gigli, and G.~Savar{\'e}}, {\em Gradient flows in metric
  spaces and in the space of probability measures}, Lectures in Mathematics ETH
  Z\"urich, Birkh\"auser Verlag, Basel, second~ed., 2008.

\bibitem{Ambrosio-Gigli-Savare14}
{\sc L.~Ambrosio, N.~Gigli, and G.~Savar\'{e}}, {\em Calculus and heat flow in
  metric measure spaces and applications to spaces with {R}icci bounds from
  below}, Invent. Math., 195 (2014), pp.~289--391.

\bibitem{Applebaum_2009}
{\sc D.~Applebaum}, {\em Lévy Processes and Stochastic Calculus}, Cambridge
  Studies in Advanced Mathematics, Cambridge University Press, 2~ed., 2009.

\bibitem{bertini2015macroscopic}
{\sc L.~Bertini, A.~De~Sole, D.~Gabrielli, G.~Jona-Lasinio, and C.~Landim},
  {\em Macroscopic fluctuation theory}, Reviews of Modern Physics, 87 (2015),
  pp.~593--636.

\bibitem{Bogachev07}
{\sc V.~I. Bogachev}, {\em Measure theory. {V}ol. {I}, {II}}, Springer-Verlag,
  Berlin, 2007.

\bibitem{Bourbaki2004IntI}
{\sc N.~Bourbaki}, {\em Elements of Mathematics: Integration I}, Springer
  Berlin Heidelberg, 2004.

\bibitem{Brezis2007}
{\sc H.~Brezis}, {\em How to recognize constant functions. {C}onnections with
  {S}obolev spaces}, Russian Mathematical Surveys, 57 (2007), p.~693.

\bibitem{ButtazzoBOOK89}
{\sc G.~Buttazzo}, {\em Semicontinuity, relaxation and integral representation
  in the calculus of variations}, vol.~207 of Pitman Research Notes in
  Mathematics Series, Longman Scientific \& Technical, Harlow; copublished in
  the United States with John Wiley \& Sons, Inc., New York, 1989.

\bibitem{Carrillo2024}
{\sc J.~A. Carrillo, M.~G. Delgadino, L.~Desvillettes, and J.~S.-H. Wu}, {\em
  The {L}andau equation as a gradient flow}, Analysis \&amp; PDE, 17 (2024),
  p.~1331–1375.

\bibitem{Carrillo2022}
{\sc J.~A. Carrillo, M.~G. Delgadino, and J.~Wu}, {\em Boltzmann to {L}andau
  from the gradient flow perspective}, Nonlinear Analysis, 219 (2022),
  p.~112824.

\bibitem{Chiarini2018}
{\sc A.~Chiarini and P.~Mathieu}, {\em Singular weighted {S}obolev spaces and
  diffusion processes: an example (due to {V}.{V}. {Z}hikov)}, Applicable
  Analysis, 98 (2018), p.~430–457.

\bibitem{ChowHuangLiZhou12}
{\sc S.-N. Chow, W.~Huang, Y.~Li, and H.~Zhou}, {\em Fokker--{P}lanck equations
  for a free energy functional or {M}arkov process on a graph}, Archive for
  Rational Mechanics and Analysis, 203 (2012), pp.~969--1008.

\bibitem{DiMarino2019}
{\sc S.~Di~Marino and M.~Squassina}, {\em New characterizations of sobolev
  metric spaces}, Journal of Functional Analysis, 276 (2019), p.~1853–1874.

\bibitem{DiNezza2012}
{\sc E.~Di~Nezza, G.~Palatucci, and E.~Valdinoci}, {\em Hitchhiker's guide to
  the fractional {S}obolev spaces}, Bulletin des Sciences Mathématiques, 136
  (2012), p.~521–573.

\bibitem{Erbar14}
{\sc M.~Erbar}, {\em Gradient flows of the entropy for jump processes}, Ann.
  Inst. Henri Poincar\'{e} Probab. Stat., 50 (2014), pp.~920--945.

\bibitem{Erbar16TR}
\leavevmode\vrule height 2pt depth -1.6pt width 23pt, {\em A gradient flow
  approach to the {B}oltzmann equation}, Arxiv preprint arXiv:01603.00540,
  (2016).

\bibitem{ErbarFathiLaschosSchlichting16TR}
{\sc M.~Erbar, M.~Fathi, V.~Laschos, and A.~Schlichting}, {\em Gradient flow
  structure for {M}c{K}ean-{V}lasov equations on discrete spaces}, Discrete
  Contin. Dyn. Syst., 36 (2016), pp.~6799--6833.

\bibitem{Esposito2025}
{\sc A.~Esposito, G.~Heinze, and A.~Schlichting}, {\em Graph-to-local limit for
  the nonlocal interaction equation}, Journal de Mathématiques Pures et
  Appliquées, 194 (2025), p.~103663.

\bibitem{Fathi2016}
{\sc M.~Fathi}, {\em A gradient flow approach to large deviations for diffusion
  processes}, Journal de Mathématiques Pures et Appliquées, 106 (2016),
  p.~957–993.

\bibitem{finkelshtein2010vlasov}
{\sc D.~Finkelshtein, Y.~Kondratiev, and O.~Kutoviy}, {\em Vlasov scaling for
  stochastic dynamics of continuous systems}, Journal of Statistical Physics,
  141 (2010), pp.~158--178.

\bibitem{GuyDiss}
{\sc G.~Foghem}, {\em $L^2$-Theory for Nonlocal Operators on Domains}, PhD
  thesis, Universit\"at Bielefeld, 2020.

\bibitem{FogKas24}
{\sc G.~Foghem and M.~Kassmann}, {\em A general framework for nonlocal
  {N}eumann problems}, Commun. Math. Sci., 22 (2024), pp.~15--66.

\bibitem{folland1999real}
{\sc G.~B. Folland}, {\em Real Analysis: Modern Techniques and Their
  Applications}, Pure and Applied Mathematics: A Wiley Series of Texts,
  Monographs and Tracts, Wiley, 1999.

\bibitem{Frank2014}
{\sc R.~L. Frank, D.~Lenz, and D.~Wingert}, {\em Intrinsic metrics for
  non-local symmetric {D}irichlet forms and applications to spectral theory},
  Journal of Functional Analysis, 266 (2014), p.~4765–4808.

\bibitem{Fukushima1994}
{\sc M.~Fukushima, Y.~Oshima, and M.~Takeda}, {\em Dirichlet Forms and
  Symmetric Markov Processes}, DE GRUYTER, Dec. 1994.

\bibitem{Garbaczewski2019}
{\sc P.~Garbaczewski and V.~Stephanovich}, {\em Fractional {L}aplacians in
  bounded domains: Killed, reflected, censored, and taboo {L}évy flights},
  Physical Review E, 99 (2019).

\bibitem{Grigoryan2011}
{\sc A.~Grigor’yan, X.~Huang, and J.~Masamune}, {\em On stochastic
  completeness of jump processes}, Mathematische Zeitschrift, 271 (2011),
  p.~1211–1239.

\bibitem{Guan2006}
{\sc Q.-Y. Guan}, {\em Integration by parts formula for regional fractional
  {L}aplacian}, Communications in Mathematical Physics, 266 (2006),
  p.~289–329.

\bibitem{Grka2022}
{\sc P.~Górka and A.~Słabuszewski}, {\em Embeddings of the fractional
  {S}obolev spaces on metric-measure spaces}, Nonlinear Analysis, 221 (2022),
  p.~112867.

\bibitem{ToniHeikkinen2013}
{\sc T.~Heikkinen, J.~Lehrb\"ack, J.~Nuutinen, and H.~Tuominen}, {\em
  Fractional maximal functions in metric measure spaces}, Analysis and Geometry
  in Metric Spaces, 1 (2013), pp.~147--162.

\bibitem{HoeksemaTh}
{\sc J.~Hoeksema}, {\em Mean-field limits and beyond}, PhD thesis, Eindhoven
  University of Technology, 2023.

\bibitem{HLS2025}
{\sc J.~Hoeksema, C.~Y. Lam, and A.~Schlichting}, {\em Variational convergence
  for an irreversible exchange-driven stochastic particle system}, arXiv
  preprint arXiv:2401.06696,  (2024).

\bibitem{HT2023}
{\sc J.~Hoeksema and O.~Tse}, {\em Generalized gradient structures for
  measure-valued population dynamics and their large-population limit},
  Calculus of Variations and Partial Differential Equations, 62 (2023), p.~158.

\bibitem{HHT2024}
{\sc A.~Hraivoronska, J.~Hoeksema, and O.~Tse}, {\em Large Population Limit of
  Interacting Population Dynamics via Generalized Gradient Structures},
  Springer Nature Switzerland, Cham, 2024, pp.~421--460.

\bibitem{Hraivoronska2024}
{\sc A.~Hraivoronska, A.~Schlichting, and O.~Tse}, {\em Variational convergence
  of the {S}charfetter-{G}ummel scheme to the aggregation-diffusion equation
  and vanishing diffusion limit}, Numerische Mathematik, 156 (2024),
  p.~2221–2292.

\bibitem{hraivoronska2023diffusive}
{\sc A.~Hraivoronska and O.~Tse}, {\em Diffusive limit of random walks on
  tessellations via generalized gradient flows}, SIAM Journal on Mathematical
  Analysis, 55 (2023), pp.~2948--2995.

\bibitem{JordanKinderlehrerOtto98}
{\sc R.~Jordan, D.~Kinderlehrer, and F.~Otto}, {\em The variational formulation
  of the {F}okker-{P}lanck {E}quation}, SIAM Journal on Mathematical Analysis,
  29 (1998), pp.~1--17.

\bibitem{kipnis2013scaling}
{\sc C.~Kipnis and C.~Landim}, {\em Scaling limits of interacting particle
  systems}, vol.~320, Springer Science \& Business Media, 2013.

\bibitem{Leo-unpubl}
{\sc C.~L{\'e}onard}, {\em Orlicz spaces}.
\newblock Unpublished lecture notes, 2007.

\bibitem{LieMiePeleRenger17}
{\sc M.~Liero, A.~Mielke, M.~A. Peletier, and D.~R.~M. Renger}, {\em On
  microscopic origins of generalized gradient structures}, Discrete Contin.
  Dyn. Syst. Ser. S, 10 (2017), pp.~1--35.

\bibitem{LieroMielkePeletierRenger17}
\leavevmode\vrule height 2pt depth -1.6pt width 23pt, {\em On microscopic
  origins of generalized gradient structures}, Discrete and Continuous
  Dynamical Systems-Series S, 10 (2017), p.~1.

\bibitem{Maas11}
{\sc J.~Maas}, {\em Gradient flows of the entropy for finite {M}arkov chains},
  Journal of Functional Analysis, 261 (2011), pp.~2250--2292.

\bibitem{Mielke13CALCVAR}
{\sc A.~Mielke}, {\em Geodesic convexity of the relative entropy in reversible
  {M}arkov chains}, Calc. Var. Partial Differential Equations, 48 (2013),
  pp.~1--31.

\bibitem{MielkePeletierRenger14}
{\sc A.~Mielke, M.~A. Peletier, and D.~R.~M. Renger}, {\em On the relation
  between gradient flows and the large-deviation principle, with applications
  to {M}arkov chains and diffusion}, Potential Analysis, 41 (2014),
  pp.~1293--1327.

\bibitem{MRS13}
{\sc A.~Mielke, R.~Rossi, and G.~Savar{\'e}}, {\em Balanced viscosity ({BV})
  solutions to infinite-dimensional rate-independent systems}, J. Eur. Math.
  Soc. (JEMS), 18 (2016), pp.~2107--2165.

\bibitem{PRST22}
{\sc M.~A. Peletier, R.~Rossi, G.~Savar\'{e}, and O.~Tse}, {\em Jump processes
  as generalized gradient flows}, Calc. Var. Partial Differential Equations, 61
  (2022), pp.~Paper No. 33, 85.

\bibitem{RaoRen}
{\sc M.~M. Rao and Z.~D. Ren}, {\em Theory of {O}rlicz spaces}, vol.~146 of
  Monographs and Textbooks in Pure and Applied Mathematics, Marcel Dekker,
  Inc., New York, 1991.

\bibitem{Reshetnyak68}
{\sc Y.~G. Reshetnyak}, {\em Weak convergence of completely additive vector
  functions on a set}, Siberian Math. J., 9 (1968), pp.~1039--1045.

\bibitem{Schilling2017measures}
{\sc R.~Schilling}, {\em Measures, Integrals and Martingales}, Cambridge
  University Press, 2017.

\bibitem{Schilling2012}
{\sc R.~L. Schilling and T.~Uemura}, {\em On the structure of the domain of a
  symmetric jump-type {D}irichlet form}, Publications of the Research Institute
  for Mathematical Sciences, 48 (2012), p.~1–20.

\bibitem{Schlichting2019}
{\sc A.~Schlichting}, {\em The exchange-driven growth model: Basic properties
  and longtime behavior}, Journal of Nonlinear Science, 30 (2019),
  p.~793–830.

\bibitem{Schmidt2018}
{\sc M.~Schmidt}, {\em A note on reflected {D}irichlet forms}, Potential
  Analysis, 52 (2018), p.~245–279.

\bibitem{SlepcevWarren23}
{\sc D.~Slep\v{c}ev and A.~Warren}, {\em Nonlocal {W}asserstein distance:
  metric and asymptotic properties}, Calc. Var. Partial Differential Equations,
  62 (2023), pp.~Paper No. 238, 66.

\bibitem{warren2025}
{\sc A.~Warren}, {\em Gradient flow structure for some nonlocal diffusion
  equations}, Preprint arXiv:1110.2794,  (2025).

\bibitem{Yeh}
{\sc J.~Yeh}, {\em Real analysis}, World Scientific Publishing Co. Pte. Ltd.,
  Hackensack, NJ, third~ed., 2014.
\newblock Theory of measure and integration.

\end{thebibliography}
}

\end{document}